\documentclass[10pt]{article}%

\usepackage[show]{ed}
\usepackage{tikz}
\usepackage{pgf}
\usetikzlibrary{arrows}
\usetikzlibrary{calc}
\usetikzlibrary{decorations.pathmorphing}

\usepackage{rotating} % needed for \begin{sidewaysfigure}\begin{sidewaystable}

\usepackage{mathrsfs}
\DeclareMathAlphabet{\mathpzc}{OT1}{pzc}{m}{it}

\usepackage{color}
\usepackage{amsmath}
\usepackage{graphicx}
\usepackage{amsfonts}
\usepackage{amssymb}%
\usepackage{latexsym}
\usepackage{psfrag}
\usepackage{accents}

\newtheorem{theorem}{Theorem}[section]
\newtheorem{corollary}[theorem]{Corollary}
\newtheorem{definition}[theorem]{Definition}
\newenvironment{proof}[1][Proof]{\noindent \emph{#1.} }
{\hfill \ \rule{0.5em}{0.5em}}
\newtheorem{lemma}[theorem]{Lemma}
\newtheorem{proposition}[theorem]{Proposition}

\newtheorem{example}[theorem]{Example}
\newtheorem{remark}[theorem]{Remark}

\numberwithin{equation}{section}
\numberwithin{table}{section}
\numberwithin{figure}{section}

\usepackage[a4paper]{geometry}
\newcommand{\cond}{\mathrm{cond}}

\newcommand{\R}{\mathbb{R}}

\newcommand{\cL}{{\cal L}}

\newcommand{\cS}{{\cal S}}

\newcommand{\cD}{{\cal D}}

\newcommand{\cM}{{\cal M}}

\newcommand{\cO}{{\cal O}}
\newcommand{\cZ}{{\cal Z}}
\newcommand{\bx}{x}%\mathbf{x}}

\newcommand{\bn}{n}%\mathbf{n}}

\newcommand{\bQ}{Q}%\mathbf{Q}}

\newcommand{\supp}{\mathrm{supp}}

\newcommand{\C}{\mathbb{C}}

\newcommand{\bZ}{Z}%\mathbf{Z}}

\newcommand{\re}{{\rm e}}
\newcommand{\ri}{{\rm i}}
\newcommand{\rd}{{\rm d}}

\newcommand{\beq}{\begin{equation}}
\newcommand{\eeq}{\end{equation}}
\newcommand{\beqs}{\begin{equation*}}
\newcommand{\eeqs}{\end{equation*}}
\newcommand{\bit}{\begin{itemize}}
\newcommand{\eit}{\end{itemize}}
\newcommand{\ben}{\begin{enumerate}}
\newcommand{\een}{\end{enumerate}}
\newcommand{\bal}{\begin{align}}
\newcommand{\eal}{\end{align}}
\newcommand{\bals}{\begin{align*}}
\newcommand{\eals}{\end{align*}}
\newcommand{\bse}{\begin{subequations}}
\newcommand{\ese}{\end{subequations}}
\newcommand{\bpr}{\begin{proposition}}
\newcommand{\epr}{\end{proposition}}
\newcommand{\bre}{\begin{remark}}
\newcommand{\ere}{\end{remark}}
\newcommand{\bpf}{\begin{proof}}
\newcommand{\epf}{\end{proof}}
\newcommand{\ble}{\begin{lemma}}
\newcommand{\ele}{\end{lemma}}
\newcommand{\bco}{\begin{corollary}}
\newcommand{\eco}{\end{corollary}}
\newcommand{\bex}{\begin{example}}
\newcommand{\eex}{\end{example}}
\newcommand{\bth}{\begin{theorem}}
\newcommand{\enth}{\end{theorem}}

\newcommand{\Rea}{\mathbb{R}}
\newcommand{\Com}{\mathbb{C}}

\newcommand{\Oi}{{\Omega_-}}

\newcommand{\Oe}{{\Omega_+}}

\newcommand{\GR}{{\Gamma_R}}
\newcommand{\OR}{{\Omega_R}}

\newcommand{\eps}{\varepsilon}

\newcommand{\pdiff}[2]{\frac{\partial #1}{\partial #2}}

\newcommand{\dnpu}{\partial_n^+ u}

\newcommand{\nus}{|u|^2}
\newcommand{\ngus}{|\nabla u|^2}

\newcommand{\gu}{\nabla u}

\newcommand{\nvs}{|v|^2}
\newcommand{\ngvs}{|\nabla v|^2}

\newcommand{\gv}{\nabla v}

\newcommand{\nT}{\nabla_{S}}

\newcommand{\half}{\frac{1}{2}}

\newcommand{\LtG}{{L^2(\Gamma)}}

\newcommand{\LtGt}{{\LtG\rightarrow \LtG}}

\newcommand{\HhG}{{H^{1/2}(\Gamma)}}
\newcommand{\HmhG}{{H^{-1/2}(\Gamma)}}

\newcommand{\HoG}{H^1(\Gamma)}

\newcommand{\tendi}{\rightarrow \infty}
\newcommand{\tendo}{\rightarrow 0}

\newcommand{\opA}{A'_{k,\eta}}
\newcommand{\opABW}{A_{k,\eta}}
\newcommand{\opAinv}{(A'_{k,\eta})^{-1}}
\newcommand{\opABWinv}{A^{-1}_{k,\eta}}

\newcommand{\normAinv}{\|\opAinv\|}

\def\XXint#1#2#3{{\setbox0=\hbox{$#1{#2#3}{\int}$}
     \vcenter{\hbox{$#2#3$}}\kern-.5\wd0}}

\usepackage{hyperref}
\definecolor{myblue}{rgb}{0,0,0.6}
\hypersetup{colorlinks=true,
linkcolor=myblue,citecolor=myblue,filecolor=myblue,urlcolor=myblue}
\newcommand*{\N}[1]{\left\|#1\right\|}

\allowdisplaybreaks[4]
\newcommand{\tfa}{\text{ for all }}

\newcommand{\tas}{\text{ as }}
\newcommand{\tand}{\text{ and }}

\newcommand{\vertiii}[1]{{\left\vert\kern-0.25ex\left\vert\kern-0.25ex\left\vert #1
    \right\vert\kern-0.25ex\right\vert\kern-0.25ex\right\vert}}

\newcommand{\GammaR}{\GR}%{\partial B_R}%\Gamma_R}

\newcommand{\DtN}{P^+_{DtN}}

\newcommand{\ItD}{P^{-,\eta}_{ItD}}

\newcommand{\opBM}{B_{k,\eta}}

\newcommand{\bound}{\Gamma}

\newcommand{\cutoff}{\chi_1 R(k)\chi_2}
\newcommand{\LtLt}{L^2(\Oe)\rightarrow L^2(\Oe)}

\usepackage{soul}

\definecolor{amcol}{rgb}{0.8,0,0}

\definecolor{escol}{rgb}{0,0,0.8}
\definecolor{estcol}{rgb}{0,0.5,0}
\definecolor{esnewcol}{rgb}{0,0.5,0}

\begin{document}

\title{High-frequency bounds for the Helmholtz equation under parabolic trapping and applications in numerical analysis
}

\author{S. N.~Chandler-Wilde\footnotemark[1]\,\,, E. A. Spence\footnotemark[2]\,,\, A.~Gibbs\footnotemark[3]$^{,\,}$\footnotemark[4]
\,, and V. P. Smyshlyaev\footnotemark[4]}
\date{}

\date{\today}

\renewcommand{\thefootnote}{\fnsymbol{footnote}}

\footnotetext[1]{Department of Mathematics and Statistics, University of Reading,
Whiteknights, PO Box 220, Reading, RG6 6AX, UK, \tt S.N.Chandler-Wilde@reading.ac.uk}
\footnotetext[2]{Department of Mathematical Sciences, University of Bath, Bath, BA2 7AY, UK, \tt E.A.Spence@bath.ac.uk }
\footnotetext[3]{Department of Computer Science, Celestijnenlaan 200 A box 2402, 3001 Leuven, Belgium}
\footnotetext[4]{Department of Mathematics, University College London, Gower Street, London, WC1E 6BT, UK, \tt Andrew.Gibbs@ucl.ac.uk, V.Smyshlyaev@ucl.ac.uk}

\maketitle

\begin{abstract}
This paper is concerned with resolvent estimates on the real axis for the Helmholtz equation posed in the exterior of a bounded obstacle with Dirichlet boundary conditions when the obstacle is \emph{trapping}.
There are two resolvent estimates for this situation currently in the literature: (i) in the case of {\em elliptic trapping} the general ``worst case'' bound of exponential growth
applies, and examples show that this growth can be realised through some sequence of wavenumbers;
(ii) in the prototypical case of {\em hyperbolic trapping} where the Helmholtz equation is posed in the exterior of two strictly convex obstacles (or several obstacles with additional constraints) the nontrapping resolvent estimate holds with a logarithmic loss.

This paper proves the first resolvent estimate for {\em parabolic trapping} by obstacles, studying a class of obstacles the prototypical example of which is the exterior of two squares (in 2-d), or two cubes (in 3-d), whose sides are parallel. We show, via developments of the vector-field/multiplier argument of Morawetz and the first application of this methodology to trapping configurations, that a resolvent estimate holds with a polynomial loss over the nontrapping estimate. We use this bound, along with the other trapping resolvent estimates, to prove results about integral-equation formulations of the boundary value problem in the case of trapping.
Feeding these bounds into existing frameworks for analysing finite and boundary element methods, we obtain
the first wavenumber-explicit proofs of convergence for numerical methods for solving the Helmholtz equation in the exterior of a trapping obstacle.

\vspace{-1.5ex}


\vspace{-1.5ex}


\end{abstract}

\renewcommand{\thefootnote}{\arabic{footnote}}

\section{Introduction}\label{sec:1}

\subsection{Context, and informal discussion of the main results}\label{sec:4star}
Trapping and nontrapping are central concepts in scattering theory. In the case of the Helmholtz equation, $\Delta u + k^2 u=-f$, posed in the exterior of a bounded, Dirichlet obstacle $\Oi$ in 2- or 3-dimensions, $\Oi$ is \emph{nontrapping}
if all billiard trajectories starting in an exterior neighbourhood of $\Oi$ escape from that neighbourhood after some uniform time, and $\Oi$ is \emph{trapping} otherwise (see Definitions \ref{def:nt1} and \ref{def:nt2} below for more precise statements, taking into account subtleties about diffraction from corners).

This paper is concerned with resolvent estimates (i.e.~a priori bounds on the solution $u$ in terms of the data $f$) for the exterior Dirichlet problem when $k$ is real. We can write these in terms of the outgoing cut-off resolvent $\chi_1 R(k)\chi_2 : L^2(\Oe) \rightarrow L^2(\Oe)$ for $k\in \Rea\setminus\{0\}$,
where $\Oe:= \Rea^d \setminus \overline{\Oi}$, $\chi_1, \chi_2 \in C^\infty_{\rm comp}(\overline{\Oe})$
and $R(k):= (\Delta +k^2)^{-1}$, with Dirichlet boundary conditions, is such that $R(k) : L^2(\Oe)\rightarrow L^2(\Oe)$ for $\Im k>0$.
When $\Oi$ is nontrapping, given $k_0>0$,
\beq\label{eq:nt_estimate}
\N{\cutoff}_{\LtLt}\lesssim \frac{1}{k} \quad\tfa k\geq k_0;
\eeq
this classic result was first obtained by the combination of the results on propagation of singularities for the wave equation on manifolds with boundary by
Melrose and Sj\"ostrand \cite{MeSj:78, MeSj:82} with either the parametrix method of Vainberg \cite{Va:75} (see \cite{Ra:79}) or the methods of Lax and Phillips \cite{LaPh:89} (see \cite{Me:79}), following the proof
by Morawetz, Ralston, and Strauss  \cite{Mo:75,MoRaSt:77} of the bound under a slightly-stronger condition than nontrapping.

In this situation of scattering by a (Dirichlet) obstacle, there are two resolvent estimates in the literature when $\Oi$ is trapping. The first is
the general result of Burq \cite[Theorem 2]{Bu:98} that, given any smooth $\Oi$ and $k_0>0$, there exists $\alpha>0$ such that
\beq\label{eq:Burq}
\N{\cutoff}_{\LtLt}\lesssim \re^{\alpha k} \quad\tfa k\geq k_0.
\eeq
If $\Oi$ has an ellipse-shaped cavity (see Figure \ref{fig:examples}(a)) then there exists a sequence of wavenumbers $0<k_1<k_2<\ldots$, with $k_j\tendi$, and $\alpha>0$ such that
\beq\label{eq:ellipse}
\N{\chi_1 R(k_j)\chi_2}_{\LtLt}\gtrsim \re^{\alpha k_j} \quad j=1,2,\ldots,
\eeq
see, e.g., \cite[\S2.5]{BeChGrLaLi:11}, and thus the bound \eqref{eq:Burq} is sharp.  More generally, if  there exists an elliptic trapped ray (i.e.~an elliptic closed broken geodesic),
and $\partial \Oi$ is analytic in neighbourhoods of the vertices of the broken geodesic, then the resolvent can grow at least as fast as $\exp{(\alpha k_j^q)}$, through a sequence $k_j$ as above and for some range of $q\in(0,1)$, by the quasimode construction of Cardoso and Popov \cite{CaPo:02} (note that Popov proved \emph{superalgebraic} growth for certain elliptic trapped rays when $\partial \Oi$ is smooth in \cite{Po:91}).

\begin{figure}[h]
\centering
\begin{tikzpicture}[line cap=round,line join=round,>=triangle 45,x=1.0cm,y=1.0cm, scale=1.65]
\colorlet{lightgray}{black!15}
%\draw[step=1cm,gray,very thin] (-1,-1) grid (6,4);

\fill[color=lightgray](3.35,0.8) -- (4.25,0.8) -- (4.25,1.8) -- (3.35,1.8)--cycle;
\draw (3.35,0.8) -- (4.25,0.8) -- (4.25,1.8) -- (3.35,1.8)--cycle;
\fill[color=lightgray](4.75,0.8) -- (5.65,0.8) -- (5.65,1.8) -- (4.75,1.8)--cycle;
\draw (4.75,0.8) -- (5.65,0.8) -- (5.65,1.8) -- (4.75,1.8)--cycle;
\draw (3.02,1.8) node {(c)};
\draw[thick,dashed] (4.25,1.3) -- (4.75,1.3);

\fill[color=lightgray] (0.6,1.3) circle (0.5);
\draw (0.6,1.3) circle (0.5);
\fill[color=lightgray] (2,1.3) circle (0.5);
\draw (2,1.3) circle (0.5);
\draw (-0.2,1.8) node {(b)};
\draw[thick,dashed] (1.1,1.3) -- (1.5,1.3);

\coordinate (P) at ($(-1.75,1.3)+(130:0.25 and 0.5)$);
\coordinate (Q) at ($(-2.65,1.3)+(130:0.25 and 0.5)$);
\coordinate (R) at ($(-2.65,1.3)+(230:0.25 and 0.5)$);
\fill[color=lightgray] (P) arc (130:230:0.25 and 0.5) -- (R) -- (Q) -- (P);
\draw (P) arc (130:230:0.25 and 0.5) -- (R) -- (Q) -- (P);

\coordinate (A) at ($(-1.75,1.3)+(310:0.25 and 0.5)$);
\coordinate (B) at ($(-0.85,1.3)+(50:0.25 and 0.5)$);
\coordinate (C) at ($(-0.85,1.3)+(310:0.25 and 0.5)$);
\fill[color=lightgray] (A) arc (-50:50:0.25 and 0.5) -- (B) -- (C) -- (A);
\draw (A) arc (-50:50:0.25 and 0.5) -- (B) -- (C) -- (A);

\draw[red,dashed] (-1.75,1.3) ellipse (0.25 and 0.5);

\draw (-3.05,1.8) node {(a)};
\draw[thick,dashed] (-2,1.3) -- (-1.5,1.3);

\end{tikzpicture}
\caption{Examples of: (a) elliptic trapping; (b) hyperbolic trapping; (c) parabolic trapping.}
\label{fig:examples}
\end{figure}
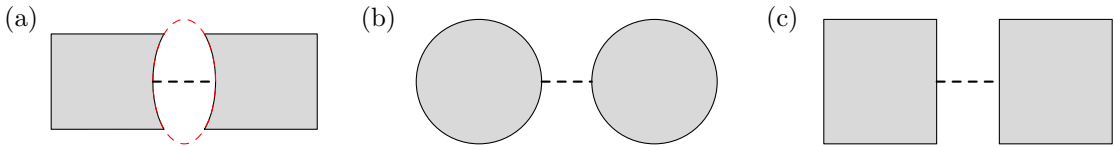

The second trapping resolvent estimate in the literature concerns hyperbolic trapping, the standard example of which is when $\Oi$ equals two disjoint convex obstacles with strictly positive curvature; see Figure \ref{fig:examples}(b).
The work of Ikawa on this problem (and its generalisation to a finite number of such obstacles satisfying additional conditions -- see Definition \ref{def:IB} below) implies that there exists $N>0$ such that
\beq\label{eq:Ikawa}
\N{\cutoff}_{\LtLt}\lesssim k^N \quad\tfa k\geq k_0
\eeq
\cite[Theorem 2.1]{Ik:88}, \cite[Theorem 4.5]{Bu:04}, and this bound was later improved by Burq \cite[Proposition 4.4]{Bu:04} to
\beq\label{eq:Ikawa2}
\N{\cutoff}_{\LtLt}\lesssim \frac{\log (2+k)}{k} \quad\tfa k\geq k_0,
\eeq
i.e.~the trapping is so weak there is only a logarithmic loss over the nontrapping estimate \eqref{eq:nt_estimate}.

\paragraph{Summary of the main results and their novelty.}
This paper considers the exterior Dirichlet problem for a certain class of parabolic-trapping obstacles,
and the heart of this paper and its main result is the following theorem, which is subsumed into the more-general Theorem \ref{thm:resol} below.

\begin{theorem}\label{thm:newintro}
For the class of obstacles in Definition \ref{def:trapb} below,
 the simplest example of which is two squares (in 2-d) or two cubes (in 3-d) with their sides parallel (see Figure \ref{fig:examples}(c)), given $k_0>0$,
\beq\label{eq:us}
\N{\cutoff}_{\LtLt}\lesssim k \quad\tfa k\geq k_0.
\eeq
\end{theorem}

We believe that \eqref{eq:us} is the first resolvent estimate proved for parabolic trapping by obstacles.
A simple construction involving the eigenfunctions of the Dirichlet Laplacian on an interval gives an example of a compactly-supported $f$ such that $\N{\chi_1 R(k) f}_{L^2(\Oe)}\gtrsim \N{f}_{L^2(\Oe)}$ (see \cite[End of \S3]{ChMo:08}),
so that \eqref{eq:us} is at most one power of $k$ away from being sharp.
Furthermore, we prove that if \emph{either} $\supp \,\chi_1$ \emph{or} $\supp \,\chi_2$ is sufficiently far away from the ``trapping region" (this is defined more precisely below, but in the example of two squares/cubes one can think of it as the region between the two obstacles), then $\N{\cutoff}_{L^2\rightarrow L^2}\lesssim 1$, and if \emph{both} $\supp\, \chi_1$ and $\supp \,\chi_2$ are sufficiently far away from the trapping region, then the nontrapping estimate $\N{\cutoff}_{L^2\rightarrow L^2}\lesssim 1/k$ holds.

We prove these resolvent estimates by adapting and developing the vector-field/multiplier argument of Morawetz; this argument famously proves the estimate \eqref{eq:nt_estimate} for the Dirichlet resolvent for star-shaped domains \cite{Mo:75, MoLu:68} (see also \cite{ChMo:08}) using the vector field $\bx$, and (in $d=2$) for a class of domains slightly more restrictive than nontrapping \cite{Mo:75}, \cite[\S4]{MoRaSt:77}.
The present paper represents the first application of this methodology to trapping by a bounded obstacle.
Our argument is based on using the
vector field $e_dx_d$ (with $e_d$ the unit vector in the $x_d$ direction) in the trapping region and the vector field $x$ in the far-field;
see Figure \ref{fig:cavity_vector_field} below for an example obstacle along with the corresponding vector field.
The main technical challenge is achieving a transition between these vector fields (and other coefficients in the multiplier) in a controllable way, and a main source of difficulty in accomplishing this is that the derivative matrix of $e_dx_d$ is only semidefinite, in contrast to related transitioning arguments applied in nontrapping configurations where the derivative matrices of the vector fields are positive definite (e.g. \cite[Lemma 2, Proof of Lemma 5]{Mo:75}); a more-detailed outline of the ideas behind the proof is given in \S\ref{sec:idea}.
We note that the vector field $e_dx_d$ (on its own) has been used by the first author and Monk \cite{ChMo:05} to prove a priori bounds on solutions of scattering by unbounded rough surfaces, and also by Burq, Hassell, and Wunsch \cite{BuHaWu:07} to study spreading of quasimodes in the Bunimovich stadium.

An advantage of these vector-field arguments in this obstacle setting is that they avoid the substantial technicalities involved with propagation of singularities on manifolds with boundary. Indeed, the only other results in the literature that deal with parabolic and/or degenerate hyperbolic configurations, these proved with propagation-of-singularities methods, are the results of Christianson and Wunsch \cite{ChWu:13}, \cite{Ch:13} in the setting of scattering by metrics (where there is no boundary); see the discussion in \S\ref{sec:discuss}. Moreover, using these propagation-of-singularities techniques to prove resolvent estimates for scattering by non-smooth obstacles is highly nontrivial; the only result for non-smooth obstacles obtained with these methods is that of Baskin and Wunsch \cite{BaWu:13},
that the nontrapping resolvent estimate \eqref{eq:nt_estimate} holds in 2-d for \emph{nontrapping polygons} (in the sense of Definition \ref{def:nt2} below).

Additionally, our vector-field arguments lead naturally to the improvements described above in the $k$-dependence of \eqref{eq:us} if either $\supp \chi_1$ or $\supp \chi_2$ is sufficiently far away from the trapping region. In the case of scattering by smooth obstacles such improvements have been
established by propagation-of-singularities arguments, but only when $\supp \chi_1\equiv \supp\chi_2$ and both are sufficiently far away from the obstacle; see Burq \cite[Theorem 4]{Bu:02a} and Cardoso and Vodev \cite[Theorem 1.1]{CaVo:02}. Related results where the cut-off functions are replaced by semiclassical pseudodifferential operators restricting attention to areas of phase space isolated from the trapped set
have been proved in the setting of scattering by a potential and/or by a metric (but not an obstacle) by Datchev and Vasy \cite[Theorems 1.1, 1.2]{DaVa:13}.

One further advantage of these vector-field arguments is that, for $k\geq k_0$ for some explicitly given $k_0$, they enable us to obtain an expression for the omitted constant in \eqref{eq:us} that is explicit in all parameters (in particular,  parameters describing the geometries of the domain, and the choice of the cut-off functions; see Lemma \ref{lem:E5new} below); thus our resolvent estimates are ``quantitative'' in the sense of, e.g., Rodnianski and Tao \cite{RoTa:14}.

The resolvent estimate \eqref{eq:us} has immediate implications for boundary-integral-equation formulations of the scattering problem, for the numerical analysis of these integral-equation formulations, and also for the numerical analysis of the finite element method (based on the standard domain-based variational formulation of the scattering problem); these implications are outlined in \S\ref{sec:FEM} and \S\ref{sec:BEM} below. In this sense, this paper follows the theme of
\cite{ChMo:08}, \cite{Sp:14}, and \cite{BaSpWu:16} of proving high-frequency estimates for the Helmholtz equation and then exploring their implications for integral equations and numerical analysis. Novelties of the present paper with respect to \cite{ChMo:08}, \cite{Sp:14}, and \cite{BaSpWu:16} include that:
\ben
\item We show how to write down the passage from resolvent estimate, to bound on the Dirichlet-to-Neumann map, to bounds on integral operators, explicitly as a general black-box ``recipe", and use this recipe -- applied implicitly to $C^\infty$ \emph{nontrapping} scenarios in \cite{BaSpWu:16} -- to deduce the first bounds for \emph{trapping} scenarios.
As a consequence, this paper  includes the first wavenumber-explicit proofs of convergence for a numerical method for solving the Helmholtz equation in a trapping domain (see \S\ref{sec:FEM} and \S\ref{sec:BEM}).
\item Whereas \cite{ChMo:08}, \cite{Sp:14}, and \cite{BaSpWu:16}  proved bounds on integral operators posed only on the space $\LtG$, where $\Gamma:=\partial \Omega_-$, we prove wavenumber-explicit bounds for the Sobolev spaces $H^s(\Gamma)$ and $H^s_k(\Gamma)$ for $-1\leq s\leq 1$ (defined in \S\ref{sec:interp}). One motivation for this is that, just as there is large interest in the $\LtG$-theory of these integral operators, there is also a large interest in the theory in the ``energy spaces" $H^{\pm 1/2}(\Gamma)$ and $H^{\pm 1/2}_k(\Gamma)$ (see, e.g., \cite[Chapter 7]{Mc:00}, \cite[Chapter 3]{SaSc:11}, \cite[Chapter 6]{St:08}).
\item To complement the \emph{upper bound} on the integral operator under parabolic trapping proved in Corollary \ref{cor:CFIE}, we prove a new \emph{lower bound} in this scenario in Lemma \ref{lem:smysh}. The arguments used in the proof of the lower bound additionally lead to
a counterexample to a conjecture on $k$-uniform coercivity of integral operators made in \cite[Conjecture 6.1]{BeSp:11}; see  \S\ref{rem:coercivity} below.
\een

\subsection{Statement of the main resolvent-estimate and DtN-map results}

\subsubsection{Geometric definitions} \label{sec:geometric}

Let $\Omega_-\subset \R^d$, $d=2,3$, be a bounded Lipschitz open set
such that the open complement
$\Omega_+:= \Rea^d \setminus \overline{\Omega_-}$ is connected, and let $\Gamma:= \partial \Omega_+ = \partial \Omega_-$ and $R_\Gamma:= \max_{x\in \Gamma}|x|$. Let $\gamma_\pm$ denote the trace operators from $\Omega_\pm$ to $\Gamma$, let $\partial_n^\pm$ denote the normal derivative trace operators (the normal pointing out of $\Omega_-$ and into $\Omega_+$), and let $\nabla_S$ denote the surface gradient operator on $\Gamma$.
Let $H^1_{\mathrm{loc}}(\Omega_+)$ denote the set of functions, $v$, such that $v$ is locally integrable on $\Omega_+$
and $\chi v \in H^1(\Omega_+)$ for every $\chi \in C_{\mathrm{comp}}^\infty(\overline{\Omega_+}):=\{ \chi|_{\Omega_+}  : \,\chi \in C^\infty(\Rea^d) \mbox{ is compactly supported}\}$.  We abbreviate $r:= |x|$, and $x_j$ and $n_j(x)$ denote the $j$th components of $x$ and $n(x)$, respectively, so that $n_j(x) = e_j\cdot n(x)$, where $e_j$ is the unit vector in the $x_j$ direction. Let $B_R(x):= \{y\in \R^d:|x-y|<R\}$ and $B_R:= B_R(0)$. Finally, let $\Omega_R:= \Oe\cap B_R$.

In discussing resolvent estimates, the following geometric definitions play a central role.

\begin{definition}[Star-shaped, and star-shaped with respect to a ball]
We say that a bounded open set $\Omega$ is:
\

(i) \emph{star-shaped with respect to the point $\bx_0\in \Omega$}  if, whenever $\bx \in \Omega$, the segment $[\bx_0,\bx]\subset \Omega$;

(ii) \emph{star-shaped} if there exists an $\bx_0\in \Omega$ such that $\Omega$ is star-shaped with respect to $\bx_0$;

(iii) \emph{star-shaped with respect to the ball $B_{a}(\bx_0)$} if it is star-shaped with respect to every point in $B_{a}(\bx_0)$;

(iv) \emph{star-shaped with respect to a ball} if there exists $a>0$ and $\bx_0\in\Omega$ such that $\Omega$ is star-shaped with respect to the ball $B_{a}(\bx_0)$.
\end{definition}

Recall that if $\Oi$ is Lipschitz, then it is star-shaped with respect to $\bx_0$ if and only if $(\bx-\bx_0)\cdot n(\bx)\geq 0$ for all $\bx \in\Gamma$ for which $n(\bx)$ is defined,
and $\Oi$ is star-shaped with respect to $B_{a}(\bx_0)$ if and only if
$(\bx-\bx_0) \cdot n(\bx) \geq {a}$ for all  $\bx \in \Gamma$ for which $n(\bx)$ is defined; see, e.g., \cite[Lemma 5.4.1]{Mo:11}.

\begin{definition}[Nontrapping]\label{def:nt1}
We say that
 $\Oi\subset \Rea^d, \,d=2, 3$, is
\emph{nontrapping} if $\bound$ is smooth ($C^\infty$) and,
given $R$ such that $\overline{\Oi}\subset B_R$, there exists a $T(R)<\infty$ such that
all the billiard trajectories (in the sense of Melrose--Sj{\"o}strand~\cite[Definition 7.20]{MeSj:82})
that start in $\Omega_R$
 at time zero leave $\Omega_R$
 by time $T(R)$.
\end{definition}

We now introduce the classes of Lipschitz obstacles  to which  our new resolvent estimates apply (Definitions \ref{def:trapb} and \ref{def:trapc}).
The most general class is the class of {\em $(R_0,R_1)$ obstacles} (Definition \ref{def:trapb}). The definition of this class is somewhat implicit, in terms of existence of an appropriate vector field $Z$ that we use when proving the resolvent estimate \eqref{eq:us}. But it  follows from the definition and Remark \ref{rem:chi} that, expressed in terms of the geometry of the obstacle, membership of this class is nothing more than a requirement that, for some concentric circles centred on the origin of radii $R_0$ and $R_1$ with $R_1/R_0>\re^{1/4}$, it holds that $x_dn_d(x)\geq 0$ for almost all $x\in \Gamma$ inside the smaller circle, that $x\cdot n(x) \geq 0$ for almost all $x\in \Gamma$ outside the larger circle, and that some particular convex combination of $x_dn_d(x)$ and $x\cdot n(x)$ is non-negative for almost all $x\in \Gamma$ in the transition zone between the two circles.

\begin{definition}[$(R_0,R_1)$ obstacle] \label{def:trapb}  For $0<R_0<R_1$ we say that $\Omega_-$ is an {\em $(R_0,R_1)$ obstacle} if there exists $\chi\in C^3[0,\infty)$ with
\begin{enumerate}
\item[(i)] $\chi(r)=0$ for $0\leq r\leq R_0$, $\chi(r) = 1$, for $r\geq R_1$, $0<\chi(r)<1$, for $R_0<r<R_1$; and
\item[(ii)] $0\leq r\chi^\prime(r) < 4$, for $r>0$;
\end{enumerate}
such that  $Z(x)\cdot n(x) \geq 0$ for all $x\in \Gamma$ for which the normal $n(x)$ is defined, where
\begin{equation} \label{eq:Zdeff}
Z(x) := e_dx_d\big(1-\chi(r)\big) + x \chi(r), \quad x\in \R^d.
\end{equation}
\end{definition}

\begin{remark}[Constraint on $R_1/R_0$] \label{rem:chi}
If $\Oi$ is an $(R_0,R_1)$ obstacle then $R_1/R_0>\re^{1/4}\approx 1.284$. For, if $\chi\in C^3[0,\infty)$ satisfies (i) and (ii), then
$$
1 = \int_{R_0}^{R_1} \chi^\prime(r) \, \rd r < \int_{R_0}^{R_1} \frac{4}{r}\, \rd r = 4 \log(R_1/R_0).
$$
Conversely (see the proof of Lemma \ref{lem:strongly} below), if $R_1>\re^{1/4}R_0$, then $\chi\in C^3[0,\infty)$ can be constructed satisfying the constraints (i) and (ii) of the above definition.
\end{remark}

 An important sub-class of $(R_0,R_1)$ obstacles is the class of {\em strongly $(R_0,R_1)$ obstacles} (Definition \ref{def:trapc} and see Figures \ref{fig:examples}(c), \ref{fig:twosquares} and  \ref{fig:cavity}). The difference between these definitions is precisely that we require that {\em both} $x_dn_d(x)$ and $x\cdot n(x)$ be non-negative for almost all $x\in \Gamma$ in the transition zone between the two circles for an obstacle to be {\em strongly} $(R_0,R_1)$.

\begin{definition}[Strongly $(R_0,R_1)$ obstacle] \label{def:trapc}  For $R_1>\re^{1/4}R_0>0$ we say that $\Omega_-$ is a {\em strongly $(R_0,R_1)$ obstacle} if, for all $x\in \Gamma$ for which $n(x)$ is defined, $x_dn_d(x)\geq 0$ if $|x|\leq R_1$, while $x\cdot n(x) \geq 0$ if $|x|\geq R_0$.
\end{definition}

\begin{lemma}[A strongly $(R_0,R_1)$ obstacle is an $(R_0,R_1)$ obstacle] \label{lem:strongly} If $\Oi$ is a strongly $(R_0,R_1)$ obstacle, then $\Oi$ is an $(R_0,R_1)$ obstacle.
\end{lemma}
\begin{proof}
 To show that a strongly $(R_0,R_1)$ obstacle $\Oi$ is an $(R_0,R_1)$ obstacle we just need to construct a $\chi\in C^3[0,\infty)$ satisfying the constraints (i) and (ii) of Definition \ref{def:trapb}. For if we do that and define $Z$ by \eqref{eq:Zdeff}, then $Z(x)\cdot n(x) = x_dn_d(x) (1-\chi(r))+x\cdot n(x) \chi(r) \geq 0$. But,
given any $0<\epsilon<(R_1-R_0)/2$, we can construct a  $p \in C^2(\R)$ such that $p(r)=0$, for $r\leq R_0$ and $r\geq R_1$, $0<p(r)\leq 1$, for $R_0<r<R_1$, and $p(r)=1$ for $R_0+\epsilon \leq r\leq R_1-\epsilon$. Then, if $R_1/R_0> \re^{1/4}$, the function
$$
\chi(r) := \int_{R_0}^r\frac{p(s)}{s}\, \rd s\bigg/\int_{R_0}^{R_1}\frac{p(s)}{s}\, \rd s, \quad r\geq 0,
$$
is in $C^3[0,\infty)$ and satisfies the constraints (i) and (ii) provided
\begin{equation} \label{eq:epsconstraint}
\epsilon \leq \epsilon_0 := \frac{R_1-R_0\re^{1/4}}{\re^{1/4}+1}\, .
\end{equation}
In particular, for $r>0$,
$$
0\leq r\chi^\prime(r) = \frac{p(r)}{\int_{R_0}^{R_1}(p(s)/s)\, \rd s} < \frac{1}{\int_{R_0+\epsilon}^{R_1-\epsilon}s^{-1}\, \rd s} = \frac{1}{\log((R_1-\epsilon)/(R_0+\epsilon)},
$$
and this last expression is $\leq 4$ if and only if \eqref{eq:epsconstraint} holds.
\end{proof}

\begin{figure}[h]
\centering
\begin{tikzpicture}[line cap=round,line join=round,>=triangle 45,x=1.0cm,y=1.0cm, scale=1.85]
\colorlet{lightgray}{black!15}
%\draw[step=1cm,gray,very thin] (-1,-1) grid (6,4);
\fill[color=lightgray](0,0) -- (1,0) -- (1,1) -- (0,1)--(0,0);
\draw (0,0) -- (1,0) -- (1,1) -- (0,1)--(0,0);
\fill[color=lightgray](1.5,-0.5) -- (2.5,-0.5) -- (2.5,0.5) -- (1.5,0.5)--(1.5,-0.5);
\draw (1.5,-0.5) -- (2.5,-0.5) -- (2.5,0.5) -- (1.5,0.5)--(1.5,-0.5);
\draw [blue,thick,<->] (1,0.3) -- (1.5,0.3);
\draw (1.25,0.3) node[anchor=south] {$a$};
\draw (1.8,0.6) node {$\Gamma$};
\draw [->] (0.6,1) -- (0.6,1.5);
\draw (0.6,1.35) node[anchor = west] {$n$};
\draw [->] (0,0) -- (3,0);
\draw [->] (0,0) -- (0,1.7);
%	\draw [line width=1.2pt,dash pattern=on 2pt off 2pt] (6,3)-- (8.5,3);
%	\draw [line width=1.2pt,dash pattern=on 2pt off 2pt] (2,0)-- (2,-3.5);
\draw (2.8,-0.05) node[anchor=north west] {$x_1$};
\draw (0,1.4) node[anchor=south west] {$x_2$};

\end{tikzpicture}
\caption{The obstacle $\Omega_-$ is the union of two parallel squares. For $R_1>\re^{1/4}R_0\geq R_\Gamma:= \max_{x\in \Gamma}|x|$, it is a 2-d example both of a strongly $(R_0,R_1)$ obstacle and of an $(R_0,R_1,a)$ parallel trapping obstacle,  since $x_2n_2(x)\geq 0$ for all $x\in \Gamma$ for which $n(x)$ is defined.}
\label{fig:twosquares}
\end{figure}
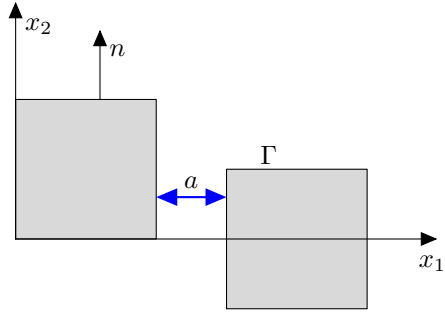

\begin{figure}[h]
\centering
\begin{tikzpicture}[line cap=round,line join=round,>=triangle 45,x=1.0cm,y=1.0cm, scale=1.2]
\colorlet{lightgray}{black!15}
%\draw[step=1cm,gray,very thin] (-1,-1) grid (6,4);
\fill[color=lightgray](0,0) -- (0,1) -- (-5,1) -- (-6,-1.6)--(-4,-2.5) -- (2.4,-1.7) -- (2.4,1) -- (1,1) -- (1,0) -- (0,0);
\draw (0,0) -- (0,1) -- (-5,1) -- (-6,-1.6)--(-4,-2.5) -- (2.4,-1.7) -- (2.4,1) -- (1,1) -- (1,0) -- (0,0);
\draw[red,thick,dashed] (0,0) circle (1.6);
\draw [red,->] (0,0) -- (160:1.6);
\draw (160:1) node[anchor=south] {$R_0$};
\draw[red,thick,dashed] (0,0) circle (3);
\draw [blue,thick,<->] (0,0.5) -- (1,0.5);
\draw (0.5,0.5) node[anchor=north] {$a$};
\draw [red,->] (0,0) -- (-45:3);
\draw (-45:2.1) node[anchor=south west] {$R_1$};
\draw (1.8,1.2) node {$\Gamma$};
\draw [->] (0,0) -- (2.8,0);
\draw [->] (0,0) -- (0,2);
\draw [->] (-4,1) -- (-4,1.5);
\draw (-4,1.35) node[anchor = west] {$n$};
%	\draw [line width=1.2pt,dash pattern=on 2pt off 2pt] (6,3)-- (8.5,3);
%	\draw [line width=1.2pt,dash pattern=on 2pt off 2pt] (2,0)-- (2,-3.5);
\draw (2.55,-0.05) node[anchor=north west] {$x_1$};
\draw (0,1.7) node[anchor=south west] {$x_2$};

\end{tikzpicture}
\caption{The grey-shaded obstacle $\Omega_-$ is a 2-d example both of a strongly $(R_0,R_1)$ obstacle and of an $(R_0,R_1,a)$ parallel trapping obstacle, with the values of $R_0$, $R_1$, and $a$ indicated, since $R_1> \re^{1/4}R_0$ and, for $x\in \Gamma$, $x_2n_2(x)\geq 0$ for $|x|\leq R_1$ and $x\cdot n(x)\geq 0$ for $|x|\geq R_0$.}
 \label{fig:cavity}
\end{figure}
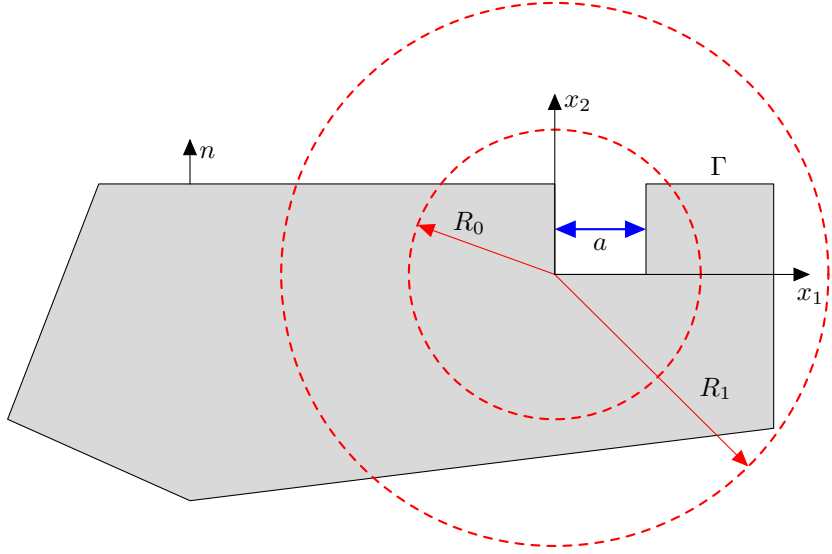

\begin{figure}[h]
\begin{center}
\includegraphics[width=12cm]{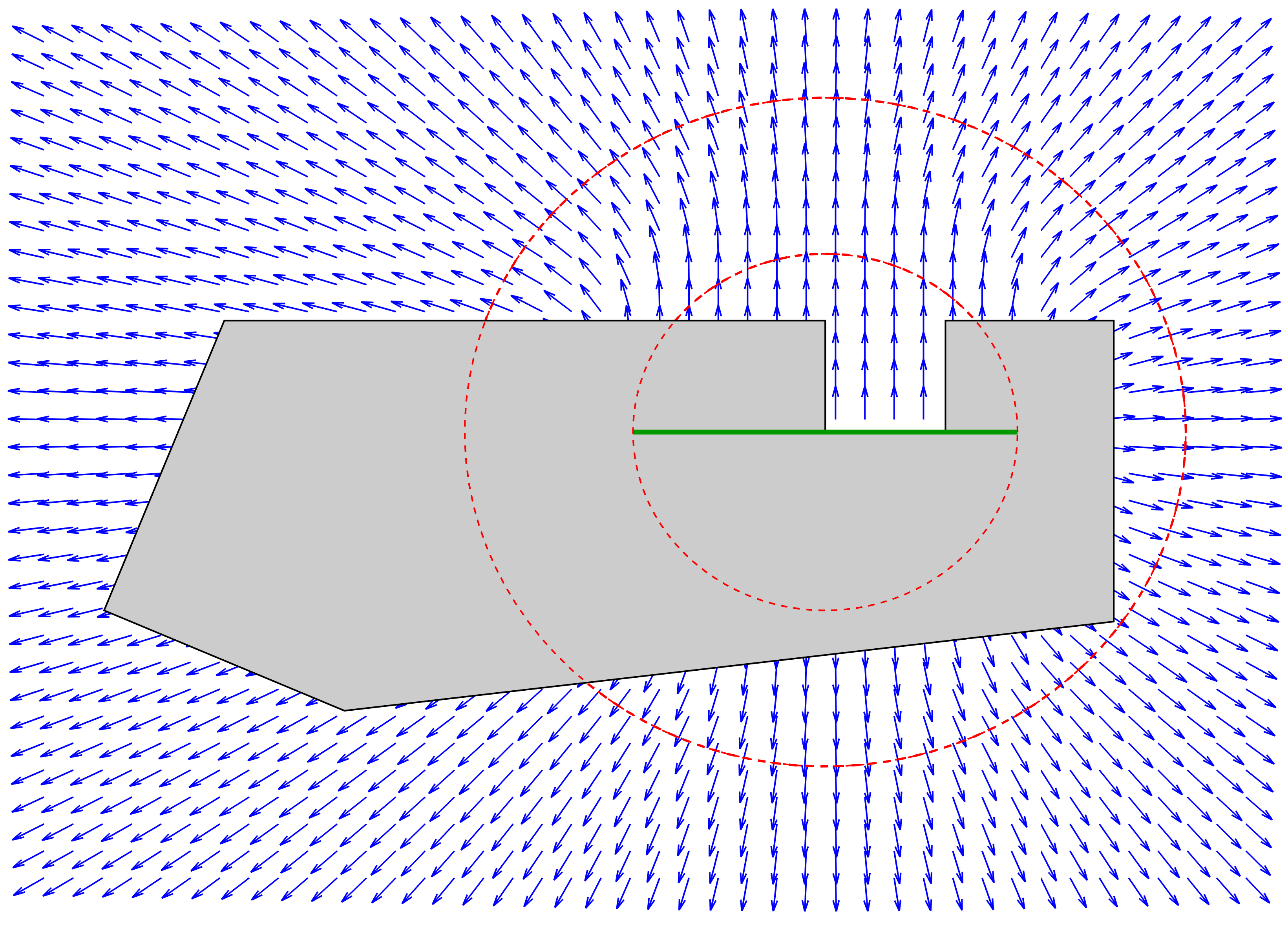}
 \end{center}
\caption{The obstacle $\Omega_-$ of Figure \ref{fig:cavity} and the same concentric circles centred on the origin of radii $R_1>R_0$, but with arrows showing the direction of the vector field $Z$, where $Z$ is defined by \eqref{eq:Zdeff} with $\chi$ constructed as in the proof of Lemma \ref{lem:strongly} (with $\epsilon:= \epsilon_0$, where $\epsilon_0$ is given by \eqref{eq:epsconstraint}). Note that $Z(x)=x$ outside the larger circle, that $Z(x)=x_2e_2$ inside the smaller circle (so that $Z$ vanishes and its direction is undefined on the thick green line), and that $Z(x)\cdot n(x)\geq 0$ for all $x\in \Gamma$ for which $n(x)$ is defined.}
\label{fig:cavity_vector_field}
\end{figure}

\begin{remark}[Examples of strongly $(R_0,R_1)$ obstacles] \label{ex:trap} It is clear that $\Omega_-$ is a strongly $(R_0,R_1)$ obstacle (and so an $(R_0,R_1)$ obstacle) if $R_1> \re^{1/4}R_0>0$ and either of the following conditions holds.

 (i) $R_0 \geq R_\Gamma :=\max_{x\in \Gamma}|x|$, and $x_dn_d(x)\geq 0$  for all $x\in \Gamma$ for which $n(x)$ is defined (e.g.\  $\Omega_-$ is the union of two or more balls with centres in the plane $x_d=0$, or the union of two or more parallel squares, see Figure \ref{fig:twosquares}).

(ii) $\min_{x\in \Gamma}|x|\geq R_1$ and $\Oi$ is star-shaped with respect to the origin.
\end{remark}

The second example shows that an $(R_0,R_1)$ obstacle need not be trapping, and so it is convenient to define a class of $(R_0,R_1)$ obstacles that \emph{are} trapping.

\begin{definition}[$(R_0,R_1,a)$ parallel trapping obstacle] \label{def:ptd} For $0<R_0<R_1$ and $a>0$ we say that $\Omega_-$ is an {\em $(R_0,R_1,a)$ parallel trapping obstacle} if it is an $(R_0,R_1)$ obstacle and there exist $y,z\in \Gamma$ with $n_d(y)=0$ and $n(z)=-n(y)$ such that $z=y+an(y)$ and, for some $\epsilon>0$, $n(x)=n(y)$ for $x\in \Gamma\cap B_\epsilon(y)$, $n(x)=n(z)$ for $x\in \Gamma\cap B_\epsilon(z)$, and
$$
\Omega_C := \big\{x+tn(x):0<t<a \mbox{ and } x\in \Gamma\cap B_\epsilon(y)\big\}\subset \Omega_+.
$$
\end{definition}

The point of this definition is that $\Gamma\cap B_\epsilon(y)$ and $\Gamma\cap B_\epsilon(z)$ are parallel parts of $\Gamma$ and that $\{x+tn(x):0< t <a\}$ is a (trapped) billiard trajectory in $\Omega_+$ for $x\in \Gamma\cap B_\epsilon(y)$.

Figures \ref{fig:twosquares} and \ref{fig:cavity} are examples both of Definition \ref{def:ptd} and Definition \ref{def:trapc}. By Lemma \ref{lem:strongly} they are also examples of Definition \ref{def:trapb}, satisfying $Z(x)\cdot n(x)\geq 0$ for all $x\in \Gamma$ for which $n(x)$ is defined, where $Z$ is given by \eqref{eq:Zdeff}, with $R_1> \re^{1/4} R_0$ as indicated in the figures and $\chi\in C^3[0,\infty)$ satisfying the conditions of Definition \ref{def:trapb}. Figure \ref{fig:cavity_vector_field} illustrates the direction of the vector field $Z$ for the obstacle and choice of $R_0$ and $R_1$ in Figure \ref{fig:cavity},  with $\chi$ constructed as in the proof of Lemma \ref{lem:strongly}.

One example of a strongly $(R_0, R_1)$ obstacle supporting parabolic trapping that is \emph{not} an $(R_0,R_1,a)$ parallel trapping obstacle is a 3-d cube with a circular cylinder (of diameter $a$) taken out of one side; to be specific let us take the obstacle $\Omega_-:=\{x:|x_j|<a \mbox{ for } j=1,2,3\}\setminus \{x:x_1^2+x_2^2\leq a^2/4 \mbox{ and } x_3\geq 0\}$. This is strongly $(R_0,R_1)$ by Remark \ref{ex:trap}(i), but is not an $(R_0, R_1,a)$ parallel trapping obstacle because, although there exist $y,z \in \Gamma$ with $n_d(y)=0$ and $n(z)=-n(y)$ such that $z= y + a n(y)$ (on the inside of the cylinder), the normal vector is not constant in a neighbourhood of $y,z$, and so there does not exist an $\epsilon>0$ such that $n(x)=n(y)$ for $x \in \Gamma \cap B_\epsilon(y)$ and $n(x)= n(z)$ for $x\in\Gamma \cap B_\epsilon(z)$.

There also exist obstacles supporting parabolic trapping that are not $(R_0, R_1)$ obstacles, in which case they are also not strongly $(R_0,R_1)$ obstacles nor $(R_0,R_1,a)$ parallel trapping obstacles. For example, let
$$
S_1 := \{x:|x_1|\leq 1/2, x_2\geq 0\} \quad \mbox{and} \quad S_2 := \{x:x_1\geq 1, -3/2\leq x_2 \leq -1/2\}.
$$
Then the obstacle
$
\Omega_- := \{x:|x_1|< 2, -3<x_2<1\}\setminus S_1,
$
%i.e.~
a square with a smaller, unit square removed, is a strongly $(R_0,R_1)$ obstacle and an $(R_0,R_1,a)$ parallel trapping obstacle, if $R_1>\re^{1/4}R_0$, $R_0\geq \sqrt{5}/2$, and $a=1$.
%$a=1$, provided $2>R_1>\re^{1/4}R_0$ and $R_0\geq \sqrt{5}/2$ (alternatively, $R_1>\re^{1/4}R_0$ and $R_0\geq \sqrt{13}$).
%$\sqrt{5}/2\leq R_0< \re^{1/4} R_1 < 2$
%$R_0\geq \sqrt{13}, \,R_1 >  \re^{1/4} R_0 $,
%and
%$a=1$.
But $\Omega:= \Omega_-\setminus S_2$, %i.e.~
a square with two unit squares removed from sides that point in different directions, supports parabolic trapping for the same wavenumbers as $\Omega_-$, but is not an $(R_0,R_1)$ obstacle for any choice of $R_1>\re^{1/4} R_0$, since, where $Z$ is given by \eqref{eq:Zdeff}, it does not hold that $Z(x)\cdot n(x)\geq 0$ for all $x\in \partial \Omega \cap \partial S_2$ for which $n(x)$ is defined. Since the sides from which the squares are removed point in different directions, the obstacle $\Omega$ is, moreover, not an $(R_0, R_1)$ obstacle in any coordinate system
\footnote{Thus Theorem \ref{thm:newintro} does not apply to $\Omega$. But we expect that a modified Theorem \ref{thm:newintro} holds for $\Omega$, based on a more elaborate construction of the vector field $Z$, since the arguments related to the construction of the vector field around a trapping region are to a large extent local.}.

\subsubsection{Resolvent estimates and bounds on the Dirichlet-to-Neumann (DtN) map}

In the following theorem $\chi$ is any function such that $\Omega_-$ is an $(R_0,R_1)$ obstacle in the sense of Definition \ref{def:trapb} and, for $k>0$ and $R> R_\Gamma$,
\begin{equation} \label{eq:normdefs}
\|u\|^2_{H^1_k(\Omega_R)} := \int_{\Omega_R} \left(|\nabla u|^2+k^2|u|^2\right) \rd x \;\; \mbox{ and } \;\; \|u\|^2_{H^1_k(\Omega_R;\chi)} := \int_{\Omega_R} \left(|\nabla u|^2+k^2|u|^2\right)\chi \rd x.
\end{equation}
The notation $A\lesssim B$ (or $B\gtrsim A$) means that $A \leq CB$, where the constant $C>0$ does not depend on $k$ or $f$ (but will depend on $\Omega_+$, $R$, and $k_0$). We write $A\sim B$ if $A\lesssim B$ and $A\gtrsim B$.
\begin{theorem}[Resolvent estimates] \label{thm:resol} Let $f\in L^2(\Omega_+)$ have compact support in $\overline{\Omega_+}$, and let $u\in H^1_{\mathrm{loc}}(\Omega_+)$ be a solution to the Helmholtz equation $\Delta u + k^2 u = -f$ in $\Omega_+$ that satisfies the Sommerfeld radiation condition
\begin{equation}\label{eq:src}
\pdiff{u}{r}(x)
- \ri k u(x) = o\left(\frac{1}{r^{(d-1)/2}}\right),
\end{equation}
as $r \tendi$, uniformly in $\widehat{x}:= x/r$, and the boundary condition $\gamma_+ u = 0$. If $\Omega_-$ is an $(R_0,R_1)$ obstacle for some $R_1>R_0>0$ then, for all $R> \max_{x\in \Gamma\cup \supp(f)}|x|$, given $k_0>0$,
\begin{equation}\label{eq:resol}
k^{-1}\|u\|_{H^1_k(\Omega_R)} + \|\partial_d u \|_{L^2(\Omega_R)} + \|u\|_{H^1_k(\Omega_R;\chi)}\lesssim k\|f\|_{L^2(\Omega_+)},
\end{equation}
for all $k\geq k_0$. If the support of $f$ does not intersect $B_{R_0}$ and
\begin{equation} \label{eq:fchinorm}
\|f\|_{L^2(\Omega_+;\chi^{-1})} := \left(\int_{\Omega_+} \frac{|f|^2}{\chi}\, \rd x\right)^{1/2}<\infty,
\end{equation}
then the bound \eqref{eq:resol} holds with $k\|f\|_{L^2(\Omega_+)}$ replaced by $\|f\|_{L^2(\Omega_+;\chi^{-1})}$.
\end{theorem}

This theorem contains the following important special cases.

1.~$\Oi$ is star-shaped and $R_1\leq \inf_{x\in \Gamma}|x|$.
In this case, since $\chi(r)=1$ for $r\geq R_1$, the bound recovers the standard bound when $\Omega_-$ is Lipschitz and star-shaped  that is sharp in its dependence on $k$ (see \cite{ChMo:08} and the discussion in \S\ref{sec:discuss}), namely
\begin{equation}\label{eq:resol2}
\|u\|_{H^1_k(\Omega_R)}\lesssim \|f\|_{L^2(\Omega_+)}, \quad \mbox{for } k\geq k_0.
\end{equation}

2.~$\Oi$ is an $(R_0,R_1,a)$ parallel trapping obstacle, such as those in Figures \ref{fig:twosquares} and \ref{fig:cavity}. In this case it holds that
\begin{equation}\label{eq:resol4}
\|u\|_{H^1_k(\Omega_R)}+ k\|u\|_{H^1_k(\Omega_R;\chi)}\lesssim k^2\|f\|_{L^2(\Omega_+)}.
\end{equation}
Furthermore, if, for some $R^\prime>R_0$, the support of $f$ does not intersect $B_{R^\prime}$ then
\begin{equation}\label{eq:resol3}
\|u\|_{H^1_k(\Omega_R)}+ k\|u\|_{H^1_k(\Omega_R;\chi)}\lesssim k\|f\|_{L^2(\Omega_+)}.
\end{equation}

The simple constructions at the end of \cite[\S3]{ChMo:08} show that, for every $(R_0,R_1)$ obstacle, there exists an $f$ supported outside $B_{R_1}$ such that $\|u\|_{H^1_k(\Omega_R;\chi)}\gtrsim \|f\|_{L^2(\Omega_+)}$, so that the power of $k$ in front of $\|u\|_{H^1_k(\Omega_R;\chi)}$ in \eqref{eq:resol3} is sharp, and that, for every $(R_0,R_1,a)$ parallel trapping obstacle, there exists a compactly supported $f$ such that
\begin{equation}\label{eq:resol5}
\|u\|_{H^1_k(\Omega_R)}\gtrsim k\|f\|_{L^2(\Omega_+)}, \quad  \mbox{for } k\in \{m\pi/a:m\in \mathbb{N}\}
\end{equation}
(this quantisation condition is a requirement that the length $a$ of the billiard orbits between the parallel sides of the trapping domain is a multiple of half the wavelength). These lower bounds show that
the power of $k$
on the right hand side of \eqref{eq:resol4} can be reduced at most from $k^2$ to $k$, and that on the right hand side of \eqref{eq:resol3} cannot be reduced.

In the following theorem we use the notation $\|\cdot\|_{H_k^s(\Gamma)}$ defined by equations \eqref{eq:Ganorm}--\eqref{eq:Ganorm3} below.

\begin{theorem}[Bounds on the DtN map] \label{thm:dtn} Let $u\in H^1_{\mathrm{loc}}(\Omega_+)$ be a solution to the Helmholtz equation $\Delta u + k^2 u = 0$ in $\Omega_+$ that satisfies the Sommerfeld radiation condition \eqref{eq:src} and the boundary condition $\gamma_+u=g$. If $\Omega_-$ is an $(R_0,R_1)$ obstacle for some $R_1>R_0>0$ then, for all $R> R_\Gamma$, given $k_0>0$,
\begin{equation}\label{eq:dtn2}
\|u\|_{H^1_k(\Omega_R)} + \|\partial_n^+ u\|_{L^2(\Gamma)} \lesssim k^2 \,\|g\|_{H_k^1(\Gamma)},
\end{equation}
for all $k\geq k_0$ if $g\in H^1(\Gamma)$. Further, uniformly for $0\leq s \leq 1$, provided $g\in H^{s}(\Gamma)$,
\begin{equation}\label{eq:dtn3}
\|\partial_n^+ u\|_{H_k^{s-1}(\Gamma)} \lesssim k^2 \|g\|_{H_k^{s}(\Gamma)}, \quad \mbox{and} \quad \|\partial_n^+ u\|_{H^{s-1}(\Gamma)} \lesssim k^3 \|g\|_{H^{s}(\Gamma)} \quad \mbox{for }k\geq k_0.
\end{equation}
If $g = -u^i|_\Gamma$, where $u^i$ satisfies $\Delta u^i + k^2 u^i=0$ in a neighbourhood $G$ of $\overline{\Omega_-\cup B_{R_0}}$, then
\begin{equation}\label{eq:dtn4}
\|u\|_{H^1_k(\Omega_R)} + \|\partial_n^+ u\|_{L^2(\Gamma)} \lesssim k^2 \sup_{x\in G}|u^i(x)|.
\end{equation}
\end{theorem}

We derive the bounds \eqref{eq:dtn2} and \eqref{eq:dtn3} from the resolvent estimate in Theorem \ref{thm:resol} using the method in Baskin et al.\ \cite{BaSpWu:16} (a sharpening of previous arguments in \cite{LaVa:11,Sp:14}), which we capture below in Lemma \ref{lem:recipe}. (This method was used in \cite{BaSpWu:16} to deduce the sharp DtN map bound $\|\partial_n^+ u\|_{L^2(\Gamma)} \lesssim \|g\|_{H_k^1(\Gamma)}$, when $\Oi$ is nontrapping, from the resolvent estimate \eqref{eq:nt_estimate}/\eqref{eq:resol2}.)

We also apply Lemma \ref{lem:recipe} to write down DtN bounds for the two other trapping configurations for which resolvent estimates are known,
namely elliptic and hyperbolic trapping discussed in \S\ref{sec:4star}; see Corollaries \ref{cor:dtnws} and \ref{cor:dtnib} below.

The final bound \eqref{eq:dtn4} in Theorem \ref{thm:dtn} is derived from \eqref{eq:resol3}. To illustrate this result, suppose that $u$ in Theorem \ref{thm:dtn} is  the scattered field corresponding to an incident plane wave $u^i(x) = \exp(\ri k x\cdot \hat{a})$, for some unit vector $\hat{a}$, with $\gamma_+u = g = - u^i|_\Gamma$. Then $\|g\|_{H_k^1(\Gamma)}\sim k$ so that \eqref{eq:dtn2} implies $\|\partial_n^+ u\|_{L^2(\Gamma)} \lesssim k^3$ for $k\geq k_0$, while \eqref{eq:dtn4} implies the sharper bound
 $\|\partial_n^+ u\|_{L^2(\Gamma)} \lesssim k^2$.

\subsection{Discussion of  related results} \label{sec:discuss}

In \S\ref{sec:4star} we discussed the resolvent estimate in Theorem \ref{thm:resol} in the context of the nontrapping resolvent estimate \eqref{eq:nt_estimate} and the resolvent estimates for elliptic trapping \eqref{eq:Burq} and hyperbolic trapping \eqref{eq:Ikawa2}, all in the obstacle case. In this section we discuss Theorem \ref{thm:resol} in a slightly wider context.

\paragraph{Local energy decay and resonance-free regions.}

In the paper so far, we have only been concerned with resolvent estimates on the real axis (i.e.~for $k$ real), but establishing such an estimate is intimately related to (i) meromorphic continuation of the resolvent and resonance-free regions beneath the real axis, and (ii) local energy decay of the solution of the wave equation; results about the link between these three properties can be found in
\cite[Theorems 1.1 and 1.2]{Vo:99}, \cite[Proposition 4.4 and Lemma 4.7]{Bu:04}, \cite[Theorem 1.3]{BoPe:06}, \cite{Da:12}, \cite[Theorem 1.5]{Vo:14}, and \cite[Theorem 1]{In:17}, and overviews of results about resonances in obstacle scattering can be found in \cite[Page 24]{Zw:17}, \cite[Chapter 6]{DyZw:17}.
In particular, the result of Datchev \cite{Da:12} (suitably translated from the setting of scattering from a potential to the setting of scattering by an obstacle) could be used to prove that the resolvent estimate of Theorem \ref{thm:resol} holds for $k$ in a prescribed neighbourhood below the real axis, but we do not pursue this here.

\paragraph{Trapping by diffraction from corners.}
When a ray hits a corner of, say, a polygon, it produces diffracted rays emanating from the corner, and in particular some that travel along the sides of the polygon. This means that there exist glancing rays that travel around the boundary of the polygon (hitting a corner and then either continuing on the next side or travelling back) and do not escape to infinity; thus the exterior of a polygon is, in this sense, a trapping domain. At each diffraction from a corner, however, these rays lose energy, and thus the trapping is in a weaker sense than having a closed path of rays. Baskin and Wunsch \cite{BaWu:13} proved that the nontrapping resolvent estimate \eqref{eq:nt_estimate} holds when $\Oi$ is a \emph{nontrapping polygon}.

\begin{definition}[Nontrapping polygon \cite{BaWu:13}]\label{def:nt2}
$\Oi\subset \Rea^2$ is a \emph{nontrapping polygon} if $\Oi$ is a finite union of disjoint polygons such that: (i) no three vertices are colinear; and (ii),
given $R> R_\Gamma$, there exists a $T(R)<\infty$ such that
all the
billiard trajectories
that start in $\Omega_R$ at time zero and miss the vertices leave $\Omega_R$ by time $T(R)$. (For a more precise statement of (ii) see \cite[Section 5]{BaWu:13}.)
\end{definition}

\paragraph{Parabolic and degenerate hyperbolic trapping by metrics.}

In the setting of scattering by metrics, Christianson and Wunsch \cite{ChWu:13} exhibited a sequence of metrics, indexed by $m=1,2,\ldots$, where the case $m=1$ corresponds to a single trapped hyperbolic geodesic, but the hyperbolicity degenerates as $m$ increases, and for $m\geq 2$ the sharp bound
\beq\label{eq:ChWu}
\N{\cutoff}_{L^2\rightarrow L^2}\lesssim k^{-2/(m+1)}
\eeq
holds  (see also the review \cite{Wu:12}).
Observe that, as $m\tendi$, the right-hand side of the bound tends to $k^0$, i.e.~a constant. This case of infinite-degeneracy was studied by Christianson \cite{Ch:13}, who proved the bound
\beq\label{eq:ChWu2}
\N{\cutoff}_{L^2\rightarrow L^2}\lesssim k^{\eps}
\eeq
for any $\eps>0$ (where the omitted constant depends on $\eps$)
\cite[Theorem 1 and Proposition 3.8]{Ch:13}.
The analogue of the situation in \cite{ChWu:13} in the obstacle setting is two strictly convex obstacles being flattened (in the neighbourhood of the trapped ray), and the bounds \eqref{eq:ChWu} and \eqref{eq:ChWu2} are therefore consistent with an expectation that the sharp bound for an $(R_0,R_1,a)$ parallel trapping obstacle should be $\N{\cutoff}_{\LtLt}\lesssim 1$.

The situation of two convex obstacles being flattened was investigated by Ikawa in \cite{Ik:85} and \cite{Ik:94}, with \cite[Theorem 3.6.2]{Ik:94} bounding a mapping related to the resolvent in a region below the real axis but excluding neighbourhoods of the resonances. Although this estimate depends on the order of the degeneracy, it is not a resolvent estimate per se and does not apply everywhere on the real axis, so it does not appear to lead to a bound similar to \eqref{eq:ChWu}.

\paragraph{Obstacles rougher than Lipschitz.} The vector-field/commutator method of Morawetz can be used to obtain resolvent estimates for rough domains under the assumption of star-shapedness. Indeed, essentially this method was used in Chandler-Wilde and Monk \cite{ChMo:08} to prove the nontrapping resolvent estimate \eqref{eq:nt_estimate}, not only for Lipschitz star-shaped  $\Oi$ in 2- and 3-d, but also for $C^0$ star-shaped $\Oi$; indeed their proof of the resolvent estimate \eqref{eq:nt_estimate} assumes only that $\R^d\setminus \Omega_+$ is bounded, that $0\not\in \Omega_+$, and that if $x\in \Omega_+$ then $sx\in \Omega_+$ for every $s>1$ \cite[Lemma 3.8]{ChMo:08}.

\paragraph{Parallel trapping domains in rough surface scattering.}

A resolvent estimate with the same $k$-dependence as \eqref{eq:us} was proved for the Helmholtz equation posed above an unbounded rough surface in \cite[Theorem 4.1]{ChMo:05}. Denoting the domain above the surface by $\Oe$, the geometric assumption in \cite[Theorem 4.1]{ChMo:05} is that $x \in \Oe$ implies that $x+ se_d \in \Oe$ for all $s>0$, and that, for some $h\leq H$, $U_H\subset \Oe \subset U_h$, where $U_a:= \{x:x_d>a\}$; these conditions allow square/cube-shaped cavities in the surface, and therefore allow the same type of parabolic trapping as present for $(R_0,R_1,a)$ parallel trapping obstacles. At the beginning of  \S\ref{sec:resol} we discuss how the proof of Theorem \ref{thm:resol} uses ideas from the proof of \cite[Theorem 4.1]{ChMo:05}.

\subsection{Application to finite element discretisations}\label{sec:FEM}

A standard reformulation of the problem studied in Theorem \ref{thm:resol}, and the starting point for discretisation by finite element methods (FEMs) (e.g., \cite{MeSa:11}), is the variational problem \eqref{eq:weakst} below in which the unknown is $u_R:= u|_{\Omega_R}$, for some $R>R_\Gamma$, which lies in the Hilbert space $V_R:= \{w|_{\Omega_R}: w\in H^1_{\text{loc}}(\Omega_+)\mbox{ and }\gamma_+w=0\}$. The following corollary bounds the inf-sup constant in this formulation, the upper bound \eqref{eq:ses2} taken from \cite{ChMo:08}.

\begin{corollary}[Bound on the inf-sup constant] \label{cor:infsup}  If $\Omega_-$ is an $(R_0,R_1)$ obstacle for some $R_1>R_0>0$,  then, for all $R> R_\Gamma$, given $k_0>0$,
\begin{equation} \label{eq:infsupfirst}
\beta_R := \inf_{0\neq u\in V_R}\, \sup_{0\neq v\in V_R} \frac{|a(u,v)|}{\|u\|_{H^1_k(\Omega_R)}\|v\|_{H^1_k(\Omega_R)}} \gtrsim k^{-3}
\end{equation}
for all $k\geq k_0$, where
\begin{equation} \label{eq:ses}
a(u,v) := \int_{\Omega_R} (\nabla u\cdot \overline{\nabla v} - k^2 u \bar v)\,\rd x - \int_{\Gamma_R} \overline{\gamma  v} \,P_R^+ \gamma u \, \rd s, \quad \mbox{for }u,v\in V_R,
\end{equation}
is the sesquilinear form in \eqref{eq:weakst}. Here $\gamma$ is the trace operator from $\Omega_R$ to $\Gamma_R:= \partial B_R$, and $P_R^+$ is the DtN map in the case that $\Omega_-=B_R$ and $\Gamma=\Gamma_R$. Further, if $\Omega_+$ is an $(R_0,R_1,a)$ parallel trapping obstacle for some $a>0$, then
\begin{equation}\label{eq:ses2}
\beta_R \lesssim k^{-2}, \quad  \mbox{for } k\in \{m\pi/a:m\in \mathbb{N}\}.
\end{equation}
\end{corollary}

We point out in Remark \ref{rem:resinfsup} (and see Table \ref{tab:Ainv}) that the arguments (derived from \cite{ChMo:08}) that we use to derive the lower bound on $\beta_R$ from the resolvent estimates in Theorem \ref{thm:resol} apply whenever a resolvent estimate is available. Thus we can also write down lower bounds on $\beta_R$ for the worst case of elliptic trapping, and for the case of the mild hyperbolic trapping between two smooth strictly convex obstacles: see Remark \ref{rem:resinfsup}  and Table \ref{tab:Ainv} for details.

Our results also prove the missing assumption needed to apply the wavenumber-explicit $hp$-finite element analysis of Melenk and Sauter \cite{MeSa:11} to problems of obstacle scattering when $\Omega_+$ is trapping. Suppose that $\Gamma$ is analytic, $R>R_\Gamma$, and let $\mathcal{T}_h$ be a quasi-uniform triangulation of $\Omega_R$ in the sense of \cite[Assumption 5.1]{MeSa:11}, with $h:=\max_{K\in \mathcal{T}_h}\mathrm{diam}(K)$ the maximum element diameter. Let $\mathcal{S}_0^{p,1}(\mathcal{T}_h):= \mathcal{S}^{p,1}(\mathcal{T}_h)\cap V_R$, where $\mathcal{S}^{p,1}(\mathcal{T}_h)$ is the space of continuous, piecewise polynomials of degree $\leq p$ on the triangulation $\mathcal{T}_h$  \cite[Equation (5.1)]{MeSa:11}. Then a (Galerkin) finite element approximation, $u_{hp}\in \mathcal{S}_0^{p,1}(\mathcal{T}_h)$, to the solution $u_R$ of \eqref{eq:weakst} is defined by
$$
a(u_{hp},v_{hp}) = G(v_{hp}), \quad \mbox{for all }v_{hp}\in \mathcal{S}_0^{p,1}(\mathcal{T}_h),
$$
where the anti-linear functional $G$ is given by \eqref{eq:Gdef}.
Melenk and Sauter's results imply  that if, given $k_0>0$, there exists $q\geq 1$ such that
\beq \label{eq:bpoly}
\beta_R \gtrsim k^{-q}, \quad \mbox{for } k\geq k_0,
\eeq
then the finite element method is quasi-optimal, i.e.\
\beq \label{eq:qo}
\|u_R-u_{hp}\|_{H^1_k(\Omega_R)} \leq C \inf_{v_{hp}\in \mathcal{S}_0^{p,1}(\mathcal{T}_h)}\|u_R-v_{hp}\|_{H^1_k(\Omega_R)},
\eeq
provided that $p$ increases logarithmically with $k$, and $N_{hp}$, the degrees of freedom (the dimension of the subspace $\mathcal{S}_0^{p,1}(\mathcal{T}_h)$), increases with $k$ so as to maintain a fixed number of degrees of freedom per wavelength (so that $N_{hp} \sim k^{d}$). This is a strong result, in particular the ``pollution effect'' \cite{BaSa:00} that arises with standard $h$-version finite element methods, which implies a requirement to increase  $N_{hp}$ at a faster rate than $k^{d}$, is avoided, and this analysis is fully wavenumber-explicit (the constant $C$ in \eqref{eq:qo} is independent of $k$, $h$, and $p$). However, the result in \cite{MeSa:11} is established only for the case when $\Omega_-$ is star-shaped with respect to a ball. Corollary \ref{cor:infsup} implies that \eqref{eq:bpoly}, and hence also \eqref{eq:qo}, applies also for $(R_0,R_1)$ obstacles with  $\Gamma$ analytic. This class includes many domains $\Omega_+$ that allow trapped periodic orbits, though not $(R_0,R_1,a)$ parallel trapping obstacles for which $\Gamma$ is not analytic. However, we expect that a version of \eqref{eq:qo} can be proved for $(R_0,R_1,a)$ parallel trapping obstacles in 2-d that are polygonal, by combining Corollary \ref{cor:infsup} with the wavenumber-explicit $hp$-FEM analysis for non-convex polygonal domains in \cite{EsMe:12}.

\subsection{Our main results for boundary integral equations}\label{sec:BIE}

The results of Theorems \ref{thm:dtn} can be used to prove results about integral equations. Our main result concerns the standard boundary integral equation formulations of the Helmholtz exterior Dirichlet problem.

If $u$ is the solution to the Helmholtz exterior Dirichlet problem, the Neumann trace of $u$, $\dnpu$, satisfies the integral equation
\beq\label{eq:CFIE}
\opA \, \dnpu = f_{k,\eta}
\eeq
on $\Gamma$, where the integral operator $\opA$ is the so-called \emph{combined-potential} or \emph{combined-field} integral operator (defined by \eqref{eq:CFIEdef} below), the parameter $\eta$ is a real constant different from zero, and
$f_{k,\eta}$ is given in terms of the known Dirichlet data $\gamma_+ u$ (see \eqref{eq:CFIE2}).  The equation \eqref{eq:CFIE} also arises in so-called {\em sound soft} scattering problems in which $u$ is interpreted as the scattered field corresponding to an incident field $u^i$, the total field $u^t:=u+u^i$ satisfies $\gamma_+u^t=0$ on $\Gamma$, and \eqref{eq:CFIE} is satisfied by $\partial_n^+u^t$ with $f_{k,\eta}$ given in terms of Dirichlet and Neumann traces of $u^i$ on $\Gamma$; see \eqref{eq:f}. The other standard integral equation for the exterior Dirichlet problem  (\eqref{eq:BW} below) takes the form $A_{k,\eta}\phi= h$, where $h=\gamma^+u$ and $A_{k,\eta}$ is the adjoint of $\opA$ with respect to the real inner product on $L^2(\Gamma)$ (as defined below \eqref{eq:interp}).

\subsubsection{Bounds on $\opAinv$ and $\opABWinv$}\label{sec:151}

The following corollary gives bounds on $\opAinv$; bounds on $\opABWinv$ follow by duality (see \eqref{eq:duals} and \eqref{eq:duals2} below).

\begin{corollary}[Bounds on $\opAinv$]\label{cor:CFIE}
Suppose that the (finite number of) disjoint components of the Lipschitz open set $\Omega_-$ are each either star-shaped with respect to a ball or $C^\infty$, that $\Omega_-$ is an $(R_0,R_1)$ obstacle for some $R_1>R_0>0$, and that $\eta = ck$, for some $c\in \R\setminus\{0\}$.
Then, given $k_0>0$, for all $k\geq k_0$,
\beq\label{eq:Ainv_bound_main}
\normAinv_{\LtGt} \lesssim k^2;
\eeq
indeed
\beq\label{eq:Ainv_bound_main1}
\normAinv_{H_k^s(\Gamma)\to H_k^s(\Gamma)} \lesssim k^2\quad \mbox{ and } \quad \normAinv_{H^s(\Gamma)\to H^s(\Gamma)} \lesssim k^{2-s},
\eeq
for $-1\leq s\leq 0$.
If $\Omega_-$ is an $(R_0,R_1,a)$ parallel trapping obstacle for some $a>0$, then
\beq\label{eq:Ainv_bound_mainlower}
\normAinv_{\LtGt} \gtrsim k,\quad  \mbox{for } k\in \{m\pi/a:m\in \mathbb{N}\}.
\eeq
\end{corollary}

\begin{definition}[Piecewise smooth] \label{def:psmooth}
We say that  the bounded Lipschitz open set $\Omega_-$ and its boundary $\Gamma$ are {\em piecewise smooth} if $\Gamma$ can be written as a finite union $\Gamma=\cup_{j=1}^M\overline{\Gamma_j}$ where each $\Gamma_j$ is relatively open in $\Gamma$, the $\Gamma_j$ are pairwise disjoint, $\Gamma_{\mathrm{sing}}:= \Gamma \setminus \cup_{j=1}^M\Gamma_j$ has zero surface measure, and each $\Gamma_j\subset \widetilde \Gamma_j$, where $\widetilde \Gamma_j$ is the boundary of a bounded $C^\infty$ open set.
\end{definition}

\begin{remark}[Extensions of Corollary \ref{cor:CFIE}] \label{rem:A} If the components of $\Omega_-$ are not all $C^\infty$ or star-shaped with respect to a ball, using the bounds from Theorem \ref{cor:ItD} that apply in more general cases it  follows that \eqref{eq:Ainv_bound_main} and \eqref{eq:Ainv_bound_main1} still hold but with $k^2$ and $k^{2-s}$ replaced by $k^{9/4}$ and $k^{9/4-s}$, respectively, if each component of $\Omega_-$ is piecewise smooth or star-shaped with respect to a ball, with $k^2$ and $k^{2-s}$ replaced by $k^{5/2}$ and $k^{5/2-s}$, respectively, in the general case.
\end{remark}

\bre[How sharp are the bounds in Corollary \ref{cor:CFIE}?]
The numerical computations in \cite[\S4.7]{BeChGrLaLi:11} and  \cite[Example 5.2]{GrLoMeSp:15} give an example of an $(R_0,R_1,a)$ parallel trapping domain for which, when $\eta=\pm k$, $\normAinv_{\LtGt} \sim k$ (at least for the range of $k$ considered in the experiments), i.e., they indicate that the lower bound rate of $k$ in \eqref{eq:Ainv_bound_mainlower} is sharp.
\ere

\bre[Previous upper bounds on $\opAinv$]\label{rem:CFIE}
There have been three previous upper bounds for $\normAinv_{L^2(\Gamma)\to L^2(\Gamma)}$ proved in the literature, all for nontrapping cases. The first  is the bound
\beq \label{eq:Ainvknown}
\normAinv_{L^2(\Gamma)\to L^2(\Gamma)} \lesssim 1 + \frac{k}{|\eta|}, \quad \mbox{for } k>0,
\eeq
 for the case when $\Omega_-$ is Lipschitz and star-shaped with respect to a ball \cite[Theorem 4.3]{ChMo:08} (\cite{ChMo:08} assumes additionally that $\Gamma$ is piecewise smooth, but this requirement can be avoided using density results from \cite[Lemmas 2 and 3]{CoDa:98}; see \cite[Remark 3.8]{Sp:14}).
 The second is the bound  $\normAinv_{L^2(\Gamma)\to L^2(\Gamma)} \lesssim 1 $ when $\Oe$ is nontrapping (in the sense of Definition \ref{def:nt1}) and $|\eta|\sim k$ \cite[Theorem 1.13]{BaSpWu:16}. The third is the bound, given $k_0>0$, that
\beq \label{eq:Ainvntp}
\normAinv_{L^2(\Gamma)\to L^2(\Gamma)} \lesssim k^{5/4}\left(1 + \frac{k^{3/4}}{|\eta|}\right), \quad \mbox{for } k\geq k_0,
\eeq
when $\Omega_-$ is a nontrapping polygon \cite[Theorem 1.11]{Sp:14}. (In \S\ref{sec:other} below we improve this bound, when $|\eta|\sim k$,  to $\normAinv_{L^2(\Gamma)\to L^2(\Gamma)} \lesssim k^{1/4}$ as a corollary of results in \cite{BaSpWu:16}.)

The only other known bound on $\opAinv$ appears in the thesis of the third author \cite[Theorem 5.19]{Gi:17} and we obtain a sharpened and generalised form of it as \eqref{eq:AinvknownGibbs} below.
\ere

The upper bounds on $\normAinv_{\LtGt}$ in Corollary \ref{cor:CFIE} and Remark \ref{rem:CFIE} use the representation \eqref{eq:key} below that expresses $\opAinv$ in terms of the exterior DtN map and an interior impedance to Dirichlet map (the use of this representation in \cite{ChMo:08} was implicit, and the representation was stated explicitly for the first time as \cite[Theorem 2.33]{ChGrLaSp:12}).
We prove the bound \eqref{eq:Ainv_bound_main} in the same way, using the DtN map bound \eqref{eq:dtn2} along with existing bounds on the solution of the interior impedance problem; see \S\ref{sec:6.2} and \S\ref{sec:6.3} below.
We extend this methodology to prove the bounds \eqref{eq:Ainv_bound_main1} on $\opAinv$ also as an operator on $H^s(\Gamma)$, for $-1\leq s \leq 0$. These arguments also show that, when $|\eta|\sim k$,
  \beq\label{eq:Ainv_bound_main2}
\normAinv_{H_k^s(\Gamma)\to H_k^s(\Gamma)} \lesssim 1\quad \mbox{ and } \quad \normAinv_{H^s(\Gamma)\to H^s(\Gamma)} \lesssim k^{-s},
\eeq
for $-1\leq s\leq 0$ in the cases when $\Omega_-$ is \emph{either} Lipschitz and star-shaped with respect to a ball \emph{or} nontrapping.

The arguments from \cite{ChMo:08} and \cite{BaSpWu:16} are summarised in Lemma \ref{lem:recipe2} as a general ``recipe" where the input is a resolvent estimate for the exterior Dirichlet problem, and the output is a bound on $\opAinv$ and $\opABW^{-1}$.
We apply this recipe to the two other existing resolvent estimates for trapping obstacles \eqref{eq:Burq} and \eqref{eq:Ikawa2}, showing that, when $|\eta|\sim k$,
\beq \label{eq:mild}
\normAinv_{L^2(\Gamma)\to L^2(\Gamma)} \lesssim \log(2+k)
\eeq
for the mild (hyperbolic) trapping case of a finite number of smooth convex obstacles with strictly positive curvature (additionally satisfying the conditions in Definition \ref{def:IB}). Similarly
\beq \label{eq:worst}
\normAinv_{L^2(\Gamma)\to L^2(\Gamma)} \lesssim \exp(\alpha k),
\eeq
for some $\alpha>0$, for the general $C^\infty$ case, this ``worst case'' exponential growth achieved, as observed earlier in 2-d \cite{BeChGrLaLi:11}, when the geometry of $\Gamma$ is such that there exists a stable (elliptic) periodic orbit. The same bounds hold on $\normAinv_{H_k^s(\Gamma)\to H_k^s(\Gamma)}$ for $-1\leq s\leq 0$, and  they apply also to $\normAinv_{H^s(\Gamma)\to H^s(\Gamma)}$ with the bounds increased by an additional factor $k^{-s}$.
In particular, in the hyperbolic trapping case, we have that,
when $|\eta|\sim k$, given $k_0>0$,
\beq \label{eq:AinvknownGibbs}
\normAinv_{H^{-1/2}(\Gamma)\to H^{-1/2}(\Gamma)} \lesssim k^{1/2} \log(2+k), \quad \mbox{for } k\geq k_0;
\eeq
this is an improvement of the bound in \cite[Theorem 5.19]{Gi:17} by a factor $k^{1/2}$.

These new bounds on the norm of $\opAinv$ in the cases of elliptic and hyperbolic trapping are of interest in their own right, but also contrast strongly with the new bounds for $(R_0,R_1,a)$ parallel trapping obstacles in Corollary \ref{cor:CFIE}. The bound \eqref{eq:mild} is only worse than the nontrapping bound \eqref{eq:Ainvknown} by a log factor, while in the worst case of elliptic trapping the norm of $\opAinv$ can grow exponentially through some sequence of wavenumbers \cite[Theorem 2.8]{BeChGrLaLi:11}. In between, for $(R_0,R_1,a)$ parallel trapping obstacles, Corollary \ref{cor:CFIE} proves polynomial growth (through a particular sequence of wavenumbers) at a rate between $k$ and $k^2$.

\subsubsection{Bounds on the condition numbers of $\opA$ and $\opABW$}

Many authors  \cite{KrSp:83,Kr:85,Am:90,ChGrLaLi:09,BeChGrLaLi:11,BaSpWu:16} have studied, in addition to the norm of $\opA$, its $L^2$ {\em condition number}, defined by
\beq \label{eq:cond}
\cond(\opA) := \|\opA\|_{L^2(\Gamma)\to L^2(\Gamma)}\,\|\opAinv\|_{L^2(\Gamma)\to L^2(\Gamma)}.
\eeq
This quantity is of interest because the condition number at a continuous level is closely related to the condition numbers of the matrices that arise in Galerkin-method discretisations. Indeed, if orthogonal basis functions are used and $\Gamma$ is smooth enough, the condition number of the Galerkin matrix converges to \eqref{eq:cond} as the discretisation is refined \cite[\S3]{BeChGrLaLi:11}. Thus, understanding the dependence of $\cond(\opA)$ on $k$ and on the geometry provides quantitative information about condition numbers of matrices at a discrete level, which in turn is relevant to the stability of numerical methods and the convergence of iterative solvers (though see the discussion in \cite[Section 7.2]{BaSpWu:16}, \cite{GaMuSp:16} regarding related quantities that may be more informative still). In \S\ref{sec:cond} we study $\cond(\opA)$ for trapping geometries, by combining bounds on $\opAinv$ with known bounds on the norm of $\opA$, notably those in Chandler-Wilde et al.~\cite{ChGrLaLi:09} and those due to Galkowski and Smith \cite{GaSm:15,HaTa:15}, proving the first upper bounds on the condition number for trapping obstacles, see Corollary \ref{cor:condno}.

\subsubsection{$k$-explicit convergence of boundary element methods}\label{sec:BEM}

Along with bounding the condition number of $\opA$, our results have another important application in the numerical solution of scattering problems by boundary integral equation methods.
Recall that the {\em boundary element method (BEM)} is the standard term for the numerical solution of boundary integral equations by the Galerkin method when the finite-dimensional subspaces consist of piecewise polynomials. When convergence is achieved by \emph{both} increasing the degree $p$ of the polynomials \emph{and} decreasing the mesh diameter $h$ the method is called the $hp$-BEM; when only the mesh diameter $h$ is decreased the method is called the $h$-BEM.

\paragraph{$hp$-BEM.}

L\"ohndorf and Melenk \cite{LoMe:11} provided the first wavenumber-explicit error analysis for $hp$-boundary element methods applied to the integral equations \eqref{eq:CFIE} and \eqref{eq:BW} under the assumption that $\Gamma$ is analytic. Their convergence results however require that, for some $k_0>0$ and $\gamma\geq 0$,
\beq \label{eq:poly}
\|\opAinv\|_{L^2(\Gamma)\to L^2(\Gamma)} \lesssim k^\gamma, \quad \mbox{for }k\geq k_0,
\eeq
so that these convergence results have been proved to date only for nontrapping domains (see \cite{LoMe:11} \cite[\S1.4]{BaSpWu:16}). Corollary \ref{cor:CFIE} above shows that \eqref{eq:poly} holds for all $(R_0,R_1)$ obstacles, and the bound \eqref{eq:mild} shows that \eqref{eq:poly} holds also for an Ikawa-like union of convex obstacles (in the sense of Definition \ref{def:IB}). Putting these results together with \cite[Corollary 3.18]{LoMe:11} we have the following result. In this corollary we use the notation $\mathcal{S}^p(\mathcal{T})$ for the set of piecewise polynomials of degree $p$ on the triangulation $\mathcal{T}$ in the sense of \cite[Equation (3.17)]{LoMe:11}.

\begin{corollary}[Quasi-optimality of the $hp$-BEM] \label{cor:mel}
Suppose that $\Gamma$ is analytic, that $\mathcal{T}_h$ is a quasi-uniform triangulation with mesh size $h$ of $\Gamma$ in the sense of \cite[Definition 3.15]{LoMe:11}, that $\eta = ck$, for some non-zero real constant $c$, and that $\Omega_-$ is \emph{either} nontrapping, \emph{or} an Ikawa-like union of convex obstacles, \emph{or} an $(R_0,R_1)$ obstacle.

Let $\partial_n^+u$ be the solution of \eqref{eq:CFIE} and let $v_{hp}\in \mathcal{S}^p(\mathcal{T}_h)$ be its Galerkin-method approximation, defined by
\beq \label{eq:galerkin}
(\opA v_{hp},v)_\Gamma = (f_{k,\eta},v)_\Gamma, \quad \mbox{for all } v\in \mathcal{S}^p(\mathcal{T}_h),
\eeq
where $(\cdot,\cdot)_\Gamma$ denotes the inner product on $L^2(\Gamma)$.
Then, given $k_0>0$, there exist constants $C_1, C_2, C_3$ (independent of $h$, $p$, and $k$) such that, if $k\geq k_0$,
\beq \label{eq:melcon}
\frac{kh}{p} \leq C_1, \quad \mbox{and} \quad p \geq C_2 \log(2+k),
\eeq
then the quasi-optimal error estimate
\beq \label{eq:melenk}
\|v_{hp} - \partial_n^+u\|_{L^2(\Gamma)} \leq C_3 \inf_{v\in \mathcal{S}^p(\mathcal{T}_h)} \|v - \partial_n^+u\|_{L^2(\Gamma)}
\eeq
holds.
\end{corollary}
An attractive feature of this result is that it demonstrates, via the bounds \eqref{eq:melcon}, that it is enough to maintain a ``fixed number of degrees of freedom per wavelength'', meaning increasing the dimension $N_{hp}$ of the approximating subspace  $\mathcal{S}^p(\mathcal{T}_h)$ in proportion to $k^{d-1}$, in order to maintain accuracy as $k$ increases, in agreement with much computational experience \cite{Ma:02} (and the numerical results in \cite{LoMe:11} show that this requirement is sharp). This corollary applies to all $(R_0,R_1)$ obstacles, including geometries that allow trapped periodic orbits, but does not apply to $(R_0,R_1,a)$ parallel trapping obstacles for which $\Gamma$ is not analytic.

\paragraph{$h$-BEM.}

It is commonly believed that, for nontrapping obstacles, the error estimate \eqref{eq:melenk} holds (with $C_3$ independent of $k$) for the $h$-BEM when $hk$ is sufficiently small, i.e., that a fixed number of degrees of freedom per wavelength is sufficient to maintain accuracy; this property can also be described by saying that the $h$-BEM does not suffer from the pollution effect \cite{BaSa:00}. However, the recent numerical experiments of Marburg \cite{Ma:16a}, \cite{Ma:17},  \cite{BaMa:17} give examples of nontrapping situations where pollution appears to occur, and therefore determining the sharp threshold on $h$ for the error estimate \eqref{eq:melenk} to hold in general is an exciting open question.

The best results so far in this direction are by Galkowski et al.~\cite{GaMuSp:16} (building on  results in \cite{GrLoMeSp:15}). Indeed \cite[Theorem 1.10]{GaMuSp:16} proves that \eqref{eq:melenk} holds (with $C_3$ independent of $k$) if: (i) $\Oi$ is smooth with strictly positive curvature\footnote{Here (and elsewhere in the paper), when $d=3$ we say that a piecewise-smooth $\Gamma$ has \emph{strictly positive curvature} if there exists $c>0$ such that, for almost every $\bx\in \Gamma$, the principal curvatures at $\bx$ are $\geq c$.
When $d=2$ we say that $\Gamma$ has strictly positive curvature if the above holds with the principal curvatures replaced by just the curvature.}
and $hk^{4/3}$ is sufficiently small; and (ii) $\Oi$ is nontrapping and $hk^{3/2}$ is sufficiently small (2-d) and $hk^{3/2}\log(2+k)$ is sufficiently small (3-d).

The arguments and results in \cite{GaMuSp:16, GrLoMeSp:15}, combined with the bounds on $\|\opAinv\|_{L^2(\Gamma)\to L^2(\Gamma)}$ that we obtain in this paper, enable us to prove in the next corollary the first $h$-BEM convergence results for trapping obstacles. The bounds in this corollary are, unsurprisingly, weaker  than the best results for nontrapping obstacles, but only by log factors for Ikawa-like unions of convex obstacles.

\begin{corollary}[Quasi-optimality of the $h$-BEM] \label{cor:hversion}
Suppose that $\Omega_-$ is $C^{2,\alpha}$ for some $\alpha\in (0,1)$, that $\eta = ck$, for some non-zero real constant $c$, that $k_0>0$, that $p\geq 0$, and that $\mathcal{T}_h$ is a shape-regular triangulation of $\Gamma$ in the sense of Definition \ref{def:triang}, with $h>0$ the maximum diameter of the elements $K\in \mathcal{T}_h$.
Let  $\partial_n^+u$ be the solution of \eqref{eq:CFIE}, and let $v_{hp}\in \mathcal{S}^p(\mathcal{T}_h)$ be the Galerkin-method approximation to $\partial_n^+ u$, defined by \eqref{eq:galerkin}.

(a) If $\Omega_-$ is an Ikawa-like union of convex obstacles then there exists $C>0$ such that, provided $k\geq k_0$ and $h k^{4/3}\log(2+k) \leq C$, it holds that
\beq \label{eq:melenk2}
\|v_{hp} - \partial_n^+u\|_{L^2(\Gamma)} \lesssim \log(2+k) \inf_{v\in \mathcal{S}^p(\mathcal{T}_h)} \|v - \partial_n^+u\|_{L^2(\Gamma)}.
\eeq

(b) If $\Omega_-$ is a piecewise smooth $(R_0,R_1)$ obstacle, then there exists $C>0$ such that, provided $k\geq k_0$ and $h k^{7/2}\log(2+k) \leq C$, it holds that
\beq \label{eq:melenk3}
\|v_{hp} - \partial_n^+u\|_{L^2(\Gamma)} \lesssim k^2 \inf_{v\in \mathcal{S}^p(\mathcal{T}_h)} \|v - \partial_n^+u\|_{L^2(\Gamma)}.
\eeq
The hidden constants in \eqref{eq:melenk2} and \eqref{eq:melenk3} are independent of $h$, $p$, and $k$.
\end{corollary}

\subsection{Outline of paper}

In \S\ref{sec:prelim} we establish notations and definitions and collect a few basic results that are used throughout the paper. In \S\ref{sec:resol} we prove Theorem \ref{thm:resol} (the resolvent estimates for $(R_0,R_1)$ obstacles). In \S\ref{sec:dtn} we prove Theorem \ref{thm:dtn} (bounds on the DtN map for $(R_0,R_1)$ obstacles), and deduce DtN bounds also for hyperbolic and elliptic trapping. In \S\ref{sec:infsup} we deduce bounds on the inf-sup constant for trapping confugurations, proving Corollary \ref{cor:infsup}. We consider applications to boundary integral equations in \S\ref{sec:ie}, proving Corollaries \ref{cor:CFIE} and \ref{cor:hversion}, and discussing the other issues summarised in \S\ref{sec:BIE}. We also, as an extension of the proof of the lower bound \eqref{eq:Ainv_bound_mainlower}, provide in \S\ref{rem:coercivity} a counterexample to the conjecture of Betcke and Spence \cite[Conjecture 6.2]{BeSp:11} that $\opA$ is coercive uniformly in $k$ for large $k$ whenever $\Oi$ is nontrapping. Table \ref{tab:Ainv} provides a useful summary of the results of this paper, and of the existing known sharpest bounds.

\section{Preliminaries} \label{sec:prelim}

\subsection{Morawetz/Rellich-type identities and associated results}

\ble[Morawetz-type identity]\label{le:morid1}
Let $v\in C^2(D)$ for some open set $D\subset \Rea^d$, $d\geq 2$.
Let $\cL v := (\Delta  + k^2) v$ with $k \in \mathbb{R}$.
Let $\bZ\in (C^1(D))^d$, $\beta \in C^1(D)$, and $\alpha\in C^2(D)$ (i.e. $\bZ$ is a vector and $\beta$ and $\alpha$ are scalars) and let all three be \emph{real}-valued.
Let
\beq\label{eq:multa}
\cZ v:= Z\cdot \gv - \ri k \beta v + \alpha v.
\eeq
Then, with the usual summation convention,
\begin{align}\nonumber 2 \Re \big(\overline{\cZ v } \,\cL v \big) = &\, \nabla \cdot \Big[ 2 \Re\big(\overline{\cZ v}\, \nabla v \big) + \big(k^2 \nvs - \ngvs\big) \bZ - \nabla\alpha \nvs\Big] + \big(2 \alpha-\nabla \cdot \bZ  \big)\big( k^2 \nvs-\ngvs\big) \\ &
-2 \Re \big(\partial_i Z_j \partial_i v \overline{\partial_j v}\big)
- 2 \Re\big( \ri k\, \overline{v}\,\nabla \beta \cdot \nabla v \big) + \Delta\alpha \nvs.
\label{eq:morid1}
\end{align}
\ele

Lemma \ref{le:morid1} can be proved by expanding the divergence on the right-hand side; see \cite[Proof of Lemma 2.1]{SpKaSm:15}. The identity \eqref{eq:morid1} was essentially introduced by Morawetz in \cite[\S I.2]{Mo:75}; see the bibliographic remarks in \cite[Remark 2.7]{SpKaSm:15}. Identities arising from the multiplier $\bZ\cdot\gu$ are often called Rellich-type, due to Rellich's use of the multiplier $x\cdot \gv$ in \cite{Re:40} and the multiplier $e_d \cdot \gu$ in \cite{Re:43} (see, e.g., the discussion in \cite[\S  5.3]{ChGrLaSp:12} and \cite[\S I.4]{MoSp:14}).

We now prove an integrated form of the identity \eqref{eq:morid1}; when we use this in the proof of Theorem \ref{thm:resol}, it turns out that we only need to consider constant $\beta$, and so we restrict attention to this case.

\begin{lemma}[Integrated form of the identity \eqref{eq:morid1}]\label{lem:morint}
Let $D \subset \Rea^d$ be a bounded Lipschitz domain with outward-pointing unit normal $\nu$, let $\gamma$ denote the trace operator, and $\partial_\nu$ the normal derivative operator.
If $\bZ\in (C^1(\overline{D}))^d$ and $\alpha\in C^2(\overline{D})$ are real-valued, $\beta\in\Rea$, and $v\in V(D)$, where
\beq\label{eq:V}
V(D):=\Big\{
v\in H^1(D):  \Delta v \in L^2(D), \gamma v \in H^1(\partial D), \partial_\nu v \in L^2(\partial D)
\Big\},
\eeq
and if $\cL v := (\Delta  + k^2) v$ with $k \in \mathbb{R}$, and $\cZ v$ is defined by \eqref{eq:multa}, then
\begin{align}\nonumber
&\int_D\left( 2 \Re \big(\overline{\cZ v } \,\cL v \big)
+2 \Re \big(\partial_i Z_j \partial_i v \overline{\partial_j v}\big)
- \big(2 \alpha-\nabla \cdot \bZ  \big)\big( k^2 \nvs-\ngvs\big)-\Delta \alpha \nvs\right)\rd x\\ \nonumber
&= \int_{\partial D}\left[ (Z\cdot \nu) \left( \left|\partial_\nu v\right|^2 - |\nT (\gamma v)|^2 + k^2 |\gamma v|^2 \right)\right.\\
&  \hspace{0.5cm} \left. + 2 \Re \left(
\big(Z\cdot \overline{\nT (\gamma v)} + \ri k \beta \overline{\gamma v} + \alpha\, \overline{\gamma v}\big)
\partial_\nu v
\right)-\pdiff{\alpha}{\nu}|\gamma v|^2\right]\rd s.
\label{eq:morid_int}
\end{align}
\end{lemma}

\

\begin{proof}[Proof of Lemma \ref{lem:morint}]
We first assume that $Z$, $\alpha$, and $\beta$ are as in the statement of the lemma, but
$v\in \cD(\overline{D}):=\{ U|_D : U \in C^\infty(\Rea^d)\}$.
Recall that the divergence theorem $\int_D \nabla \cdot F= \int_{\partial D} \gamma F \cdot \nu$
 is valid when $F\in H^1(D)$ by \cite[Theorems 3.29, 3.34, and 3.38]{Mc:00}.
Recall also that the product of an $H^1(D)$ function and a $C^1(\overline{D})$ function is in $H^1(D)$, and the usual product rule for differentiation holds. Thus
$F= 2 \Re\big(\overline{\cZ v}\, \nabla v \big) + \big(k^2 \nvs - \ngvs\big) \bZ- \nabla \alpha \nvs$
is in $H^1(D)$
and $\nabla\cdot F$ is given by the integrand on the left-hand side of \eqref{eq:morid_int}.
Furthermore,
\begin{align*}
\gamma F\cdot \nu =(Z\cdot \nu) \left( \left|\pdiff{v}{\nu}\right|^2 + k^2 \nvs - |\nT v|^2 \right) + 2 \Re \left(
\big(Z\cdot \overline{\nT v} + \ri k \beta \overline{v} + \alpha \overline{v}\big)
\pdiff{v}{\nu}
\right)- \pdiff{\alpha}{\nu}|\gamma v|^2
\end{align*}
on $\partial D$, where we have used the fact that $\gv = \nu (\partial v/\partial \nu) + \nT v$ on $\partial D$ for $v\in \cD(\overline{D})$; the identity \eqref{eq:morid_int} then follows from the divergence theorem.

The result for $v\in V(D)$ then follows from (i) the density of $\cD(\overline{D})$ in $V(D)$ \cite[Lemmas 2 and 3]{CoDa:98}
and (ii) the fact that \eqref{eq:morid_int} is continuous in $v$ with respect to the topology of $V(D)$.
\epf

\ble\textbf{\emph{(Morawetz-Ludwig identity, \cite[Equation 1.2]{MoLu:68})}}\label{le:ML}
 Let $v \in C^2(D)$ for some open $D\subset \Rea^d$, $d\geq 2$. Let $\cL v:=(\Delta +k^2)v$ and let
 \beq\label{Malpha}
\cM_\alpha v :=	 r\left(\pdiff{v}{r}-\ri k v + \frac{\alpha}{r}v\right),
\eeq
where $\alpha \in \Rea$ and $\partial v/\partial r=\bx\cdot \gv/r$. Then
\bal\nonumber
2\Re( \overline{\cM_\alpha v} \cL v) =  &\,\nabla \cdot \bigg[2\Re \left(\overline{\cM_{\alpha} v} \gv\right)+ \left(k^2\nvs - \ngvs \right)\bx\bigg] \\&+ \big(2\alpha -(d-1)\big)\big(k^2 \nvs - \ngvs\big) - \big(\ngvs -|\partial v/\partial r|^2\big)- \big|\partial v/\partial r-\ri k v\big|^2.
\label{eq:ml2d}
\end{align}
\ele

The Morawetz-Ludwig identity is a particular example of the identity \eqref{eq:morid1} with $\bZ=\bx$, $\beta=r,$ and $\alpha$ a constant, and some further manipulation of the non-divergence terms (using the fact that $\bx= \beta\nabla\beta$). For a proof, see \cite{MoLu:68}, \cite[Proof of Lemma 2.2]{SpChGrSm:11}, or \cite[Proof of Lemma 2.3]{SpKaSm:15}.

The Morawetz-Ludwig identity \eqref{eq:ml2d} has two key properties. With this identity rearranged and written as $\nabla\cdot\bQ(v)=P(v)$, the key properties are:
\ben
\item
If $u$ is a solution of $\cL u=0$ in $\Rea^d\setminus \overline{B_{R_0}}$, for some $R_0>0$, satisfying the Sommerfeld radiation condition \eqref{eq:src}, then, where $\GR:= \partial B_R$,
\beq\label{eq:N0}
\int_\GR \bQ(u) \cdot\widehat{\bx}\, \rd s
 \tendo \quad \tas R\tendi
\eeq
(independent of the value of $\alpha$ in the multiplier $\cM_\alpha u$); see \cite[Proof of Lemma 5]{MoLu:68}, \cite[Lemma 2.4]{SpChGrSm:11}.
\item
If $\cL u=0$ and $2\alpha= (d-1)$,  then
\beq
P(u) \geq 0.\label{eq:2ndkey}
\eeq
\een
The two properties of the Morawetz-Ludwig identity above mean that if the multiplier that we use on the operator $\cL$ is equal to $\cM_{(d-1)/2}$ outside a large ball, then there is no contribution from infinity. A convenient way to encode this information is the following lemma due to Chandler-Wilde and Monk \cite[Lemma 2.1]{ChMo:08}.

\ble[Inequality on $\GammaR$ used to deal with the contribution from infinity]\label{lem:2.1}
Let $u$ be a solution of the homogeneous Helmholtz equation in $\Rea^d\setminus \overline{B_{R_0}}$, $d=2,3$, for some $R_0>0$, satisfying the Sommerfeld radiation condition \eqref{eq:src}.
Let $\alpha\in \Rea$ with
$2\alpha\geq d-1$. Then, for $R>R_0$,
\beq\label{eq:2.1}
R\int_{\GammaR} \left( \left|\pdiff{u}{r}\right|^2 - |\nabla_S u|^2 + k^2 |u|^2\right) \rd s  - 2 k R\, \Im \int_{\GammaR} \overline{u} \pdiff{u}{r}\,\rd s + 2\alpha\Re \int_{\GammaR}\overline{u}\pdiff{u}{r} \, \rd s\leq 0,
\eeq
where $\nabla_S$ is the surface gradient on $\GR=\partial B_R$.
\ele

We have purposely denoted the constant in \eqref{eq:2.1} by $\alpha$ to emphasise the fact that the left-hand side of \eqref{eq:2.1} is
$\int_\GR \bQ(u) \cdot\widehat{\bx}\,\rd s$ with $\bQ(u)$ arising from
the multiplier $\cM_\alpha u = \bx\cdot \gu - \ri kr u +\alpha u$. We will see below that the Morawetz-Ludwig identity proves the inequality \eqref{eq:2.1} when $2\alpha=d-1$, but it will be slightly more convenient to have this result for $2\alpha\geq d-1$. For the proof of this we need the following, slightly simpler, inequality on $\GR$.

\ble\label{lem:2.2}
Let $u$ be a solution of the homogeneous Helmholtz equation in $\Rea^d\setminus \overline{B_{R_0}}$, $d=2,3$, for some $R_0>0$, satisfying the Sommerfeld radiation condition \eqref{eq:src}. Then, for $R>R_0$,
\beq\label{eq:2.12}
\Re \int_{\GammaR} \bar{u}\pdiff{u}{r} \, \rd s \leq 0.
\eeq
\ele

\bpf[Proof of Lemma \ref{lem:2.2}]
This result is proved in \cite[Theorem 2.6.4, p.97]{Ne:01} or \cite[Lemma 2.1]{ChMo:08}
using the explicit expression for the solution of the Helmholtz equation in the exterior of a ball (i.e. an expansion in either trigonometric polynomials, for $d=2$, or spherical harmonics, for $d=3$, with coefficients given in terms of Bessel and Hankel functions) and then proving bounds on particular combinations of Bessel and Hankel functions.
\epf

\

\bpf[Proof of Lemma \ref{lem:2.1}]
This result is proved in \cite[Lemma 2.1]{ChMo:08} by using the explicit expression for the solution of the Helmholtz equation in the exterior of a ball, as in the proof of Lemma   \ref{lem:2.2}, and proving monotonicity properties of combinations of Bessel and Hankel functions.
We provide here an alternative, shorter, proof via the Morawetz-Ludwig identity
but note that in fact, \cite[Lemma 2.1]{ChMo:08} is slightly stronger result than Lemma \ref{lem:2.1} when $d=3$, showing that \eqref{eq:2.1} holds whenever $2\alpha\geq 1$.

By the inequality \eqref{eq:2.12}, it is sufficient to prove \eqref{eq:2.1} with $2\alpha=d-1$.
We now integrate \eqref{eq:ml2d} with $v=u$ and $2\alpha =d-1$ over $B_{R_1}\setminus B_R$, use the divergence theorem, and then let $R_1\tendi$
(note that using the divergence theorem is allowed since $u$ is $C^\infty$ by elliptic regularity).
The first key property of the Morawetz-Ludwig identity stated above (as \eqref{eq:N0}) implies that the surface integral on $\vert \bx \vert ={R_1}$ tends to zero as $R_1\tendi$ \cite[Lemma 2.4]{SpChGrSm:11}. Then, using the decomposition $\gv = \nabla_S v + \widehat{\bx}\partial_r v$ on the integral over $\Gamma_R$, we obtain that
\begin{align*}
\int_\GR \bQ(u) \cdot\widehat{\bx} \, \rd s&=
\int_{\GammaR} R\left( \left|\pdiff{u}{r}\right|^2  - |\nabla_S u|^2+ k^2 |u|^2\right) \, \rd s  \\
&\hspace{3cm}- 2 k R\, \Im \int_{\GammaR} \bar{u} \pdiff{u}{r} \, \rd s+ (d-1)\Re \int_{\GammaR}\bar{u}\pdiff{u}{r} \, \rd s \nonumber \\
& =  -\int_{\Rea^d\setminus B_R}\left(\big(\ngus -| \partial u/\partial r|^2\big)+ \left|
 \partial u/\partial r - \ri k u
 \right|^2\right)\rd x\leq 0\label{eq:2.1b}
 \end{align*}
(where this last inequality is the second key property \eqref{eq:2ndkey} above); i.e.\ we have established \eqref{eq:2.1} with $2\alpha=d-1$ and we are done.
 \epf

\

The inequality \eqref{eq:2.12} combined with Green's identity (i.e.~pairing $\cL v$ with $v$) has the following simple consequence, which we use later.

\ble\label{lem:Green}
Let $f\in L^2(\Omega_+)$ have compact support in $\overline{\Omega_+}$, and let $u\in H^1_{\mathrm{loc}}(\Omega_+)$ be a solution to the Helmholtz equation $\Delta u + k^2 u = -f$ in $\Omega_+$ that satisfies the Sommerfeld radiation condition \eqref{eq:src} and the boundary condition $\gamma_+ u = 0$. For any $R>R_\Gamma$ such that $\supp f\subset B_R$,
\beq\label{eq:Green_ineq}
\int_{\Omega_R}\ngus \, \rd x\leq  k^2 \int_{\Omega_R}\nus \, \rd x+ \Re\int_{\Omega_R} f \bar{u}\, .
\eeq
\ele

\bpf
By multiplying $\cL u =-f$ by $\bar{u}$, integrating over $\Omega_R$, and applying the divergence theorem, we have
\beqs
\int_{\Omega_R}\ngus \, \rd x- k^2 \int_{\Omega_R}\nus \, \rd x-\int_{\Omega_R} f \bar{u}\, \rd x = \int_{\GR}\bar{u}\pdiff{u}{r} \, \rd s.
\eeqs
The result then follows by taking the real part and using \eqref{eq:2.12}.
\epf

\subsection{A Poincar\'e-Friedrichs-type inequality} \label{sec:poincare}

The following Poincar\'e-Friedrichs-type inequality will play a key role in the proof of Theorem \ref{thm:resol} (see Lemma \ref{lem:E3new} below).

\begin{lemma} \label{lem:Fried}
For $R>0$ and $v\in H^1(\R^d)$ it holds that
\beq\label{eq:Fried}
\int_{B_{2R}}|v|^2 \, \rd x \leq  8\int_{B_{\sqrt{13} R}\setminus B_{2R}} |v|^2\, \rd x   + 4 R^2 \int_{B_{\sqrt{13} R}} |\partial_d v|^2 \,\rd x.
\eeq
\end{lemma}
\begin{proof}
Suppose that $\phi\in C_0^\infty(\R)$ and $h, H> 0$. Then, for $0\leq t\leq h\leq s\leq h+H$,
$$
\phi(t) = \phi(s)-\int_t^s \phi^\prime(r)\, \rd r
$$
so that, by the Cauchy-Schwarz inequality and the inequality
\beq \label{eq:Cauchy}
2ab \leq \epsilon a^2 + b^2/\epsilon\quad\tfa \, a,b,\epsilon>0,
\eeq
we have
$$
|\phi(t)|^2 \leq (1+\epsilon)|\phi(s)|^2 + (1+\epsilon^{-1})(s-t)\int_t^s|\phi^\prime(r)|^2 \, \rd r.
$$
Hence, for $0\leq h\leq s\leq h+H$,
\begin{eqnarray*}
\int_0^h |\phi(t)|^2 \, \rd t &\leq &(1+\epsilon)h|\phi(s)|^2 + (1+\epsilon^{-1})\int_0^h \left\{(s-t)\int_0^{h+H}|\phi^\prime(r)|^2 \, \rd r \right\}\rd t\\
& = & (1+\epsilon)h|\phi(s)|^2 + \frac{1}{2}(1+\epsilon^{-1})h(2s-h)\int_0^{h+H} |\phi^\prime(r)|^2\,\rd r,
\end{eqnarray*}
so that, integrating with respect to $s$ from $h$ to $h+H$ and dividing by $H$,
\begin{eqnarray} \label{eq:1DFried}
\int_0^h |\phi(t)|^2\, \rd t & \leq & \frac{(1+\epsilon)h}{H}\int_h^{h+H}|\phi(s)|^2\, \rd s + \frac{(1+\epsilon^{-1})h(h+H)}{2}\int_0^{h+H} |\phi^\prime(r)|^2 \,\rd r.
\end{eqnarray}

For $h_1<h_2$ and $A>0$ let $U(h_1,h_2,A):= \{x=(\widetilde x, x_d)\in \R^d:h_1<x_d<h_2, |\widetilde x|<A\}$. Then, for $v\in C_0^\infty(\R^d)$ in the first instance, and then by density for all $v\in H^1(\R^d)$, it follows from \eqref{eq:1DFried} with $\epsilon=3$ that, for $h,H,A>0$,
\begin{eqnarray} \nonumber
\int_{U(0,h,A)} |v|^2\, \rd x & \leq & \frac{4h}{H}\int_{U(h,h+H,A)}|v|^2\, \rd x + \frac{2h(h+H)}{3}\int_{U(0,h+H,A)}|\partial_d v|^2 \,\rd x.
\end{eqnarray}
Similarly,
\begin{eqnarray} \nonumber
\int_{U(-h,0,A)} |v|^2\, \rd x & \leq & \frac{4h}{H}\int_{U(-h-H,-h,A)}|v|^2\, \rd x + \frac{2h(h+H)}{3}\int_{U(-h-H,0,A)}|\partial_d v|^2 \,\rd x,
\end{eqnarray}
for $v\in H^1(\R^d)$. Thus, for $v\in H^1(\R^d)$  it holds for $h>0$ that
\begin{eqnarray} \nonumber
\int_{B_{h}}|v|^2 \, \rd x &\leq &\int_{U(-h,h,h)} |v|^2\, \rd x\\ \nonumber
& \leq & \frac{4h}{H}\left\{\int_{U(-h-H,-h,h)}|v|^2\, \rd x + \int_{U(h,h+H,h)}|v|^2\, \rd x\right\}\\ \nonumber
 & & \hspace{5ex} +  \frac{2h(h+H)}{3}\int_{U(-h-H,h+H,h)}|\partial_d v|^2 \,\rd x\\ \nonumber
 & \leq & \frac{4h}{H}\int_{B(h, \sqrt{h^2+(h+H)^2})}|v|^2\, \rd x
  +  \frac{2h(h+H)}{3}\int_{B(0, \sqrt{h^2+(h+H)^2})} |\partial_d v|^2 \,\rd x,
\end{eqnarray}
where for $0\leq h_1\leq h_2$, $B(h_1,h_2) := B_{h_2}\setminus B_{h_1}$. Applying this bound with $h=2R$ and $H=R$ we obtain the required result.
\end{proof}

\begin{corollary}\label{cor:Fried}
If $v\in H^1(\Omega_R)$ with $\gamma_+v =0$ on $\Gamma$, and $R\geq \sqrt{13} R_0$, then
\beq\label{eq:Fried2}
\int_{\Omega_{2R_0}}|v|^2 \, \rd x \leq  8\int_{\Omega_{R}\setminus \Omega_{2R_0}} |v|^2\, \rd x   + 4 R_0^2 \int_{\Omega_{R}} |\partial_d v|^2 \,\rd x.
\eeq
\end{corollary}

\bpf
This follows from Lemma \ref{lem:Fried} since, given $v\in H^1(\Omega_R)$ with $\gamma_+v =0$ on $\Gamma$, we can extend the definition of $v$ to $\R^d$ so that $v\in H^1(\R^d)$ and $v=0$ in $\Omega_-$.
\epf

\subsection{Boundary Sobolev spaces and interpolation} \label{sec:interp}
We use boundary Sobolev spaces $H^s(\Gamma)$ defined in the usual way (see, e.g., \cite[Pages 98 and 99]{Mc:00}), and denote by $H^s_k(\Gamma)$ the space $H^s(\Gamma)$ equipped with a wavenumber dependent norm $\|\cdot\|_{H_k^s(\Gamma)}$. Precisely, we equip $H^0(\Gamma)=H^0_k(\Gamma)=L^2(\Gamma)$ with the $L^2(\Gamma)$ norm. We define $\|\cdot\|_{H^s(\Gamma)}$ and $\|\cdot\|_{H_k^s(\Gamma)}$ for $s=1$ by
\beq \label{eq:Ganorm}
\|\phi\|_{H^1(\Gamma)}^2 = \|\nabla_S \phi\|^2_{L^2(\Gamma)} +\|\phi\|^2_{L^2(\Gamma)} \quad \mbox{and} \quad \|\phi\|_{H^1_k(\Gamma)}^2 = \|\nabla_S \phi\|^2_{L^2(\Gamma)} +k^2\|\phi\|^2_{L^2(\Gamma)},
\eeq
and for $0<s<1$ by interpolation, choosing the specific norm given by the complex interpolation method (equivalently, by real methods of interpolation appropriately defined and normalised; see \cite{McC:92}, \cite[Remark 3.6]{ChHeMo:15}). We then define the norms on $H^s(\Gamma)$ and $H^s_k(\Gamma)$ for $-1\leq s<0$ by duality,
\beq \label{eq:Ganorm3}
\|\phi\|_{H^s(\Gamma)} := \sup_{0\neq\psi\in H^{-s}(\Gamma)}\, \frac{|\langle\phi,\psi\rangle_\Gamma|}{\|\psi\|_{H^{-s}(\Gamma)}} \quad \mbox{and} \quad \|\phi\|_{H_k^s(\Gamma)} := \sup_{0\neq\psi\in H^{-s}(\Gamma)}\, \frac{|\langle\phi,\psi\rangle_\Gamma|}{\|\psi\|_{H_k^{-s}(\Gamma)}},
\eeq
for $\phi\in H^s(\Gamma)$,
where $\langle\phi,\psi\rangle_\Gamma$ denotes the standard duality pairing that reduces to $(\phi,\psi)_\Gamma$, the inner product on $L^2(\Gamma)$, when $\phi\in L^2(\Gamma)$. In the terminology of \cite[Remark 3.8]{ChHeMo:15}, with the norms we have selected, $\{H^s(\Gamma):-1\leq s\leq 1\}$ and $\{H^s_k(\Gamma):-1\leq s\leq 1\}$ are {\em exact interpolation scales}, so that, if $B:H_k^{s_j}(\Gamma)\to H_k^{t_j}(\Gamma)$ is a bounded linear operator and
$$
\|B\|_{H_k^{s_j}(\Gamma)\to H_k^{t_j}(\Gamma)}\leq C_j, \quad \mbox{for }j=1,2,
$$
with $s_j,t_j\in [-1,1]$, then $B:H_k^s(\Gamma)\to H_k^t(\Gamma)$ and
\beq \label{eq:interp}
\|B\|_{H_k^{s}(\Gamma)\to H_k^{t}(\Gamma)}\leq C_1^{1-\theta} C_2^\theta,\quad \mbox{for } s=\theta s_1+(1-\theta)s_2 \mbox{ and }t=\theta t_1+(1-\theta)t_2 \mbox{ with } 0<\theta<1.
\eeq
Moreover (by definition) $H^{-s}(\Gamma)$ is an isometric realisation of $(H^s(\Gamma))^\prime$, the dual space of $H^s(\Gamma)$, for $-1\leq s\leq 1$, so that, if $A:H_k^{s}(\Gamma)\to H_k^{t}(\Gamma)$ is bounded and $B$ is the adjoint of $A$ with respect to the $L^2(\Gamma)$ inner product, or with respect to the {\em real} inner product $(\cdot,\cdot)_\Gamma^r$, defined by $(\phi,\psi)_\Gamma^r=\int_\Gamma \phi\psi \rd s$, then $B:H^{-t}_k(\Gamma)\to H_k^{-s}(\Gamma)$ is bounded and
\begin{equation} \label{eq:adj}
\|B\|_{H_k^{-t}(\Gamma)\to H_k^{-s}(\Gamma)}=\|A\|_{H_k^{s}(\Gamma)\to H_k^{t}(\Gamma)}.
\end{equation}
Combining these observations, if $A:H_k^{s}(\Gamma)\to H_k^{t}(\Gamma)$ is bounded and self-adjoint, or is self-adjoint with respect to the real inner product, meaning that $(A\phi,\psi)_\Gamma^r = (\phi,A\psi)_\Gamma^r$, for $\phi,\psi\in H^1(\Gamma)$, then
\begin{equation} \label{eq:adj2}
\|A\|_{H_k^{\sigma}(\Gamma)\to H_k^{\tau}(\Gamma)}\leq\|A\|_{H_k^{s}(\Gamma)\to H_k^{t}(\Gamma)},
\end{equation}
for $\sigma = \theta s-(1-\theta)t$, $\tau = \theta t-(1-\theta)s$, and $0\leq \theta \leq 1$.

\section{Proof of Theorem \ref{thm:resol} on resolvent estimates} \label{sec:resol}

\subsection{The ideas behind the proof}\label{sec:idea}

The proof is based on the Morawetz-type identity \eqref{eq:morid1}.
Recall that in \cite{Mo:75}, \cite[\S4]{MoRaSt:77}, Morawetz and co-workers showed that if there exists a vector field $Z(x)$, $R>R_\Gamma= \max_{x\in \Gamma}|x|$, and $c_1>0$ such that
\beq\label{eq:Mor1}
\bZ(\bx) = \bx \quad\text{ in a neighbourhood of } \GR  = \partial B_R,
\eeq
\beq\label{eq:Mor2}
\Re \big( \partial_i Z_j(\bx) \xi_i \overline{\xi_j}\big) \geq 0 \tfa \xi \in\Com^d \tand \bx \in \OR, \quad\tand \bZ(\bx)\cdot\bn(\bx) \geq c_1 \quad \tfa\bx \in \Gamma,
\eeq
then \eqref{eq:morid1} can be used to prove the resolvent estimate \eqref{eq:nt_estimate} (\cite[\S4]{MoRaSt:77} then constructed such a $Z$ for a class of obstacles slightly more restrictive than nontrapping). Implicit in \cite{Mo:75} is the fact that one can replace the two conditions \eqref{eq:Mor2} with
\beq\label{eq:Mor2a}
\Re \big( \partial_i Z_j(\bx) \xi_i \overline{\xi_j}\big) \geq c_2|\xi|^2 \tfa \xi \in\Com^d \tand \bx \in \OR, \quad\tand \bZ(\bx)\cdot\bn(\bx) \geq 0 \quad \tfa\bx \in \Gamma,
\eeq
for some $c_2>0$; i.e., one needs strict positivity \emph{either} in $\OR$ \emph{or} on $\Gamma$. Note that $Z(x)=x$ satisfies this second set of conditions, implying that the resolvent estimate holds for $\Oi$ that are star-shaped (see also \cite{MoLu:68, ChMo:08}).

We cannot expect to satisfy one of these sets of conditions on $Z$ (either \eqref{eq:Mor1} and \eqref{eq:Mor2} or \eqref{eq:Mor1} and \eqref{eq:Mor2a}) for  every $(R_0,R_1)$ obstacle, since we know the nontrapping resolvent estimate \eqref{eq:nt_estimate} does not hold for $(R_0,R_1,a)$ parallel trapping obstacles. The $Z$ that we use in our arguments is the one in the definition of $(R_0,R_1)$ obstacles, namely \eqref{eq:Zdeff}. By the definition of $(R_0,R_1)$ obstacles, we have $Z(x)\cdot n(x)\geq 0$ for all $x\in \Gamma$, but now, for $r<R_0$ at least,
$\Re \big( \partial_i Z_j(\bx) \xi_i \overline{\xi_j}\big) = |\xi_d|^2$, which is only positive semi-definite. (Note that the vector field $e_d x_d$ is often used in arguments involving Rellich/Morawetz-type identities in rough surface scattering; see \cite{Re:43,ZhCh:98,ChMo:05,LeRi:09}, \cite[\S  5.3]{ChGrLaSp:12}, \cite[\S I.4]{MoSp:14}, and the references therein.)

We apply the Morawetz-type identity \eqref{eq:morid1} in $\OR$ with $Z$ given by \eqref{eq:Zdeff} in terms of a function $\chi$ that satisfies the constraints of Definition \ref{def:trapb}, $\beta=R$, and $\alpha$ defined by \eqref{eq:alpha_def} below (the rationale for this choice of $\alpha$ is explained below \eqref{eq:Zqf}). Using Lemma \ref{lem:2.1} to deal with the contribution at infinity, we find in Lemma \ref{lem:E0} below that
\begin{align} \nonumber
 &\int_{\Omega_R} \Big(2|\partial_d u|^2\big(1-\chi(r)\big) + |\nabla u|^2(2-q)\chi(r) +qk^2|u|^2\chi(r)+ 2r|\partial_r u|^2\chi^\prime(r) \Big)\, \rd x \\ \nonumber
 & - 2\Re\int_{\Omega_R} x_d\partial_d \bar u\partial_r u\chi^\prime(r)\, \rd x \\  \label{eq:rellich4}
&  \hspace{0.5cm} \leq -2kR\Im \int_{\Omega_R} f\bar u\, \rd x + \Re\int_{\Omega_R} f\Big(2 x_d\partial_d \bar u \big(1-\chi(r)\big) + 2r\partial_r \bar u\chi(r) +2\alpha \bar u\Big)\, \rd x + \int_{\Omega_R}  \Delta \alpha \nus\, \rd x,
\end{align}
for any $q\in [0,1]$. We see that in the ``trapping region", namely when $\chi=0$, we only have control of $|\partial_d u|^2$, but in the ``nontrapping region" (in $\supp(\chi)$) we have control of $|\gu|^2 + k^2 |u|^2$ (as expected, since here $Z=x$).

We then proceed via a series of lemmas. In Lemma \ref{lem:E1} we get rid of the sign-indefinite ``cross" term on the second line of \eqref{eq:rellich4}.
In Lemma \ref{lem:E3new} we use the Poincar\'e-Friedrichs-type inequality of Corollary \ref{cor:Fried}, and then Lemma \ref{lem:Green}, to put first $|u|^2$ and then $|\gu|^2$ back in the trapping region -- here is the place where we lose powers of $k$ compared to the nontrapping estimate, since, in the Poincar\'e-Friedrichs-type inequality, $|\partial_d u|^2$ is bounded below by $|u|^2$ without a corresponding factor of $k$.

Finally, in Lemmas \ref{lem:E4new} and \ref{lem:E5new} we  obtain bounds that imply the resolvent estimate \eqref{eq:resol}. In particular, in Lemma \ref{lem:E5new} we bound the term involving the Laplacian of $\alpha$,  using that, by our assumptions on $\chi$ back in Definition \ref{def:trapb}, $\Delta \alpha$ is continuous and vanishes for $r\leq R_0$.

As discussed in \S\ref{sec:discuss}, the analogue of the resolvent estimate \eqref{eq:resol} in the case of rough surface scattering was proved in \cite[Theorem 4.1]{ChMo:05}; the proof of this estimate uses $Z(x)= e_d x_d$, along with the analogue of Lemma \ref{lem:2.1} in this case \cite[Lemma 2.2]{ChMo:05}, avoiding the subtleties of transitioning between the vector fields $e_d x_d$ and $x$ that we encounter here.

After Theorem \ref{thm:resol} we discussed how quasimodes can be constructed that indicate that the power on the right-hand side of the bound \eqref{eq:resol4} should be reduced from $k^2$ to $k$ to obtain a sharp bound. Choosing $\beta$ in the multiplier $\cZ$ \eqref{eq:multa} to be zero when $\chi=0$ has the potential to produce this new bound, but $\beta$ cannot then be brought up to equal $R$ on $\Gamma_R$ (so that one can use Lemma \ref{lem:2.1} to deal with the contribution from infinity) whilst keeping the resulting cross term involving $\nabla \beta$ in \eqref{eq:morid1} under control.

\subsection{Lemmas \ref{lem:E0}--\ref{lem:E5new}, their proofs, and the proof of Theorem \ref{thm:resol}}
Throughout this subsection $\chi$ will be any function such that $\Oi$ is an $(R_0,R_1)$ obstacle in the sense of Definition \ref{def:trapb}, so that, in particular,
\begin{equation} \label{eq:chimax}
c_\chi := \sup_{r>0} \left(r\chi^\prime(r)\right) < 4.
\end{equation}
The vector field $Z$ will be defined in terms of $\chi$ by \eqref{eq:Zdeff}, and we will take
\begin{equation} \label{eq:Rlower}
R>\max(R_\Gamma,R_1,\sqrt{13}\,R_0)
\end{equation}
so that, respectively, $\overline{\Oi}\subset \Omega_R$, $\chi=1$ in a neighbourhood of $\partial B_R$, and Corollary \ref{cor:Fried} applies.

\begin{lemma}\label{lem:E0}
Let $\Oi$ be an $(R_0,R_1)$ obstacle,  and let
\beq\label{eq:alpha_def}
2\alpha := \nabla \cdot Z -q\,\chi(r),
\eeq
for some $q\in[0,1]$.
If $u$ is the solution of the exterior Dirichlet problem described in Theorem \ref{thm:resol} and $R$ is large enough so that $\supp(f)\subset \Omega_R$, then \eqref{eq:rellich4} holds.
\end{lemma}
\bpf
The regularity result of Ne\v{c}as \cite[\S5.1.2]{Ne:67}, \cite[Theorem 4.24(ii)]{Mc:00} (stated for $u$ the solution of the exterior Dirichlet problem in \cite[Lemma 3.5]{Sp:14}) implies that $u\in V(\Omega_R)$ defined by \eqref{eq:V}.
We use the integrated Morawetz identity \eqref{eq:morid_int} with $D=\Omega_R$,
 $Z$ the vector field in \eqref{eq:Zdeff} (observe that $Z\in(C^3(\R^d))^d$ by Definition \ref{def:trapb}),
$\beta=R$, and $\alpha\in C^2(\overline{\Omega_R})$. We fix $\alpha$  as given in \eqref{eq:alpha_def} later, but assume at this stage that $\alpha$ is constant in a neighbourhood of $\Gamma_R$ and  $2\alpha\geq d-1$ on $\Gamma_R:=\partial B_R$.
Using the fact that $\gamma_+ u=0$ on $\Gamma$ and $Z=x$ on $\Gamma_R$, we obtain
\begin{align}\nonumber
&\int_{\Omega_R} \Big[-2 \Re \big(\overline{\cZ u } \,f \big)
+2 \Re \big(\partial_i Z_j \partial_i u \overline{\partial_j u}\big)
- \big(2 \alpha-\nabla \cdot \bZ  \big)\big( k^2 \nus-\ngus\big)
-\Delta \alpha \nus
\Big]\rd x
\\&+ \int_\Gamma (Z\cdot n)\left|\partial_n u\right|^2 \rd s
= \int_{\GR} R\left( \left|\pdiff{u}{r}\right|^2 - |\nT u|^2 + k^2 \nus \right) \rd s+ 2 \Re\int_{\GR} \left(
\big(\ri k R\overline{u} + \alpha \overline{u}\big)
\pdiff{u}{r}
\right)\rd s.
\label{eq:morid_int2}
\end{align}
Since  $2\alpha\geq d-1$ on $\GR$, the right-hand side of \eqref{eq:morid_int2} is non-positive by Lemma \ref{lem:2.1}.
Then, since $Z\cdot n\geq 0$ almost everywhere on $\Gamma$ from the definition of an $(R_0,R_1)$ obstacle (Definition \ref{def:trapb}), we have
\begin{align}\nonumber
&
\int_{\Omega_R} \Big[2 \Re \big(\partial_i Z_j \partial_i u \overline{\partial_j u}\big)
- \big(2 \alpha-\nabla \cdot \bZ  \big)\big( k^2 \nus-\ngus\big)-\Delta \alpha \nus\Big]\rd x\\
&\hspace{2cm}\leq 2\Re \int_{\Omega_R}  \big(Z\cdot \overline{\gu} + \ri k R \overline{u} + \alpha \overline{u}\big) \,f\, \rd x.
\label{eq:morid_int3}
\end{align}
Simple calculations imply that (with the summation convention for the indices $i$ and $j$ but not $d$)
\begin{align}
\label{eq:Zdnv}
Z\cdot \nabla u  =  x_d\partial_d u \big(1-\chi(r)\big) + r\partial_r u\chi(r)
\end{align}
and
\begin{align} \label{eq:Zqf}
\partial_iZ_j\partial_i u\overline{\partial_j u}&= |\partial_d u|^2\big(1-\chi(r)\big) + |\nabla u|^2\chi(r) + \left(r|\partial_r u|^2-x_d\partial_d \bar u\partial_r u\right)\chi^\prime(r).
\end{align}
We now choose $\alpha \in C^2(\overline{\Omega_R})$ as in \eqref{eq:alpha_def},
the rationale behind this choice that: (i) we want $2\alpha$ to
be constant in a neighbourhood of $\Gamma_R$ and indeed satisfy
 $2\alpha \geq d-1$ there, allowing the application above of Lemma \ref{lem:2.1}; and (ii) we want $2\alpha=\nabla\cdot Z$ in the trapping region to kill the sign-indefinite combination $k^2 \nus-|\gu|^2$ in \eqref{eq:morid_int3}, and leave $2 \Re (\partial_i Z_j \partial_i u \overline{\partial_j u})$ as the only volume term in this region.
Using \eqref{eq:Zdnv} and \eqref{eq:Zqf} in \eqref{eq:morid_int3}, we find
\begin{align} \nonumber
&2\Re\int_{\Omega_R} \Big(|\partial_d u|^2\big(1-\chi(r)\big) + |\nabla u|^2\chi(r) + \left(r|\partial_r u|^2- x_d\partial_d \bar u\partial_r u\right)\chi^\prime(r) \, -\Delta \alpha \nus\Big)\, \rd x \\\ \nonumber
&-q\int_{\Omega_R}\chi(r) (|\nabla u|^2-k^2|u|^2)\, \rd x\\ \nonumber
&\hspace{2cm} \leq  -2kR\Im \int_{\Omega_R} f\bar u\, \rd x + \Re\int_{\Omega_R} f\Big(2 x_d\partial_d \bar u \big(1-\chi(r)\big)+ 2r\partial_r \bar u\chi(r) +2\alpha \bar u\Big) \rd x
\end{align}
which rearranges to the result \eqref{eq:rellich4}.
\epf

\ble\label{lem:E1}
Let $\Oi$ be an $(R_0,R_1)$ obstacle.
 If $v\in H^1(\Omega_R)$, then
\begin{align} \nonumber
& \int_{\Omega_R} \Big[2|\partial_d v|^2\big(1-\chi(r)\big)+ |\nabla v|^2(2-q)\chi(r) + 2r|\partial_r v|^2\chi^\prime(r)\Big]\, \rd x
- 2\Re\int_{\Omega_R} x_d\partial_d \bar v\partial_r v\chi^\prime(r)\, \rd x\\ \label{eq:lemE1}
& \geq  \left(2-q-\mu- \frac{c_\chi}{2}\right) \int_{\Omega_R} |\partial_d v|^2\, \rd x + \mu\int_{\Omega_R} |\nabla v|^2\chi(r) \, \rd x,
\end{align}
for all $q, \mu>0$ with $0<q+\mu \leq 2$.
\ele

\bpf
By the
inequality \eqref{eq:Cauchy},
it follows that
$$
\left|2\Re\int_{\Omega_R} x_d\partial_d \bar v\partial_r v\chi^\prime(r) \rd x\right|\leq \varepsilon^{-1} \int_{\Omega_R} r|\partial_d v|^2\chi^\prime(r)\, \rd x + \varepsilon\int_{\Omega_R} r|\partial_r v|^2\chi^\prime(r) \, \rd x,
$$
for all $\varepsilon>0$. Using this last inequality with $\varepsilon =2$, along with the definition of $c_\chi$ \eqref{eq:chimax}, we have that the left-hand side of \eqref{eq:lemE1} is
\beqs
\geq \int_{\Omega_R} \Big[\left(2(1-\chi(r)) -\frac{c_\chi}{2}\right)|\partial_d v|^2 + |\nabla v|^2(2-q)\chi(r)\Big]\, \rd x.
\eeqs

Further, since $2-q-\mu\geq 0$ and $|\nabla v|^2 \geq |\partial_dv|^2$,
\begin{align*}
2|\partial_d v|^2 (1-\chi) + (2-q)|\nabla v|^2 \chi &= 2|\partial_d v|^2 (1-\chi) + (2-q-\mu)|\nabla v|^2 \chi  + \mu |\nabla v|^2 \chi  \\
&\geq (2- q -\mu) |\partial_d v|^2 + \mu |\nabla v|^2 \chi,
\end{align*}
and the result \eqref{eq:lemE1} follows.
\epf

\ble\label{lem:E3new}
Let $\Oi$ be an $(R_0,R_1)$ obstacle and $p,q>0$. If $v\in H^1(\Omega_R)$ with $\gamma_+ v = 0$, then
\begin{align*}
p\int_{\Omega_R}|\partial_d v|^2\, \rd x + \frac{qk^2}{2}\int_{\Omega_R}|v|^2\chi(r)\, \rd x \geq \frac{p}{4R_0^2}\int_{\Omega_R}|v|^2\, \rd x,
\end{align*}
if $k$ is large enough so that
\begin{equation} \label{eq:klarge}
k^2R_0^2  \geq \frac{9p}{2q\,\chi(2R_0)}.
\end{equation}
\ele
\bpf Since $\chi$ satisfies the constraints (i) and (ii) of Definition \ref{def:trapb}, $\chi(r)\geq \chi(2R_0)$ for $r\geq 2R_0$, so that
$$
\frac{qk^2}{18}\int_{\Omega_R}|v|^2\chi(r)\, \rd x \geq \frac{qk^2\, \chi(2R_0)}{18}\int_{\Omega_R\setminus \Omega_{2R_0}}|v|^2\, \rd x.
$$
Using Corollary \ref{cor:Fried} (which applies since $R\geq \sqrt{13}R_0$ by our assumption \eqref{eq:Rlower}),
and provided that \eqref{eq:klarge} holds, we have
\begin{eqnarray*}
& & \hspace{-2cm} p\int_{\Omega_R}|\partial_d v|^2\, \rd x + \frac{8qk^2}{18}\int_{\Omega_R}|v|^2\chi(r)\, \rd x\\
& \geq & \frac{p}{4R_0^2}\left(4R_0^2 \int_{\Omega_R}|\partial_d v|^2\, \rd x + 8 \int_{\Omega_R\setminus \Omega_{2R_0}}|v|^2\, \rd x\right) \geq \frac{p}{4R_0^2}\int_{\Omega_{2R_0}}|v|^2\, \rd x,
\end{eqnarray*}
so that
\begin{eqnarray*}
& & \hspace{-2cm} p\int_{\Omega_R}|\partial_d v|^2\, \rd x + \frac{qk^2}{2}\int_{\Omega_R}|v|^2\chi(r)\, \rd x\\
& \geq & \frac{p}{4R_0^2}\int_{\Omega_{2R_0}}|v|^2\, \rd x + \frac{qk^2}{18}\int_{\Omega_R}|v|^2\chi(r)\, \rd x\\
& \geq & \frac{p}{4R_0^2}\int_{\Omega_{2R_0}}|v|^2\, \rd x + \frac{qk^2\, \chi(2R_0)}{18}\int_{\Omega_R\setminus \Omega_{2R_0}}|v|^2\, \rd x \geq \frac{p}{4R_0^2}\int_{\Omega_{R}}|v|^2\, \rd x.
\end{eqnarray*}
\epf

\ble\label{lem:E4new}
Let $\Oi$ be an $(R_0,R_1)$ obstacle and $\alpha$ be defined by \eqref{eq:alpha_def}, for some  $q\in(0,1]$.
If $u$ is the solution of the exterior Dirichlet problem described in Theorem \ref{thm:resol}, $p>0$,  $R$ is large enough so that $\supp(f)\subset \Omega_R$, and $k$ is large enough so that \eqref{eq:klarge} holds, then
\begin{eqnarray}\nonumber
& &\hspace{-2cm} \frac{p}{8k^2R_0^2}\int_{\Omega_R}\left(|\nabla u|^2 + k^2|u|^2\right) \, \rd x + \left(2-\frac{3q}{2}-p-\frac{c_\chi}{2}\right)\int_{\Omega_R} |\partial_du|^2\, \rd x\\ \nonumber
& & \hspace{-2cm} + \frac{q}{2}\int_{\Omega_R}\left(|\nabla u|^2 + k^2|u|^2\right)\chi(r) \, \rd x \\ \nonumber
& \leq & \frac{p}{8k^2R_0^2}\Re\int_{\Omega_R}f\bar u \, \rd x - 2kR\Im \int_{\Omega_R}f\bar u \, \rd x + \int_{\Omega_R}\Delta \alpha |u|^2\, \rd x\\ \label{eq:mainboundnew}
& & \hspace{0.5cm} + \Re\int_{\Omega_R}f\Big(2x_d\partial_d\bar u(1-\chi(r)) + 2r\partial_r\bar u \chi(r) + 2\alpha \bar u\Big) \, \rd x.
\end{eqnarray}
\ele

\bpf
By Lemmas \ref{lem:E0} and \ref{lem:E1}, for all $q,\mu>0$ with $q\leq 1$ and $q+\mu\leq 2$,
\begin{eqnarray*}
& &\hspace{-2cm} \left(2-q-\mu -\frac{c_\chi}{2}\right)\int_{\Omega_R} |\partial_du|^2\, \rd x + \mu \int_{\Omega_R}|\nabla u|^2\chi(r)\, \rd x + qk^2\int_{\Omega_R}|u|^2\chi(r) \, \rd x \\
& \leq & - 2kR\Im \int_{\Omega_R}f\bar u \, \rd x\\
& & \hspace{0.5cm} + \Re\int_{\Omega_R}f\Big(2x_d\partial_d\bar u(1-\chi(r)) + 2r\partial_r\bar u \chi(r) + 2\alpha \bar u\Big) \, \rd x + \int_{\Omega_R}\Delta \alpha |u|^2\, \rd x.
\end{eqnarray*}
If also $p>0$ and \eqref{eq:klarge} holds, then it follows from Lemmas \ref{lem:E3new} and \ref{lem:Green} that
\begin{eqnarray*}
p\int_{\Omega_R}|\partial_du|^2\, \rd x + \frac{qk^2}{2}\int_{\Omega_R}|u|^2\chi(r) \, \rd x \geq
\frac{p}{8k^2R_0^2}\int_{\Omega_R}\left(|\nabla u|^2 + k^2|u|^2\right) \, \rd x - \frac{p}{8k^2R_0^2}\Re\int_{\Omega_R}f\bar u \, \rd x.
\end{eqnarray*}
Combining these two inequalities and choosing $\mu =q/2$ gives the required result.
\epf

\

In the following lemma $\alpha$ is defined by \eqref{eq:alpha_def} (with $Z$ given by \eqref{eq:Zdeff}), so that
\begin{equation} \label{eq:alphaexp}
2\alpha(x) =  1+ (d-q-1)\chi(r) +(r^2-x_d^2)\chi^\prime(r)/r
\end{equation}
and
\begin{eqnarray} \nonumber
2\Delta \alpha(x) &=& \chi'(r) \left[ \frac{(d-1)(d-q)-2}{r} + \frac{(d+1)}{r} \frac{x_d^2}{r^2}\right]\\ \label{eq:DeltaAlpha}
& & \hspace{0.5cm} +\chi''(r) \left[2d-q -\frac{x^2_d}{r^2}(d+1) \right] + \chi'''(r) \frac{(r^2-x_d^2)}{r}.
\end{eqnarray}

Moreover, we use the notations
\begin{equation} \label{eq:malpha}
m_\alpha(r) := \sup_{x\in B_r} \Delta \alpha(x), \;\; \mbox{for } r>0, \;\; \mbox{ and } \; M_\alpha := m_\alpha(R).
\end{equation}
It is clear from \eqref{eq:DeltaAlpha} and since $\chi\in C^3[0,\infty)$ and $\chi(r)= 0$ for $0\leq r\leq R_0$, that $m_\alpha\in C(0,\infty)$ with $m_\alpha(r)=0$, for $0<r\leq R_0$.

\ble\label{lem:E5new}
Let $\Oi$ be an $(R_0,R_1)$ obstacle and $\alpha$ be defined by \eqref{eq:alpha_def} with
\begin{equation}\label{eq:qdef}
q = \frac{1}{8}(4-c_\chi),
\end{equation}
where $c_\chi$ is given by \eqref{eq:chimax}. Suppose that $u$ is the solution of the exterior Dirichlet problem described in Theorem \ref{thm:resol}, that $R$ is large enough so that $\supp(f)\subset \Omega_R$,  that $R_0<R_*<R_1$ and $R_*$ is chosen small enough so that
\begin{equation} \label{eq:mabound}
m_\alpha(R_*) \leq \frac{q}{128R_0^2},
\end{equation}
and that $k$ is chosen large enough so that
\begin{equation} \label{kfinalbound}
k^2 \geq \max\left(\frac{4M_\alpha}{q\chi(R_*)},\frac{9}{4R_0^2\chi(2R_0)}\right).
\end{equation}
Then (where the $k$-dependent norms are as defined in \eqref{eq:normdefs})
\begin{eqnarray} \nonumber
& &\hspace{-2cm} \frac{q^2}{32k^2R_0^2}\|u\|^2_{H^1_k(\Omega_R)}+ q^2\|\partial_d u\|^2_{L^2(\Omega_R)} + \frac{q^2}{4}\|u\|^2_{H^1_k(\Omega_R;\chi)}\\ \label{eq:mainquant}
& \leq & \left(\frac{2q^2R_0^2}{81} + 128R_0^2\left(k^2R^2 + \|\alpha\|_{L^\infty(\Omega_R)}^2\right) + 4R^2 + R_1^2\right)\|f\|_{L^2(\Omega_+)}^2.
\end{eqnarray}
If the support of $f$ does not intersect $\Omega_{R_0}$ and \eqref{eq:fchinorm} holds, then also
\begin{eqnarray} \nonumber
& &\hspace{-1.5cm} \frac{5q^2}{128k^2R_0^2}\|u\|^2_{H^1_k(\Omega_R)}+ q^2\|\partial_d u\|^2_{L^2(\Omega_R)} + \frac{q^2}{8}\|u\|^2_{H^1_k(\Omega_R;\chi)}\\ \label{eq:mainquant2}
& \leq & \left(\frac{2q^2R_0^2}{81} + 128R_0^2 \|\alpha\|_{L^\infty(\Omega_R)}^2 + 4R^2 + R_1^2\right)\|f\|_{L^2(\Omega_+)}^2 + 8R^2\|f\|_{L^2(\Omega_+;\chi^{-1})}^2.
\end{eqnarray}
\ele

\begin{proof} The assumption \eqref{kfinalbound} ensures that \eqref{eq:klarge} holds with $p=q/2$, so that \eqref{eq:mainboundnew} holds with $p=q/2$ and $q$ given by \eqref{eq:qdef}, which implies that
\begin{eqnarray}\nonumber
& &\hspace{-2cm} \frac{q}{16k^2R_0^2}\|u\|^2_{H^1_k(\Omega_R)} + 2q\|\partial_d u\|^2_{L^2(\Omega_R)} + \frac{q}{2}\|u\|^2_{H^1_k(\Omega_R;\chi)} \\ \nonumber
& \leq & \frac{q}{16k^2R_0^2}\Re\int_{\Omega_R}f\bar u \, \rd x - 2kR\Im \int_{\Omega_R}f\bar u \, \rd x + \int_{\Omega_R}\Delta \alpha |u|^2\, \rd x\\ \label{eq:mainboundnew2}
& & \hspace{0.5cm} + \Re\int_{\Omega_R}f\Big(2x_d\partial_d\bar u(1-\chi(r)) + 2r\partial_r\bar u \chi(r) + 2\alpha \bar u\Big) \, \rd x.
\end{eqnarray}
 We proceed by bounding, in terms of the left hand side, the various terms on the right hand side of this last inequality. Firstly, we note that
\begin{eqnarray} \nonumber
\int_{\Omega_R}\Delta \alpha |u|^2\, \rd x &\leq &m_\alpha(R_*)\int_{\Omega_{R_*}}|u|^2\, \rd x + M_\alpha \int_{\Omega_R\setminus \Omega_{R_*}}|u|^2\, \rd x\\ \nonumber
&\leq &m_\alpha(R_*)\|u\|_{L^2(\Omega_R)}^2 + \frac{M_\alpha}{\chi(R_*)} \int_{\Omega_R}|u|^2\chi(r)\, \rd x\\ \label{eq:Deltaalphabound}
&\leq &\frac{q}{128R_0^2}\|u\|_{L^2(\Omega_R)}^2 + \frac{k^2q}{4} \int_{\Omega_R}|u|^2\chi(r)\, \rd x,
\end{eqnarray}
since \eqref{eq:mabound} and \eqref{kfinalbound} hold. Secondly, using
\eqref{eq:Cauchy}, we see that, for $\epsilon_1,\epsilon_2>0$,
\begin{eqnarray}\nonumber
& &\hspace{-1.5cm}  \Re\int_{\Omega_R}f\Big(2x_d\partial_d\bar u(1-\chi(r)) + 2r\partial_r\bar u \chi(r)\Big) \, \rd x\\ \nonumber
& \leq &\epsilon^{-1}_1R_1^2\|f\|_{L^2(\Omega_+)}^2 + \epsilon_1\|\partial_d u\|_{L^2(\Omega_R)}^2 + R^2\epsilon^{-1}_2 \|f\|_{L^2(\Omega_+)}^2 + \epsilon_2\int_{\Omega_R}|\nabla u|^2\chi(r)\, \rd x\\ \label{eq:boundsder}
& = & \frac{R_1^2}{q}\|f\|_{L^2(\Omega_+)}^2 + q\|\partial_d u\|_{L^2(\Omega_R)}^2+ \frac{4R^2}{q} \|f\|_{L^2(\Omega_+)}^2 + \frac{q}{4}\int_{\Omega_R}|\nabla u|^2\chi(r)\, \rd x
\end{eqnarray}
if $\epsilon_1=q$ and $\epsilon_2 = q/4$. Similarly, for $\epsilon_3,\epsilon_4>0$, since $k^2R_0^2 \geq 9/4$ by \eqref{kfinalbound},
\begin{eqnarray}\nonumber
& & \hspace{-1.5cm} \frac{q}{16k^2R_0^2}\Re\int_{\Omega_R}f\bar u \, \rd x + 2\Re\int_{\Omega_R}f\alpha \bar u \, \rd x\\ \nonumber
& \leq & \frac{q}{36}\left|\int_{\Omega_R}f\bar u \, \rd x\right| + 2\left|\int_{\Omega_R}f\alpha \bar u \, \rd x\right|\\ \nonumber
& \leq &\frac{q^2}{5184}\epsilon^{-1}_3\|f\|_{L^2(\Omega_+)}^2 + \epsilon_3\|u\|_{L^2(\Omega_R)}^2 +\epsilon^{-1}_4\|\alpha\|_{L^\infty(\Omega_R)}^2 \|f\|_{L^2(\Omega_+)}^2 + \epsilon_4\|u\|_{L^2(\Omega_R)}^2\\ \label{eq:boundsstraight}
& = & \frac{2qR_0^2}{81}\|f\|_{L^2(\Omega_+)}^2 + \frac{128R_0^2}{q}\,\|\alpha\|_{L^\infty(\Omega_R)}^2\|f\|_{L^2(\Omega_+)}^2 + \frac{q}{64R_0^2}\|u\|_{L^2(\Omega_R)}^2
\end{eqnarray}
if $\epsilon_3=\epsilon_4=q/(128R_0^2)$. Finally, for $\epsilon>0$,
\begin{eqnarray}\nonumber
- 2kR\Im \int_{\Omega_R}f\bar u \, \rd x & \leq &k^2R^2\epsilon^{-1} \|f\|_{L^2(\Omega_+)}^2 + \epsilon\|u\|_{L^2(\Omega_R)}^2\\ \label{eq:boundskR}
& = & \frac{128k^2R^2R_0^2}{q}\|f\|_{L^2(\Omega_+)}^2 + \frac{q}{128R_0^2}\|u\|_{L^2(\Omega_R)}^2
\end{eqnarray}
if $\epsilon = q/(128R_0^2)$. Combining \eqref{eq:mainboundnew2}, \eqref{eq:Deltaalphabound}, \eqref{eq:boundsstraight}, and \eqref{eq:boundskR} we obtain \eqref{eq:mainquant}.

If the support of $f$ does not intersect $\Omega_{R_0}$ and \eqref{eq:fchinorm} holds then we can replace \eqref{eq:boundskR} by
\begin{eqnarray}\nonumber
- 2kR\Im \int_{\Omega_R}f\bar u \, \rd x & \leq &R^2\epsilon^{-1} \|f\|_{L^2(\Omega_+;\chi^{-1})}^2 + \epsilon k^2\int_{\Omega_R}|u|^2\chi(r)\, \rd x\\ \label{eq:boundskR2}
& = & \frac{8R^2}{q}\|f\|_{L^2(\Omega_+;\chi^{-1})}^2 + \frac{qk^2}{8}\int_{\Omega_R}|u|^2\chi(r)\, \rd x,
\end{eqnarray}
if we choose $\epsilon = q/8$. Combining \eqref{eq:mainboundnew2}, \eqref{eq:Deltaalphabound}, \eqref{eq:boundsstraight}, and \eqref{eq:boundskR2} we obtain \eqref{eq:mainquant2}.
\end{proof}

\

\bpf[Proof of Theorem \ref{thm:resol} from Lemma \ref{lem:E5new}]
The bound \eqref{eq:mainquant} implies that there exists a constant $C>0$, depending only on the function $\chi$ in Definition \ref{def:trapb},  such that
\begin{align} \label{eq:finalb5}
\frac{1}{kR_0}\|u\|_{H^1_k(\Omega_R)} + \|\partial_d u\|_{L^2(\Omega_R)} + \|u\|_{H^1_k(\Omega_R;1-\chi)}
   \leq C  R(1+kR_0) \|f\|_{L^2(\Omega_+)},
\end{align}
for $k\geq k_1$ and $R>\max(R_1,\sqrt{13}\,R_0)$, where $k_1>0$ is given by \eqref{kfinalbound}.
Clearly, \eqref{eq:finalb5}
implies that the same bound holds also for $R_\Gamma<R\leq \max(R_1,\sqrt{13}\, R_0)$, with $C$ replaced by $\max(R_1,\sqrt{13}\,R_0)C/R_\Gamma$. Thus \eqref{eq:resol} holds for $k\geq k_1$.  Given $k_0>0$, that the bound \eqref{eq:resol} holds for $k\in (k_0, k_1)$ follows by standard arguments; see
the text after Definition \ref{def:resol}.

Arguing similarly, if the support of $f$ does not intersect $\Omega_{R_0}$ and \eqref{eq:fchinorm} holds, then \eqref{eq:mainquant2} implies that there exists a constant $C^\prime>0$, depending only on the function $\chi$ in Definition \ref{def:trapb},  such that
\begin{align} \label{eq:finalb52}
\frac{1}{kR_0}\|u\|_{H^1_k(\Omega_R)} + \|\partial_d u\|_{L^2(\Omega_R)} + \|u\|_{H^1_k(\Omega_R;1-\chi)}
   \leq C^\prime R \|f\|_{L^2(\Omega_+;\chi^{-1})},
\end{align}
for $k\geq k_1$ and $R>\max(R_1,\sqrt{13}\,R_0)$, and it follows that the bound \eqref{eq:resol} holds with $k\|f\|_{L^2(\Omega_+)}$ replaced by $\|f\|_{L^2(\Omega_+;\chi^{-1})}$.
\epf

\section{Proof of Theorem \ref{thm:dtn} on the exterior DtN map} \label{sec:dtn}

\begin{definition}[$K$ resolvent estimate] \label{def:resol}
For $K\in C[0,\infty)$, with $K(k)\geq 1$ for $k>0$, we say that $\Omega_+$ {\em satisfies a $K$ resolvent estimate} if, whenever $u\in H^1_{\mathrm{loc}}(\Omega_+)$ satisfies the radiation condition \eqref{eq:src}, the boundary condition $\gamma_+u=0$, and the Helmholtz equation $\Delta u+k^2 u = -f$ in $\Omega_+$, with $f\in L^2(\Omega_+)$ compactly supported, it holds for all $R>\max_{x\in \Gamma\cup\supp(f)}|x|$ that
\begin{equation}\label{eq:resolgen}
\|u\|_{H^1_k(\Omega_R)}\lesssim K(k)\|f\|_{L^2(\Omega_+)}, \quad \mbox{for } k>0,
\end{equation}
where the omitted constant depends on $R$.
\end{definition}
To show that $\Omega_+$ satisfies a $K$ resolvent estimate it is enough to show that \eqref{eq:resolgen} holds for all sufficiently large $k$. For, as observed at the end of \S\ref{sec:infsup} below in \eqref{eq:L} and Lemma \ref{lem:CWMo}, the bound \eqref{eq:resolgen} holds for every $\Omega_+$ for all sufficiently small $k>0$. Further, for every $k_0>0$, it then follows, by continuity arguments and well-posedness at every fixed $k>0$, that \eqref{eq:resolgen} holds for $0<k\leq k_0$.  (Concretely, one route to carrying out these latter arguments is to  note that the inf-sup constant $\beta_R$, given by \eqref{eq:infsupg} below, is positive for each fixed $k$ and depends continuously on $k$, and then apply Lemma \ref{lem:CWMo}.)

The bounds \eqref{eq:dtn2} and \eqref{eq:dtn3} in Theorem \ref{thm:dtn} follow from Theorem \ref{thm:resol} combined with the following lemma; this lemma encapsulates the method laid out in \cite[\S3]{BaSpWu:16} for deriving wavenumber-explicit bounds on the exterior DtN map from resolvent estimates in the exterior domain.

\begin{lemma}[From resolvent estimates to DtN map bounds]\label{lem:recipe} Suppose that $\Omega_+$ satisfies a $K$ resolvent estimate, for some $K\in C[0,\infty)$ with $K(k)\geq 1$ for $k>0$. Then, whenever $u\in H^1_{\mathrm{loc}}(\Omega_+)$ satisfies the radiation condition \eqref{eq:src} and $\Delta u+k^2 u = 0$ in $\Omega_+$, it holds for all
$R> R_\Gamma$ that, given $k_0>0$,
\begin{equation}\label{eq:dtn2l}
\|u\|_{H^1_k(\Omega_R)} + \|\partial_n^+ u\|_{L^2(\Gamma)} \lesssim K(k) \, \|g\|_{H^1_k(\Gamma)}, \quad \mbox{for } k\geq k_0,
\end{equation}
provided  $g:= \gamma_+u\in H^1(\Gamma)$. Moreover, for $k\geq k_0$,
\begin{equation}\label{eq:dtn3l}
\|\partial_n^+ u\|_{H_k^{s-1}(\Gamma)} \lesssim K(k) \|g\|_{H_k^{s}(\Gamma)} \quad \mbox{and} \quad \|\partial_n^+ u\|_{H^{s-1}(\Gamma)} \lesssim kK(k) \|g\|_{H^{s}(\Gamma)},
\end{equation}
uniformly for $0\leq s\leq 1$, assuming, in the case $s>1/2$, that $g\in H^s(\Gamma)$.
\end{lemma}
\begin{proof}
We sketch the proof, which is essentially contained in \cite[\S3]{BaSpWu:16}. Suppose that $\Omega_+$ satisfies a $K$ resolvent estimate, with the given conditions on $K$, and that $u\in H^1_{\mathrm{loc}}(\Omega_+)$ satisfies \eqref{eq:src}, $\Delta u+k^2 u = 0$ in $\Omega_+$, and  $g:= \gamma_+u\in H^1(\Gamma)$. Let $w\in H^1(\Omega_+)$ satisfy $\Delta w + (k^2 + \ri k) w = 0$ in $\Omega_+$ and the boundary condition $\gamma_+ w = g$. Green's identity can then be used to show that, given $k_0>0$,
\begin{equation} \label{eq:w}
\|w\|_{H_k^1(\Omega_+)} \lesssim \|g\|_{H^1_k(\Gamma)}, \quad \mbox{for } k\geq k_0
\end{equation}
 \cite[Lemma 3.3]{BaSpWu:16}.
Choose $\psi\in C_0^\infty(\R^d)$ that is equal to one on $\Omega_-$, and define $v:= u-\psi w$. Then $v\in H^1_{\mathrm{loc}}(\Omega_+)$ satisfies \eqref{eq:src}, $\gamma_+v=0$, and $\Delta v + k^2 v = h:= \ri k \psi w - w\Delta \psi - 2\nabla w \cdot \nabla \psi\in L^2(\Gamma)$, and $h$ is compactly supported. Thus, for all
$R> R_\Gamma$,
\begin{equation}\label{eq:resolgen2}
\|v\|_{H^1_k(\Omega_R)}\lesssim K(k)\|h\|_{L^2(\Omega_+)} \lesssim K(k) \|w\|_{H_k^1(\Omega_+)},
\end{equation}
for $k\geq k_0$. Combining \eqref{eq:w} and \eqref{eq:resolgen2} we see that, for all $R>R_\Gamma$,
$$
\|u\|_{H_k^1(\Omega_R)} \lesssim K(k)\|g\|_{H^1_k(\Gamma)}, \quad \mbox{for } k\geq k_0.
$$
That $\|\partial_n^+ u\|_{L^2(\Gamma)}$ is also bounded by the right hand side of this last equation follows from \cite[Lemma 2.3]{BaSpWu:16} (essentially Ne\v{c}as' regularity result \cite[\S5.1.2]{Ne:67}, \cite[Theorem 4.24(ii)]{Mc:00}, proved using a Rellich identity).
Using the notation $\DtN:H^{1/2}(\Gamma)\to H^{-1/2}(\Gamma)$ to denote the DtN map for the exterior domain $\Omega_+$, we see equation \eqref{eq:dtn2l} implies that
$$
\|\DtN\|_{H^1_k(\Gamma)\to L^2(\Gamma)} \lesssim K(k) \quad \mbox{so} \quad  \|\DtN\|_{H^1(\Gamma)\to L^2(\Gamma)} \lesssim k K(k),
$$
for $k\geq k_0$. It is well known (e.g., \cite[Theorem 2.31]{ChGrLaSp:12}) that $\DtN$ can be extended uniquely to a bounded mapping from $H^{s+1/2}(\Gamma) \rightarrow H^{s-1/2}(\Gamma)$ for $|s|\leq 1/2$.
Since $\DtN$ is self-adjoint with respect to the real inner product $(\cdot,\cdot)_\Gamma^r$ (see \cite[Section 2.7]{ChGrLaSp:12}), \eqref{eq:dtn3l} follows from \eqref{eq:adj2} (cf.\ \cite[Lemma 2.3]{Sp:14}).
\end{proof}

\bre[Previous uses of the arguments in Lemma \ref{lem:recipe}]
The method in Lemma \ref{lem:recipe} is a sharpening of arguments used to obtain bounds on the DtN map from resolvent estimates in \cite{LaVa:11,Sp:14}, with this type of argument going back at least to \cite[\S5]{LaPh:72}. Indeed, in \cite{LaVa:11,Sp:14} the equation $\Delta w + (k^2 + \ri k) w = 0$ in the proof of Lemma \ref{lem:recipe} below is replaced by $\Delta w -k^2w = 0$, losing a factor $k$ in the final estimates.
\ere

\vspace{1ex}

To prove the last part of Theorem \ref{thm:dtn}, namely the bound \eqref{eq:dtn4}, we use the interior elliptic regularity estimate that if, for some $x\in \R^d$ and $\varepsilon>0$, $v\in C^2(B_\epsilon(x))$, $\Delta v = f$ in $B_\varepsilon(x)$ and $v,f \in L^\infty(B_\varepsilon(x))$, then \cite[Theorem 3.9]{GilTru}, for some constant $C_d>0$ that depends only on $d$,
$$
|\nabla v(x)| \leq \frac{C_d}{\varepsilon}\left(\|v\|_{L^\infty(B_\varepsilon)} + \varepsilon^2\|f\|_{L^\infty(B_d(x))}\right).
$$
In the particular case that $f=-k^2v$, so that $\Delta v + k^2v=0$, this estimate is
\begin{equation} \label{eq:intreg}
|\nabla v(x)| \leq C_d\,\frac{(1+k^2\varepsilon^2)}{\varepsilon}\,\|v\|_{L^\infty(B_\varepsilon(x))}.
\end{equation}

\begin{proof}[Proof of Theorem \ref{thm:dtn}] The bounds \eqref{eq:dtn2} and \eqref{eq:dtn3} follow immediately from Lemma \ref{lem:recipe} and Theorem \ref{thm:resol}, which shows (under the conditions on $\Omega_+$, $R_1$, and $R_0$, and taking into account the text after Definition \ref{def:resol})
that $\Omega_+$ satisfies a $K$ resolvent estimate with $K(k) = 1+k^2$. The bound \eqref{eq:dtn2} and \eqref{eq:intreg} imply a version of \eqref{eq:dtn4}, but with $k^2$ replaced by $k^3$.
To show the sharper bound \eqref{eq:dtn4}, choose $\psi\in C^\infty_0(\R^d)$ supported in $G$  that is equal to one on $\Omega_-\cup B_{R^\prime}$, for some $R^\prime>R_0$, and let $w := u+ \psi u^i$. Then $w$ satisfies \eqref{eq:src}, $\gamma_+w = 0$, and $\Delta w + k^2 w = h:= 2\nabla u^i\cdot \nabla \psi + u^i\Delta \psi$ in $\Omega_+$. Further, $h\in$ $L^2(\Oe)$ is compactly supported, $h=0$ in $B_{R^\prime}\cap \Omega_+$, and, applying \eqref{eq:intreg} with $\varepsilon=\min(\epsilon, k^{-1})$ for some sufficiently small $\epsilon$, we see that
$$
\|h\|_{L^2(\Omega_+)} \lesssim (1+k) \max_{x\in G}|u^i(x)|, \quad \mbox{for } k>0.
$$
It follows from Theorem \ref{thm:resol} (see \eqref{eq:resol3}) that, given $k_0>0$ and $R>R_\Gamma$,
$$
\|w\|_{H^1_k(\Omega_R)}  \lesssim k\|h\|_{L^2(\Omega_+)},
$$
for $k\geq k_0$. Applying \cite[Lemma 2.3]{BaSpWu:16} we deduce that also $\|\partial_n^+ w\|_{L^2(\Gamma)}\lesssim k\|h\|_{L^2(\Omega_+)}$. Applying \eqref{eq:intreg} again, with the same choice of $\varepsilon$, we obtain also that
$$
\|\psi u^i\|_{H^1_k(\Omega_+)} + \|\partial_n^+u^i\|_{L^2(\Omega_+)} \lesssim (1+k) \max_{x\in G}|u^i(x)|, \quad \mbox{for } k>0.
$$
Combining these inequalities it follows that \eqref{eq:dtn4} holds.
\end{proof}

\

Using Lemma \ref{lem:recipe} we can derive other bounds on the exterior DtN map that apply to classes of trapping domains, using the two other trapping resolvent estimates in the literature, which we discussed in \S\ref{sec:4star}.

\begin{corollary}[Worst case bounds on the DtN map] \label{cor:dtnws} Let $u\in H^1_{\mathrm{loc}}(\Omega_+)$ be a solution to the Helmholtz equation $\Delta u + k^2 u = 0$ in $\Omega_+$ that satisfies \eqref{eq:src} and $\gamma_+u=g$. If $\Omega_-$ is $C^\infty$ there exists $\alpha>0$ such that, given $k_0>0$,
\begin{equation}\label{eq:dtnwc2}
\|\partial_n^+ u\|_{L^2(\Gamma)} \lesssim \exp(\alpha k) \,\|g\|_{H^1_k(\Gamma)},
\end{equation}
for all $k\geq k_0$ if $g\in H^1(\Gamma)$. In fact, for $k\geq k_0$,
\begin{equation}\label{eq:dtnwc3}
\|\partial_n^+ u\|_{H_k^{s-1}(\Gamma)} \lesssim \exp(\alpha k) \|g\|_{H_k^{s}(\Gamma)} \quad \mbox{and} \quad \|\partial_n^+ u\|_{H^{s-1}(\Gamma)} \lesssim k\exp(\alpha k) \|g\|_{H^{s}(\Gamma)},
\end{equation}
uniformly for $0\leq s\leq 1$, assuming, in the case $s>1/2$, that $g\in H^s(\Gamma)$.
\end{corollary}

The second resolvent estimate, developed by Ikawa \cite{Ik:83,Ik:88} and Burq \cite{Bu:04}, is for mild, hyperbolic trapping, where $\Omega_-$ is an {\em Ikawa-like union of convex obstacles} in the following sense.

\begin{definition}[Ikawa-like union of convex obstacles \cite{Ik:88,Bu:04}] \label{def:IB} We say that $\Omega_-$ is an {\em Ikawa-like union of convex obstacles} if:
\begin{itemize}
\item[(i)] for some $M\in \mathbb{N}$, $\overline{\Omega_-} = \displaystyle{\bigcup_{i=1}^N}\Theta_i$, where $\Theta_1,...,\Theta_N\subset \R^d$ are disjoint compact $C^\infty$ strictly convex sets with $\kappa>0$, where $\kappa$ is the infimum of the principal curvatures of the boundaries of the obstacles $\Theta_i$;
\item[(ii)] for $1\leq i,j,\ell \leq N$, $i\neq j$, $j\neq \ell$, $\ell\neq i$,
$$
\mathrm{Convex}\; \mathrm{hull}(\Theta_i\cup\Theta_j)\cap \Theta_\ell = \emptyset;
$$
\item[(iii)]  if $N>2$, $\kappa L > N$, where $L$ denotes the minimum of the distances between pairs of obstacles.
\end{itemize}
\end{definition}

\begin{corollary}[DtN map for Ikawa-like union of convex obstacles] \label{cor:dtnib} Let $u\in H^1_{\mathrm{loc}}(\Omega_+)$ be a solution to the Helmholtz equation $\Delta u + k^2 u = 0$ in $\Omega_+$ that satisfies \eqref{eq:src} and $\gamma_+u=g$. If $\Omega_-$ is an  Ikawa-like union of convex obstacles then, given $k_0>0$,
\begin{equation}\label{eq:dtnib2}
\|\partial_n^+ u\|_{L^2(\Gamma)} \lesssim \log(2+k) \,\|g\|_{H^1_k(\Gamma)},
\end{equation}
for all $k\geq k_0$ if $g\in H^1(\Gamma)$. In fact, for $k\geq k_0$,
\begin{equation}\label{eq:dtnik}
\|\partial_n^+ u\|_{H_k^{s-1}(\Gamma)} \lesssim \log(2+k) \|g\|_{H_k^{s}(\Gamma)} \quad \mbox{and} \quad \|\partial_n^+ u\|_{H^{s-1}(\Gamma)} \lesssim k\log(2+k) \|g\|_{H^{s}(\Gamma)},
\end{equation}
uniformly for $0\leq s\leq 1$, assuming, in the case $s>1/2$, that $g\in H^s(\Gamma)$.
\end{corollary}

\section{Proof of Corollary \ref{cor:infsup} on the inf-sup constant} \label{sec:infsup}

Since $\Omega_+$ is unbounded, standard FEMs cannot be applied directly to the exterior Dirichlet problem. A standard fix is to reformulate the exterior Dirichlet problem as a variational problem in the truncated domain $\Omega_R$, for some $R>R_\Gamma$. The effect of the rest of $\Omega_+$, i.e.\ of $\Omega_R^+:= \Omega_+\setminus{\overline{\Omega_R}}$, is replaced by the exact DtN map on $\Gamma_R$ for $\Omega^+_R$, abbreviated as $P^+_R$ (our notation as in Corollary \ref{cor:infsup}). As $\Omega_R^+$ is a geometry in which the Helmholtz equation separates, the action of $P^+_R$ can be computed analytically (e.g.\ \cite[Equations (3.5)--(3.6)]{ChMo:08}).

Given $f\in L^2(\Omega_+)$ with compact support in $\overline{\Omega_+}$, consider the problem of finding $u\in H^1_{\mathrm{loc}}(\Omega_+)$ such that $u$ satisfies the radiation condition \eqref{eq:src}, the Helmholtz equation $\Delta u + k^ 2 u = -f$ in $\Omega_+$, and $\gamma_+u=0$ on $\Gamma$.
It is well-known that a variational formulation in $\Omega_R$ can be obtained by multiplying the Helmholtz equation by a test function $v_R\in V_R$, integrating by parts, and applying the boundary condition $\gamma_+u=0$. In particular (e.g., \cite{Ne:01}), if the support of $f$ lies in $\Omega_R$, $u$ satisfies this BVP in $\Omega_+$ if and only if $u_R:= u|_{\Omega_R}\in V_R$ and
\begin{align} \label{eq:weakst}
a(u_R,v_R) = G(v_R), \quad \mbox{for all } v_R\in V_R,
\end{align}
where $a(\cdot,\cdot)$ is defined in \eqref{eq:ses}, $V_R$ is defined immediately before Corollary \ref{cor:infsup}, and
\begin{align} \label{eq:Gdef}
G(v) := \int_{\Omega_R} \bar v f \rd x, \quad \mbox{for } v\in V_R.
\end{align}

The following lemma is proved as \cite[Lemmas 3.3, 3.4]{ChMo:08}.

\begin{lemma}[Link between resolvent estimates and bounds on the inf-sup constant]\label{lem:CWMo} Suppose that $R>R_\Gamma$, $L>0$, $k>0$, and that
\beq \label{eq:rb}
\|u\|_{H^1_k(\Omega_R)} \leq L\|f\|_{L^2(\Omega_+)},
 \eeq
 for all $f\in L^2(\Omega_+)$ supported in $\Omega_R$, where $u\in H^1_{\mathrm{loc}}(\Omega_+)$ is the solution of $\Delta u + k^ 2 u = -f$ in $\Omega_+$ that satisfies \eqref{eq:src} and $\gamma_+u=0$. Then
\begin{equation} \label{eq:infsupg}
\beta_R:= \inf_{0\neq u\in V_R}\, \sup_{0\neq v\in V_R} \frac{|a(u,v)|}{\|u\|_{H^1_k(\Omega_R)}\|v\|_{H^1_k(\Omega_R)}} \geq \alpha,
\end{equation}
where $\alpha = (1+2kL)^{-1}$. Conversely, if \eqref{eq:infsupg} holds for some $\alpha>0$, then \eqref{eq:rb} holds for all $f\in L^2(\Omega_+)$ supported in $\Omega_R$, with $L=\alpha^{-1}\min(k^{-1},c_R)$, where
\beq \label{eq:poin}
c_R := \sup_{0\neq v\in V_R} \frac{\|v\|_{L^2(\Omega_R)}}{\|\nabla v\|_{L^2(\Omega_R)}}.
\eeq
\end{lemma}

Corollary \ref{cor:infsup} follows immediately from Theorem \ref{thm:resol} and Lemma \ref{lem:CWMo}.

We remark also that (see \cite{Ne:01} or \cite[Lemma 2.1]{ChMo:08})
\begin{equation} \label{eq:infsup2}
\beta_R \geq \inf_{0\neq v\in V_R}\, \frac{\Re(a(u,v))}{\|v\|^2_{H^1_k(\Omega_R)}} \geq \inf_{0\neq v\in V_R}\,\frac{\int_{\Omega_R}(|\nabla v|^2-k^2|v|^2)\rd x}{\int_{\Omega_R}(|\nabla v|^2+k^2|v|^2)\rd x} \geq \frac{1-k^2c_R^2}{1+k^2c_R^2}.
\end{equation}
This, combined with Lemma \ref{lem:CWMo}, shows that, if $kc_R<1$, \eqref{eq:rb} holds for all $f\in L^2(\Omega_+)$ supported in $\Omega_R$, with
\beq \label{eq:L}
L=c_R\,\frac{1+k^2c_R^2}{1-k^2c_R^2}.
\eeq

\begin{remark}[Bound on $\beta_R^{-1}$ from a $K$ resolvent estimate]\label{rem:resinfsup}
In the language of Definition \ref{def:resol}, Lemma \ref{lem:CWMo} tells us that $\Omega_+$ satisfies a $K$ resolvent estimate (with $K$ satisfying the conditions of Definition \ref{def:resol}) if and only if the inf-sup constant satisfies
\beq \label{eq:infsupgen}
\beta_R^{-1} \lesssim (1+k)K(k), \quad \mbox{for } k>0,
\eeq
for all $R>R_\Gamma$. Table \ref{tab:Ainv} lists the known resolvent estimates for scattering by an obstacle, as well as the bounds $\beta_R^{-1}$ that follows from these.
\end{remark}

\bre[Upper bound on $\beta_R$]
The simple constructions in \cite[Lemma 3.10]{ChMo:08} (see also \cite[Lemma 4.12]{Sp:14}) show that for every $\Omega_+$ and every $R>R_\Gamma$,
\beq \label{eq:infsupub3}
\beta_R \lesssim (1+k)^{-1}, \quad \mbox{for } k>0;
\eeq
and the nontrapping resolvent estimate combined with \eqref{eq:infsupgen} shows that this is sharp.
\ere

\section{Combined-potential integral equation formulations and the proof of Corollary \ref{cor:CFIE}} \label{sec:ie}

Integral equation methods are widely used for both the theoretical analysis and the numerical solution of direct and inverse acoustic scattering problems (e.g., \cite{CoKr:83,CoKr:98,ChGrLaSp:12}). In this section we recall the standard integral equation formulations for the exterior Dirichlet problem, and derive new wavenumber-explicit bounds in the case when $\Omega_-$ is trapping, combining the resolvent and DtN estimates in Theorems \ref{thm:resol} and \ref{thm:dtn} (proved in Sections \ref{sec:resol} and \ref{sec:dtn}) with the sharp bounds for the interior impedance problem recently obtained in \cite{BaSpWu:16}.

\subsection{Integral equations for the exterior Dirichlet problem}
If $u$ is a solution of $\Delta u + k^2 u = 0$ in $\Oe$ that satisfies the radiation condition \eqref{eq:src}
then Green's representation theorem (see, e.g., \cite[Theorem 2.21]{ChGrLaSp:12}) gives
\beq\label{eq:Green}
u(x) = - \int_\Gamma \Phi_k(x,y) \dnpu(y) \, \rd s(y) + \int_\Gamma \pdiff{\Phi_k(x,y)}{n(y)} \gamma_+ u(y) \, \rd s(y), \quad x\in \Oe,
\eeq
where
$\Phi_k(x,y)$ is the fundamental solution of the Helmholtz equation given by
\beq\label{eq:fund}
\Phi_k(x,y):= \frac{\ri}{4}\left(\frac{k}{2\pi |x-y|}\right)^{(d-2)/2}H_{(d-2)/2}^{(1)}\big(k|x-y|\big)= \left\{\begin{array}{cc}
                                                                                                            \displaystyle{\frac{\ri}{4}H_0^{(1)}\big(k|x-y|\big)}, & d=2, \\
                                                                                                            \displaystyle{\frac{\re^{\ri k |x-y|}}{4\pi |x-y|}}, & d=3,
                                                                                                          \end{array}\right.
\eeq
where $H^{(1)}_\nu$ denotes the Hankel function of the first kind of order $\nu$.
Taking the exterior Dirichlet and Neumann traces of \eqref{eq:Green} on $\Gamma$ and using the jump relations for the single- and double-layer potentials (e.g.\ \cite[Equation 2.41]{ChGrLaSp:12}) we obtain the integral equations
\beq\label{eq:BIE1}
S_k \dnpu = \left( -\half I + D_k\right) \gamma_+ u \quad
\mbox{ and } \quad
\left( \half I + D_k'\right) \dnpu = H_k \gamma_+ u,
\eeq
where $S_k$, $D_k$ are the single- and double-layer operators, $D_k'$ is the adjoint double-layer operator, and $H_k$ is the hypersingular operator. These four integral operators are defined for $\phi\in\LtG$, $\psi \in \HoG$, and almost all $x\in\Gamma$ by
\begin{align}\label{eq:SD}
&S_k \phi(x) := \int_\Gamma \Phi_k(x,y) \phi(y)\,\rd s(y), \qquad
D_k \phi(x) := \int_\Gamma \frac{\partial \Phi_k(x,y)}{\partial n(y)}  \phi(y)\,\rd s(y),\\ \label{eq:DH}
&D'_k \phi(x) := \int_\Gamma \frac{\partial \Phi_k(x,y)}{\partial n(x)}  \phi(y)\,\rd s(y),\quad
H_k \psi(x) := \pdiff{}{n(x)} \int_\Gamma \pdiff{\Phi_k(x,y)}{n(y)} \psi(y)\, \rd s(y).
\end{align}
When $\Gamma$ is Lipschitz, the integrals defining $D_k$ and $D'_k$
must be understood as Cauchy principal value integrals and even when $\Gamma$ is smooth there are subtleties in defining $H_k\psi$ for $\psi\in H^1(\Gamma)$ which we ignore here (see, e.g., \cite[\S2.3]{ChGrLaSp:12}).

For the exterior Dirichlet problem, the integral equations
\eqref{eq:BIE1} are both equations for
the unknown Neumann trace $\dnpu$. However the first of these equations is not
uniquely solvable when $-k^2$ is a Dirichlet eigenvalue of the
Laplacian in $\Oi$, and the second is not uniquely solvable
when $-k^2$ is a Neumann eigenvalue of the Laplacian in $\Oi$; see, e.g., \cite[Theorem 2.25]{ChGrLaSp:12}.

 One standard way to resolve this difficulty (going back to the work of \cite{BuMi:71}) is to take a linear combination of the two equations, which yields the integral equation
\beq\label{eq:CFIE2}
\opA \dnpu = \opBM \gamma_+ u
\eeq
where
\beq\label{eq:CFIEdef}
\opA := \half I + D_k' - \ri \eta S_k \quad \mbox{and} \quad \opBM := H_k + \ri \eta \left(\half I - D_k\right).
\eeq
If $\eta \in \Rea\setminus\{0\}$
then the integral operator $\opA$ is invertible (on appropriate Sobolev spaces) and so \eqref{eq:CFIE2} can be used to solve the exterior Dirichlet problem for all  $k>0$.
Indeed, if $\eta\in \Rea\setminus\{0\}$ then $\opA$ is a bounded invertible operator from $H^s(\Gamma)$ to itself for $-1\leq s\leq 0$; \cite[Theorem 2.27]{ChGrLaSp:12}.

An alternative resolution (proposed essentially simultaneously by \cite{BrWe:65,Le:65,Pa:65}) is to work with a so-called {\em indirect} formulation, looking for a solution to the exterior Dirichlet problem as the {\em combined double- and single-layer potential}
\beqs
u(x) = \int_\Gamma \pdiff{\Phi_k(x,y)}{n(y)} \phi(y) \, \rd s(y) - \ri \eta \int_\Gamma \Phi_k(x,y) \phi(y) \, \rd s(y), \quad x\in \Oe,
\eeqs
for some $\phi\in H^{1/2}(\Gamma)$ and $\eta \in \R\setminus\{0\}$. It follows from the jump relations \cite[Equation 2.41]{ChGrLaSp:12} that this ansatz satisfies the exterior Dirichlet problem with Dirchlet data $h=\gamma_+u\in H^{1/2}(\Gamma)$ if and only if
\beq \label{eq:BW}
A_{k,\eta} \phi = h,
\eeq
where
\beq\label{eq:CFIEdef2}
A_{k,\eta} := \half I + D_k - \ri \eta S_k.
\eeq
 If $\eta\in \Rea\setminus\{0\}$ then $A_{k,\eta}$ is a bounded invertible operator from $H^s(\Gamma)$ to itself for $0\leq s\leq 1$; \cite[Theorem 2.27]{ChGrLaSp:12}. The operators $\opA$ and $A_{k,\eta}$ are closely related in that $\opA$ is the adjoint of $A_{k,\eta}$ with respect to the real $L^2$ inner product on $\Gamma$, i.e.\ $(A_{k,\eta}\phi,\psi)_\Gamma^r=(\phi,\opA\psi)_\Gamma^r$, for all $\phi, \psi\in L^2(\Gamma)$. Thus, by \eqref{eq:adj},
\begin{eqnarray} \label{eq:duals}
\|\opAinv\|_{H_k^{-s}(\Gamma)\to H_k^{-s}(\Gamma)} &=& \|A_{k,\eta}^{-1}\|_{H_k^{s}(\Gamma)\to H_k^s(\Gamma)} \quad \mbox{and}\\ \label{eq:duals2}
\|\opAinv\|_{H^{-s}(\Gamma)\to H^{-s}(\Gamma)} &=& \|A_{k,\eta}^{-1}\|_{H^{s}(\Gamma)\to H^s(\Gamma)} \quad \mbox{for }0\leq s\leq 1.
\end{eqnarray}

For the general exterior Dirichlet problem it is natural to pose Dirichlet data in $\HhG$ (since $\gamma_+ u \in \HhG$). The mapping properties of $H_k$ and $D_k$ (see \cite[Theorems 2.17, 2.18]{ChGrLaSp:12}) imply that $\opBM: H^{s+1}(\Gamma) \rightarrow H^s(\Gamma)$ for $-1\leq s\leq 0$, and thus $\opBM \gamma_+ u \in \HmhG$. Thus, for Dirichlet data in $H^{1/2}(\Gamma)$, the invertibility of $\opA$ on $H^{-1/2}(\Gamma)$ is particularly relevant and, for the solution of \eqref{eq:BW}, the invertibility of $A_{k,\eta}$ on $H^{1/2}(\Gamma)$. The major application of \eqref{eq:CFIE2}, however, is the solution of problems of sound soft acoustic scattering (see \cite[Definition 2.11, Theorem 2.46]{ChGrLaSp:12}), in which $u$ is interpreted as the {\em scattered field} corresponding to an {\em incident field} $u^i$ that satisfies $\Delta u^i + k^2 u^i = 0$ in some neighbourhood $G$ of $\overline{\Omega_-}$, and here the Dirichlet data $\gamma_+ u = -u^i|_\Gamma\in H^1(\Gamma)$ is smoother, so that $\opBM \gamma_+ u \in L^2(\Gamma)$. Indeed, in this case \cite[Theorem 2.46]{ChGrLaSp:12},
the unknown $\partial_n^+ u^t$ satisfies the integral equation,
\beq\label{eq:f}
\opA \partial_n^+ u^t = f_{k,\eta} := \dnpu^i - \ri \eta \gamma_+ u^i \in L^2(\Gamma).
\eeq
where $u^t:=u+u^i$ is the so-called {\em total field} satisfying $\gamma_+u^t=0$.
Therefore, in applications to acoustic scattering, the invertibility of $\opA$ on $L^2(\Gamma)$ is also important. Indeed, $L^2(\Gamma)$ is a natural function space setting for implementation and analysis of Galerkin numerical methods for the solution of the direct equations \eqref{eq:CFIE2} and \eqref{eq:f}, and the indirect equation \eqref{eq:BW} (e.g., \cite{LoMe:11,ChHeLaTw:12,GrLoMeSp:15,EcOz:17, GaMuSp:16} and recall the discussion in \S\ref{sec:BEM}).

\subsection{Inverses of the combined-field operators in terms of the exterior DtN and the interior impedance to Dirichlet maps}\label{sec:6.2}
We introduced in the proof of Lemma \ref{lem:recipe} the notation $\DtN$ for the exterior DtN map.  Similarly, for the Lipschitz open set $\Omega_-$, let $\ItD:H^{-1/2}(\Gamma)\to H^{1/2}(\Gamma)$ denote the interior impedance-to-Dirichlet map, that takes impedance data $g\in H^{-1/2}(\Gamma)$  to $\gamma_-u\in H^{1/2}(\Gamma)$, where $u$ is the solution of $\Delta u + k^2 u = 0$ in $\Omega_-$ that satisfies the impedance boundary condition \eqref{eq:ibc} below.  $\ItD$ extends uniquely to a bounded mapping from $H^{s}(\Gamma) \rightarrow H^{s+1}(\Gamma)$ for $-1\leq s\leq 0$ (see \cite[Theorem 2.32]{ChGrLaSp:12}).

The inverse of $\opA$ can be written in terms of $\DtN$ and $\ItD$ as
\beq\label{eq:key}
\opAinv = I - \DtN\ItD + \ri \eta \ItD
\eeq
\cite[Theorem
2.33]{ChGrLaSp:12}. The fact that $\ItD$, as well as $\DtN$, appears in this formula is because a boundary integral equation formulation of the interior impedance problem leads to the same operator $\opA$; see \cite[Theorem 2.38]{ChGrLaSp:12}.
To use \eqref{eq:key} to bound $\opAinv$ one therefore needs bounds on the exterior DtN map, provided for $(R_0,R_1)$ obstacles in Theorem \ref{thm:dtn}, but also bounds on the interior impedance to Dirichlet map given by the following theorem.

\begin{theorem}\label{cor:ItD}
If $\Omega_-$ is either star-shaped with respect to a ball or $C^\infty$ then
\beq \label{eq:ItDmain}
\|\ItD\|_{H^s_k(\Gamma)\to H^{s+1}_k(\Gamma)} \lesssim 1, \quad \mbox{for } k> 0,
\eeq
uniformly for $-1\leq s\leq 0$,
and
\beq \label{eq:ItDmain2}
\|\ItD\|_{H^s_k(\Gamma)\to H^{s}_k(\Gamma)} \lesssim k^{-1}, \quad \mbox{for } k> 0,
\eeq
uniformly for $-1\leq s\leq 1$.
If $\Omega_-$ is only piecewise smooth, the bounds \eqref{eq:ItDmain} and \eqref{eq:ItDmain2} hold with $1$ and $k^{-1}$ replaced by $k^{1/4}$ and $k^{-3/4}$.
If $\Omega_-$ is only Lipschitz, the bounds \eqref{eq:ItDmain} and \eqref{eq:ItDmain2} hold with $1$ and $k^{-1}$ replaced by $k^{1/2}$ and $k^{-1/2}$.
\end{theorem}

\bpf
We first show that \eqref{eq:ItDmain} holds for $s=0$.
Given $g\in L^2(\Gamma)$ let $u\in H^1(\Omega_-)$ denote the solution to $\Delta u + k^2 u = 0$ in $\Omega_-$ that satisfies
\begin{equation} \label{eq:ibc}
\partial_n^- u - \ri \eta \gamma_- u = g \quad \mbox{on } \Gamma,
\end{equation}
with $\eta =ck$, for some $c\in \R\setminus\{0\}$. Then, for $k> 0$,
\begin{equation} \label{eq:ItD1}
 \|\partial_n^-u\|_{L^2(\Gamma)} + k \|\gamma_-u\|_{L^2(\Gamma)} \lesssim \|g\|_{L^2(\Gamma)}.
\end{equation}
by Green's theorem; see, e.g., \cite[Lemma 4.2]{Sp:14}.
If $\Omega_-$ is either star-shaped with respect to a ball or $C^\infty$, then
\begin{equation} \label{eq:ItD}
\|\nabla_S(\gamma_-u)\|_{L^2(\Gamma)} \lesssim \|g\|_{L^2(\Gamma)}
\end{equation}
by \cite[Equation 3.12]{MoSp:14} and \cite[Corollary 1.9]{BaSpWu:16} respectively.
Combining \eqref{eq:ItD1} and \eqref{eq:ItD} then gives \eqref{eq:ItDmain} for $s=0$.

Since $\ItD$ is self-adjoint with respect to the real inner product $(\cdot,\cdot)_\Gamma^r$  \cite[p.\ 130]{ChGrLaSp:12}, it follows from \eqref{eq:adj2} that \eqref{eq:ItDmain} holds for $-1\leq s\leq 0$, uniformly in $s$.
Further, it is immediate from the definition of the norm on $H^s_k(\Gamma)$ that the embedding operator from $H^s_k(\Gamma)$ to $H^{s-1}_k(\Gamma)$ has norm $\leq k^{-1}$ for $k>0$ and $s=0,1$, and hence for $0\leq s\leq 1$ by interpolation (see \eqref{eq:interp}). Thus \eqref{eq:ItDmain2} follows from \eqref{eq:ItDmain}.

If $\Omega_-$ is piecewise smooth, then, given $k_0>0$, the bound \eqref{eq:ItD} holds for $k\geq k_0$, but with $\|g\|_{L^2(\Gamma)}$ replaced by $k^{1/4}\|g\|_{L^2(\Gamma)}$ and in the general Lipschitz case \eqref{eq:ItD} holds for $k\geq k_0$ with $\|g\|_{L^2(\Gamma)}$ replaced by $k^{1/2}\|g\|_{L^2(\Gamma)}$; see \cite[Lemma 4.6]{Sp:14}. The adjustments to \eqref{eq:ItDmain} and \eqref{eq:ItDmain2} then follow.
\epf

\subsection{From resolvent estimates to Corollary \ref{cor:CFIE}}\label{sec:6.3}
The following lemma captures arguments made  in \cite{BaSpWu:16} for the nontrapping case (where $K(k)=1$ for $k\geq0$), and provides a general recipe for bounding $\opAinv$ as a corollary of resolvent estimates in $\Omega_+$. Bounds on $A_{k,\eta}^{-1}$ (as opposed to $\opAinv$ ) then follow immediately from \eqref{eq:duals} and \eqref{eq:duals2}.

\begin{lemma} \label{lem:recipe2} Suppose that $\Omega_+$ satisfies a $K$ resolvent estimate, for some $K\in C[0,\infty)$ with $K(k)\geq 1$ for $k>0$, and that $\eta = ck$, for some $c\in \R\setminus\{0\}$. Then,  given $k_0>0$, provided each component of $\Omega_-$ is either star-shaped with respect to a ball or $C^\infty$,
\begin{equation}\label{eq:Ainvb}
\|\opAinv\|_{H_k^s(\Gamma)\to H_k^s(\Gamma)} \lesssim K(k) \;\mbox{ and }\; \|\opAinv\|_{H^s(\Gamma)\to H^s(\Gamma)} \lesssim k^{-s}K(k), \quad \mbox{for } k\geq k_0,
\end{equation}
uniformly for $-1\leq s\leq 0$. The bounds \eqref{eq:Ainvb} hold with $K(k)$ and $k^{-s}K(k)$ replaced by $k^{1/4}K(k)$ and $k^{1/4-s}K(k)$, respectively, if each component of $\Omega_-$ is either star-shaped with respect to a ball (and Lipschitz) or piecewise smooth. They hold with $K(k)$ and $k^{-s}K(k)$ replaced by $k^{1/2}K(k)$ and $k^{1/2-s}K(k)$, respectively, in the general Lipschitz case.
\end{lemma}

\bpf
The first bound in \eqref{eq:Ainvb} follows by combining Lemma \ref{lem:recipe}, \eqref{eq:key}, and
Theorem \ref{cor:ItD}.
The second bound  in \eqref{eq:Ainvb} when $s=0$ and $s=-1$ follows from the first and the fact that $\|\psi\|_{H^{-1}_k(\Gamma)}\lesssim \|\psi\|_{H^{-1}(\Gamma)}\lesssim k \|\psi\|_{H^{-1}_k(\Gamma)}$ for $\psi \in H^{-1}(\Gamma)$ and $k\geq k_0$. The second bound  in \eqref{eq:Ainvb} when $-1<s<0$ then follows by the interpolation bound \eqref{eq:interp}.
\epf

\

The upper bounds \eqref{eq:Ainv_bound_main} and \eqref{eq:Ainv_bound_main1} in
Corollary \ref{cor:CFIE} and the comments in Remark \ref{rem:A} follow immediately from combining Lemma \ref{lem:recipe2} and Theorem \ref{thm:resol}. The following lemma proves the lower bound \eqref{eq:Ainv_bound_mainlower} and so completes the proof of Corollary \ref{cor:CFIE}.
In this lemma, for $x=(x_1,...,x_d)\in \R^d$ we write $\widetilde x := (x_2,...,x_{d})\in \R^{d-1}$, so that $x=(x_1,\widetilde x)$.

\begin{lemma} \label{lem:smysh}
Suppose that $\epsilon>0$, $a_2>a_1$, and that $\Gamma_1\cup \Gamma_2\subset \Gamma$ and $\Omega_C\subset \Omega_+$, where $\Gamma_j:=\{x=(a_j,\widetilde x):|\widetilde x|<\epsilon\}$, for $j=1,2$, and $\Omega_C:= \{x=(x_1,\widetilde x):|\widetilde x|<\epsilon \mbox{ and }a_1<x_1<a_2\}$. Then, where $a:=a_2-a_1$, provided $|\eta|\lesssim k$,
\beq\label{eq:Ainv_bound_main3}
\normAinv_{\LtGt} \gtrsim k,\quad  \mbox{for } k\in Q:=\{m\pi/a:m\in \mathbb{N}\}.
\eeq
\end{lemma}

Observe that the geometric assumptions in Lemma \ref{lem:smysh} include every $(R_0,R_1,a)$ parallel trapping obstacle, if necessary after an appropriate change of coordinate system.

\

\begin{proof}[Proof of Lemma \ref{lem:smysh}]
Let $S:= \{\widetilde x\in \R^{d-1}:|\widetilde x|<\epsilon/2\}$. Choose a non-zero $\chi\in C_0^\infty(\R^{d-1})$ supported in $S$. For some $c_j\in \C$ with $|c_1|=|c_2|=1$, let $\phi_j((a_j,\widetilde x)) := c_j\chi(\widetilde x)$, $\widetilde x\in \R^{d-1}$, for $j=1,2$. Let $\phi\in C^1(\Gamma)$ be defined by $\phi(x) = \phi_j(x)$, $x\in \Gamma_j$, for $j=1,2$, $\phi(x)=0$ otherwise, and define $u\in H^1_{\mathrm{loc}}(\R^d)\cap C(\R^d)\cap C^2(\R^{d}\setminus \supp(\phi))$ by
\beq \label{eq:slp}
u(x) := \int_\Gamma \Phi_k(x,y)\phi(y)\,\rd s(y), \quad \mbox{for }x\in \R^d.
\eeq
Using the standard jump relations \cite[p.\ 115]{ChGrLaSp:12}, we see that
\beq \label{eq:fke}
\opA \phi =f_{k,\eta}:= \partial_n^- u - \ri \eta \gamma_- u.
\eeq
Clearly, $\|\phi\|_{L^2(\Gamma)}\gtrsim 1$. We prove the lemma by showing that $\|f_{k,\eta}\|_{L^2(\Gamma)}\lesssim k^{-1}$ if $k\in Q$ and $|\eta|\lesssim k$, provided we choose the phase of $c_2/c_1$ correctly.

Let $\widehat \chi$ denote the Fourier transform of $\chi$, given by
$$
\widehat \chi(\xi) := \int_{\R^{d-1}} \chi(\widetilde x) \re^{-\ri \widetilde x \cdot \xi}\, \rd \widetilde x, \quad \xi\in \R^{d-1}.
$$
Clearly $u = u^{(1)}+u^{(2)}$, where
\begin{align} \label{eq:uj1}
u^{(j)}(x) &:= \int_{\Gamma_j} \Phi_k(x,y)\phi_j(y)\, \rd s(y)\\ \label{eq:uj2}
&= \frac{\ri c_j}{2(2\pi)^{d-1}}\int_{\R^{d-1}} \frac{\widehat \chi(\xi)}{\sqrt{k^2-|\xi|^2 }}\,\exp\Big(\ri\big(\widetilde x\cdot \xi + |x_1-a_j|\sqrt{k^2-|\xi|^2 }\,\big)\Big)\, \rd \xi,
\end{align}
for $j=1,2$ and $x\in \R^d$, with $\sqrt{k^2-|\xi|^2 } = \ri \sqrt{|\xi|^2 -k^2}$ for $|\xi|>k$.
The fact that \eqref{eq:uj1} and \eqref{eq:uj2} are equivalent follows from Fourier representations for layer potentials and boundary integral operators; see, e.g., \cite[Theorem 3.1]{ChHe:15}.

For $x\in \R^d$,
$$
u^{(j)}(x) = \frac{\ri c_j}{2k}\, \chi(\widetilde x)\, \re^{\ri k|x_1-a_j|} + v^{(j)}(x)
$$
where
$$
v^{(j)}(x) = \frac{\ri c_j}{2(2\pi)^{d-1}}\int_{\R^{d-1}} \widehat \chi(\xi)\,\re^{\ri \widetilde x\cdot\xi}\left(\frac{\exp(\ri |x_1-a_j|\sqrt{k^2-|\xi|^2 }\,)}{\sqrt{k^2-|\xi|^2 }}-\frac{\re^{\ri k|x_1-a_j|}}{k}\right)  \, \rd \xi.
$$
The point of this decomposition is that
$v^{(j)}(x)= \cO(k^{-2})$ as $k\to\infty$, uniformly on every bounded subset of $\R^d$, and, provided $k\in Q$, one can choose $c_1$ and $c_2$ such that
\beq\label{eq:dagger}
\frac{\ri c_1}{2k}\, \chi(\widetilde x)\, \re^{\ri k|x_1-a_1|} + \frac{\ri c_2}{2k}\, \chi(\widetilde x)\, \re^{\ri k|x_1-a_2|}
\eeq
is zero for $x\in \Gamma$ (indeed for all $x\not\in\Omega_C$); these observations will lead to the required estimate $\|f_{k,\eta}\|_{L^2(\Gamma)}= \cO(k^{-1})$.

To obtain the bound on $v^{(j)}(x)$ we observe that, for $x\in \R^d$,
\begin{eqnarray*}
|v^{(j)}(x)| &\leq &\frac{1}{2(2\pi)^{d-1}}\int_{\R^{d-1}} |\widehat \chi(\xi)|\,\left(\left|\frac{\exp(\ri |x_1-a_j|(\sqrt{k^2-|\xi|^2 }\,-k))-1}{\sqrt{k^2-|\xi|^2 }}\right|+\frac{|k-\sqrt{k^2-|\xi|^2 }|}{k|\sqrt{k^2-|\xi|^2 }|}\right)  \, \rd \xi\\
& \leq & \frac{k|x_1-a_j|+3}{2(2\pi)^{d-1} k^2}\int_{\R^{d-1}} \frac{|\widehat \chi(\xi)||\xi|^2 }{|\sqrt{k^2-|\xi|^2 }|}  \, \rd \xi,
\end{eqnarray*}
since $|\re^{\ri t}-1|\leq |t|$  for $t\in \R$,
\beq\label{eq:expbound1}
|\sqrt{k^2-|\xi|^2 }\,-k| = |\xi|^2 /|\sqrt{k^2-|\xi|^2 }\,+k|\leq |\xi|^2 /k \quad \mbox{for } \xi\in \R^{d-1},
\eeq
and, for $|\xi|>k$ and $b\geq 0$,
\beq\label{eq:expbound2}
\left|\exp(\ri b(\sqrt{k^2-|\xi|^2 }\,-k))- 1\right|\leq 2 \leq  2|\xi|^2 /k^2.
\eeq
Moreover, since $\hat \chi$ is in the Schwartz space $\cS(\R^{d-1})$, it vanishes rapidly at infinity, and thus for some $C>0$ we have $|\hat \chi(\xi)| \leq C(1+|\xi|)^{-2-d}$ for $\xi\in \R^{d-1}$, so that, for some $C', C''>0$,
\begin{eqnarray*}
\int_{\R^{d-1}} \frac{|\hat\chi(\xi)||\xi|^2 }{|\sqrt{k^2-|\xi|^2 }|}  \, \rd \xi &\leq & C' \int_0^\infty \frac{\rd r}{|\sqrt{k^2-r^2}|(1+r)^2}\\ & \leq &
C''\left(\frac{1}{k^2}\int_{k/2}^{3k/2} \frac{\rd r}{|\sqrt{k^2-r^2}|} + \frac{1}{k} \int_0^\infty (1+r)^{-2}\,\rd r\right) = \cO(k^{-1}).
\end{eqnarray*}
Thus $v^{(j)}(x) = \cO(k^{-2})$ as $k\to \infty$ for $j=1,2$, uniformly in $x$ in every bounded subset of $\R^d$.

Since $\Omega_C\subset \Omega_+$, we have $\Gamma \subset \overline{\Omega_*}$  and $\Gamma\setminus (\Gamma_1 \cup \Gamma_2)\subset \Omega_*$, where $\Omega_* := \{x\in \R^d: x_1< a_1 \mbox{ or } x_1> a_2 \mbox{ or }|\widetilde x|> \epsilon/2\}$.  Choosing $c_1=1$ and $c_2=-\re^{\ri ka}$, we see that \eqref{eq:dagger} equals zero for $x\in \overline{\Omega_*}$ and $k\in Q$, and thus
\beq \label{eq:bound1}
u(x) = u^{(1)}(x)+u^{(2)}(x)=v^{(1)}(x)+v^{(2)}(x), \quad\mbox{so that}\quad u(x)=\cO(k^{-2})
\eeq
for $k\in Q$,
uniformly on bounded subsets of $\overline{\Omega_*}$, in particular uniformly on $\Gamma$. Since $\Delta u + k^2 u = 0$ in $\Omega_*$,   it follows using \eqref{eq:intreg} that $\nabla u(x)  = \cO(k^{-1})$
for $k\in Q$,
uniformly  for $x\in \Gamma\setminus (\Gamma_1 \cup \Gamma_2)$. Finally,
from \eqref{eq:uj2} we have that,
with this choice of $c_1$ and $c_2$ and $k\in Q$, for $x\in \Gamma_1\cup \Gamma_2$,
\begin{align} \nonumber
|\partial_n^- u(x)| &= \frac{1}{2(2\pi)^{d-1}}\left|\int_{\R^{d-1}} \widehat \chi(\xi)\re^{\ri \widetilde x \cdot \xi}\left(\exp(\ri a(\sqrt{k^2-|\xi|^2 }\,-k))-1\right)\, \rd \xi\right|\\ \label{eq:bound2}
& \leq  \frac{a}{2(2\pi)^{d-1} k}\int_{\R^{d-1}} |\widehat \chi(\xi)| |\xi|^2 \, \rd \xi,
\end{align}
using \eqref{eq:expbound1}, \eqref{eq:expbound2}, and that $\pi\leq ka$ for $k\in Q$.
Putting these bounds together in \eqref{eq:fke}, we have shown that
$f_{k,\eta}(x) = \cO(k^{-1})$ as $k\to\infty$ with $k\in Q$ and $|\eta|\lesssim k$, uniformly on $\Gamma$.
\end{proof}

\

The proof of Lemma \ref{lem:smysh} was inspired by the billiard-type arguments used to construct high-frequency quasimodes, going back to Keller and Rubinow \cite{KeRu:60}; see, e.g., \cite{BaBu:08} and the references therein. We also expect that lower bounds on $\normAinv$ similar to that in
Lemma \ref{lem:smysh} can be obtained when $\Omega_+$ supports arbitrary closed finite billiards and $\Gamma$ is flat in the neighbourhood of each reflection.

\subsubsection{Comparison between Lemma \ref{lem:smysh} and the results of \cite{ChGrLaLi:09}}
In the proof of Lemma \ref{lem:smysh}, $f_{k,\eta}$ is bounded via its representation \eqref{eq:fke} as boundary data for an interior impedance problem satisfied by $u$. In
\cite{ChGrLaLi:09} a less-sharp bound is obtained in 2-d, that $\normAinv_{\LtGt} \gtrsim k^{9/10}$ for $k\in Q$, via an alternative formula for $f_{k,\eta}$. Precisely, with $\phi$ and $f_{k,\eta}$ as in the above proof, it is shown that $\phi=\partial_n^+ u^t$ is the normal derivative of the total field for sound soft scattering when the incident field is
\beq \label{eq:ui}
u^i(x) = \int_{\Omega_+} \Phi_k(x,y) f(y) \, \rd s(y), \quad x\in \R^d,
\eeq
with $f$ supported in $\Omega_C\subset \Omega_+$ given by $f(x) := k^{-1}\sin(k x_1) \widetilde \Delta \chi(\widetilde x)$,
 for $x\in \Omega_C$, where $\widetilde \Delta$ is the Laplacian in $\R^{d-1}$. It follows from
\eqref{eq:f} that $f_{k,\eta}=\dnpu^i - \ri \eta \gamma_+ u^i$. This, together with \eqref{eq:ui}, is a formula for $f_{k,\eta}$ as an oscillatory integral over $\supp(f)\subset \Omega_C$. Estimating this integral (suboptimally) in \cite[Theorem 5.1]{ChGrLaLi:09} led to the bound $\|f_{k,\eta}\|_{L^2(\Gamma)}\lesssim k^{-9/10}$.

\subsubsection{Counterexample to a conjecture on coercivity}\label{rem:coercivity}
Under the assumptions that $\Omega_-$ is $C^3$, piecewise analytic, and has strictly positive
curvature,
\cite{SpKaSm:15} shows  that there exists an $\eta_0>0$ (equal to one when $\Oi$ is a ball) and $k_0>0$ such that if $\eta \geq \eta_0 k$ then $\opA$ is {\em coercive} uniformly in $k$ for $k\geq k_0$, meaning that
\beq \label{eq:coer}
\left|\left(\opA\phi,\phi\right)_\Gamma\right| \geq c_k \|\phi\|_{L^2(\Gamma)}^2, \quad \mbox{for all } k\geq k_0 \mbox{ and } \phi\in L^2(\Gamma), \quad \mbox{with } c_k \gtrsim 1;
\eeq
this is shown via a novel use of Morawetz identities in \cite{SpKaSm:15}, generalising an earlier result for the case of a circle/sphere obtained via Fourier analysis \cite{DoGrSm:07}.
This result implies that $\|\opAinv\|_{L^2(\Gamma)\to L^2(\Gamma)} \lesssim 1$, but this bound on the inverse does not imply the stronger \eqref{eq:coer}.

\begin{figure}[h]
\centering
\scalebox{0.7}{
\begin{tikzpicture}[line cap=round,line join=round,>=triangle 45,x=1.0cm,y=1.0cm, scale=1.75]
\colorlet{lightgray}{black!15}
%\draw[step=1cm,gray,very thin] (-1,-1) grid (6,4);
\fill[color=lightgray] (0,0) -- (0,1) .. controls (0,1.5) and (-1,1.5) .. (-1,0) .. controls (-1,-2) and (5,-2) .. (5,0) .. controls (5,1.5) and (4,1.5) .. (4,1) -- (4,0) .. controls (4,-1) and (3.5,0.5) .. (2.5,1.5) .. controls (2,2) and (2,2) .. (1.5,1.5) .. controls (0.5,0.5) and (0,-1) .. (0,0);
\draw (0,0) -- (0,1) .. controls (0,1.5) and (-1,1.5) .. (-1,0) .. controls (-1,-2) and (5,-2) .. (5,0) .. controls (5,1.5) and (4,1.5) .. (4,1) -- (4,0) .. controls (4,-1) and (3.5,0.5) .. (2.5,1.5) .. controls (2,2) and (2,2) .. (1.5,1.5) .. controls (0.5,0.5) and (0,-1) .. (0,0);
\draw [very thick] (0,0) -- (0,1);
\draw [very thick] (4,0) -- (4,1);
\draw (0,0.6) node[anchor=east] {\large $\Gamma_1$};
\draw (4,0.6) node[anchor=west] {\large $\Gamma_2$};
\draw [blue,thick,<->] (0,0.5) -- (4,0.5);
\draw (2,0.5) node[anchor=south] {\Large $a$};
%	\draw [line width=1.2pt,dash pattern=on 2pt off 2pt] (6,3)-- (8.5,3);
%	\draw [line width=1.2pt,dash pattern=on 2pt off 2pt] (2,0)-- (2,-3.5);
\end{tikzpicture}}
\caption{The obstacle $\Omega_-$ shaded grey is nontrapping, so that $\|\opAinv\|_{L^2(\Gamma)\to L^2(\Gamma)} \lesssim 1$ if $\eta = ck$, for some constant $c\in \R\setminus \{0\}$. However, \S\ref{rem:coercivity} shows that $\opA$ is not coercive uniformly in $k$.}
\label{fig:coer}
\end{figure}
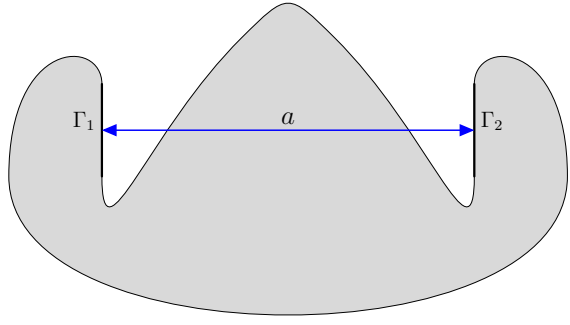

The advantage of coercivity, as opposed to just boundedness of the inverse, for the numerical analysis of Galerkin methods is discussed in \cite{SpKaSm:15};
for example, the coercivity result in \cite{SpKaSm:15} completes
the numerical analysis of high frequency numerical-asymptotic boundary element methods for scattering by convex obstacles \cite{DoGrSm:07,EcOz:17}.

Based on computations of the numerical range (an operator is coercive if and only if zero is not in the closure of its numerical range), \cite{BeSp:11}  conjectured that, if $\Omega_-$ is nontrapping, then \eqref{eq:coer} holds with $\eta = k$ (i.e.\ $A^\prime_{k,k}$ is coercive uniformly in $k$)
 \cite[Conjecture 6.2]{BeSp:11}. This conjecture implies that $\|(A^\prime_{k,k})^{-1}\|_{L^2(\Gamma)\to L^2(\Gamma)} \lesssim 1$ for nontrapping domains, and this result was recently proved in \cite[Theorem 1.13]{BaSpWu:16}. The calculations in Lemma \ref{lem:smysh}, however, show that this conjecture is false.

Suppose that $\Omega_+$ satisfies the conditions of Lemma \ref{lem:smysh}, except that we no longer require that $\Omega_C\subset \Omega_+$, instead we require that $\Omega_-$ is nontrapping (which implies that $\Gamma$ passes through $\Omega_C$), and we require that $n(x)=e_1$ on $\Gamma_1$, $n(x)=-e_1$ on $\Gamma_2$. An example is Figure \ref{fig:coer}. Define $\phi\in L^2(\Gamma)$ as in the proof of Lemma \ref{lem:smysh}, so that the value of $\|\phi\|_{L^2(\Gamma)}\neq 0$ is independent of $k$. Equations \eqref{eq:slp}, \eqref{eq:fke}, \eqref{eq:bound1} and \eqref{eq:bound2} still hold, and still imply that $f_{k,\eta}(x)=\cO(k^{-1})$ for $k\in Q=\{m\pi/a:m\in \mathbb{N}\}$ with $|\eta| \lesssim k$, uniformly on $\Gamma_1\cup\Gamma_2$ (but now not on all of $\Gamma$ since $\Gamma\not\subset \overline{\Omega_*}$). Thus, provided $|\eta| \lesssim k$, since $\supp (\phi)\subset \Gamma_1 \cup \Gamma_2$,
$$
\left(\opA\phi,\phi\right)_\Gamma = \int_{\Gamma_1\cup\Gamma_2} f_{k,\eta}\overline{\phi}\, \rd s =\cO(k^{-1}),
$$
as $k\to\infty$ through the sequence $Q$,
so that \eqref{eq:coer} is false in this case. It may still hold that $\opA$ is coercive, but if  this is the case then the coercivity constant $c_k=\cO(k^{-1})$ as $k\to\infty$ through the sequence $Q$.

\subsection{Summary of wavenumber-explicit bounds on $\opAinv$} \label{sec:other}

Table \ref{tab:Ainv} below summarises: (i) the (sharpest) known resolvent estimates for scattering by obstacles, discussed in \S\ref{sec:1}; (ii) the sharpest known bounds on the DtN map, taken from \cite{BaSpWu:16} for the nontrapping cases, proved as corollaries in \S\ref{sec:dtn} for the trapping cases; (iii) the bounds on the inf-sup constant obtained from the resolvent estimates (as discussed in Remark \ref{rem:resinfsup}); and (iv) the upper bounds on $\normAinv_{\LtGt}$  that follow from the resolvent estimates by the general Lemma \ref{lem:recipe2}. (The bounds that were already known have been discussed earlier in \S\ref{sec:151}, the other bounds are stated here for the first time as corollaries of the resolvent estimates and Lemma \ref{lem:recipe2}.)  The upper bounds in the last column, by Lemma \ref{lem:recipe2}, are also upper bounds for $\|\opAinv\|_{H_k^s(\Gamma)\to H_k^s(\Gamma)}$, uniformly for $-1\leq s\leq 0$, and the same bounds, multiplied by a factor $k^{-s}$, are upper bounds for $\|\opAinv\|_{H^s(\Gamma)\to H^s(\Gamma)}$. Further, bounds on $A_{k,\eta}^{-1}$ follow immediately from \eqref{eq:duals} and \eqref{eq:duals2}.

In the last column of Table \ref{tab:Ainv} and in row 6 we include lower as well as upper bounds. Each lower bound holds {\em for at least one example} in the class indicated and {\em for at least some unbounded sequence of wavenumbers}. (The particular bound $\|\opAinv\|_{L^2(\Gamma)\to L^2(\Gamma)} \gtrsim 1$ \cite[Lemma 4.1]{ChGrLaLi:09} holds for $k\geq k_0$ whenever part of $\Gamma$ is $C^1$.)
The lower bounds in the last row and column of the table, and their relationship to the upper bound, should be interpreted as follows. Firstly, that in 2-d there exists an $\Omega_+$ that is $C^\infty$ (\cite{BeChGrLaLi:11} gives specific elliptic-cavity trapping examples  of which Figure \ref{fig:examples}(a) is typical) and positive constants $\alpha_2 \geq \alpha_1$ such that, with $\eta=ck$ for some $c\in \R\setminus\{0\}$,
$$
\exp(\alpha_1 k) \lesssim \|\opAinv\|_{L^2(\Gamma)\to L^2(\Gamma)} \lesssim \exp(\alpha_2k)
$$
as $k\to\infty$ through some positive, unbounded sequence of wavenumbers. Secondly, that in both 2-d and 3-d, whenever $\Omega_-$ permits elliptic trapping, allowing an elliptic closed broken geodesic $\gamma$, provided $\Gamma$ is analytic in neighbourhoods of the vertices of $\gamma$ and the local Poincar\'e map near $\gamma$ satisfies the additional conditions of \cite[(H1)]{CaPo:02}, it holds for every $q<2/11$ ($d=2$), $q<1/7$ ($d=3$), that there exists $\alpha_3>0$ such that
\beq \label{eq:carpop}
\exp(\alpha_3k^q) \lesssim \|\opAinv\|_{L^2(\Gamma)\to L^2(\Gamma)}
\eeq
as $k\to\infty$ through some positive, unbounded sequence of wavenumbers. The lower bound \eqref{eq:carpop} follows immediately from Theorem 1 in Cardoso and Popov \cite{CaPo:02}, which shows the existence, under these assumptions, of exponentially small quasi-modes, which moreoever can be constructed to be localised arbitrarily close to $\gamma$, and \cite[Equation (5.39)]{ChGrLaSp:12}, which converts exponentially small quasi-modes into lower bounds on $\|\opAinv\|_{L^2(\Gamma)\to L^2(\Gamma)}$.

\begin{table}[h]
\footnotesize

\begin{center}
\begin{tabular}{|p{2.9 cm} |l |l |p{2cm} |p{3cm} |}
  \hline
Geometry of $\Omega_-$ & $K(k)$ & $\| \DtN\|_{H^1\to L^2}$ & $\beta_R^{-1}$ & $\|\opAinv\|_{L^2(\Gamma)\to L^2(\Gamma)}$\\ \hline \hline
1.\ $C^\infty$ and nontrapping & $\lesssim 1$ \cite{Va:75,MeSj:82}
& $\lesssim 1$ \cite{BaSpWu:16} & $\lesssim k$ \cite{Sp:14}& $\lesssim 1$ \cite{BaSpWu:16} \hfill ($\gtrsim 1$ \cite{ChGrLaLi:09})\\
2.\ Nontrapping polygon  & $\lesssim 1$ \cite{BaWu:13} & $\lesssim 1$ \cite{BaSpWu:16}& $\lesssim k$ \cite{Sp:14}& $\lesssim k^{1/4}$
\hfill ($\gtrsim 1$ \cite{ChGrLaLi:09})  \\
3.\ Star-shaped and Lipschitz & $\lesssim 1$ \cite{Mo:75,ChMo:08}
& $\lesssim 1$ \cite{BaSpWu:16} & $\lesssim k$ \cite{ChMo:08} & $\lesssim k^\beta$ \hfill ($\gtrsim 1$ \cite{ChGrLaLi:09})\\
4.\ Star-shaped with respect to a ball and Lipschitz & $\lesssim 1$ \cite{Mo:75,ChMo:08}& $\lesssim 1$ \cite{BaSpWu:16}
& $\lesssim k$ \cite{ChMo:08} & $\lesssim 1$ \cite{ChMo:08,Sp:14} \hfill ($\gtrsim 1$ \cite{ChGrLaLi:09})\\ \hline
\hline\hline
5.\ Ikawa-like union of convex obstacles & $\lesssim \log(2+k)$ \cite{Bu:04} & $\lesssim \log(2+k)$  &  $\lesssim k\log(2+k)$ & $\lesssim \log(2+k)$ \hfill ($\gtrsim 1$ \cite{ChGrLaLi:09})\\
6.\ $(R_0,R_1)$ obstacle & $\lesssim k^2$ ($\gtrsim k$ \cite{ChMo:08})& $\lesssim k^2$ \hfill & $\lesssim k^3$  ($\gtrsim k^2$\hspace{-0.5ex} \cite{ChMo:08}) & $\lesssim k^{2+\beta}$ \hfill ($\gtrsim k$)\\
&&&&\\
7.\ Arbitrary $C^\infty$ & $\lesssim \re^{\alpha k}$ \cite{Bu:98}& $\lesssim \re^{\alpha k}$ & $\lesssim k\re^{\alpha k}$ & $\lesssim \re^{\alpha k}$ \hfill ($\gtrsim \re^{\alpha k}$ (2-d) \cite{BeChGrLaLi:11},\newline
\hspace*{3ex} \hfill $\gtrsim \re^{\alpha k^q}$)\\
  \hline
\end{tabular}
\end{center}
\caption{\footnotesize Summarising the known wavenumber-explicit upper bounds that hold for $k\geq k_0>0$; in the last column and in row 6 we also show the known lower bounds. Rows 1-4 apply in nontrapping cases. Rows 5-7 apply to trapping geometries, row 6 in particular to $(R_0,R_1,a)$ parallel trapping obstacles. In the last column we assume that $\eta = ck$, for some non-zero real constant $c$, and $\beta=0$ if each component of $\Omega_-$ is $C^\infty$ or star-shaped with respect to a ball, $\beta=1/4$ if each component is merely piecewise smooth or star-shaped with respect to  a ball, $\beta=1/2$ for general Lipschitz $\Omega_-$. The bounds without citations are stated explicitly for the first time in this paper.}\label{tab:Ainv}
\end{table}

\normalsize

\subsection{Bounds on $\cond(\opA)$} \label{sec:cond}

There has been sustained interest in the condition number $\cond(\opA)$, defined by
\eqref{eq:cond}, of $\opA$ as an operator on $L^2(\Gamma)$; see Remark \ref{rem:history} below and the references therein. We therefore put the bounds in the last column of Table \ref{tab:Ainv} together with existing bounds on the norm of $\opA$ to produce the following result giving upper and lower bounds on $\cond(\opA)$
and how this depends on the geometry of $\Omega_-$. The bounds in Parts (iii), (iv), (v), and the upper bound and most of the lower bounds in (vi) are given here for the first time, with the bounds in (i) and (ii) given in \cite[\S6]{ChGrLaLi:09} and \cite[\S7.1]{BaSpWu:16}, and the lower bound \eqref{eq:cnb7} in (vi) for a 2-d elliptic cavity given in \cite[Theorem 2.8]{BeChGrLaLi:11}.

\begin{theorem}[Bounds on the condition number] \label{cor:condno} Suppose that $\eta = ck$, for some non-zero real constant $c$, and that $k_0>0$.
\begin{itemize}
\item[(i)] Let $\Omega_-$ be $C^\infty$ and nontrapping, or star-shaped with respect to a ball and piecewise smooth, and suppose that
$\Gamma$ has strictly positive curvature.
Then, for $k\geq k_0$,
\beq \label{eq:cnb}
k^{1/3} \lesssim \cond(\opA) \lesssim k^{1/3}\log(2+k); \quad \mbox{indeed} \quad \cond(\opA) \sim k^{1/3}
\eeq
if $\Omega_-$ is a ball in 2-d or 3-d (i.e., a circle or sphere).
\item[(ii)] Let $\Omega_-$ be $C^\infty$ and nontrapping, or star-shaped with respect to a ball and piecewise smooth. Then, for $k\geq k_0$,
\beq \label{eq:cnb2}
k^{1/3} \lesssim \cond(\opA) \lesssim k^{1/2}\log(2+k); \quad \mbox{indeed} \quad k^{1/2} \lesssim \cond(\opA) \lesssim k^{1/2}\log(2+k)
\eeq
if $\Gamma$ contains a line segment. Moreover these bounds hold without the log factors in 2-d; in particular $\cond(\opA) \sim k^{1/2}$ in 2-d if $\Omega_-$ is $C^\infty$ and nontrapping and $\Gamma$ contains a line segment.
\item[(iii)] Let $\Omega_-$ be a nontrapping polygon. Then, for $k\geq k_0$,
\beq \label{eq:cnb3}
k^{1/2} \lesssim \cond(\opA) \lesssim k^{3/4}; \quad \mbox{indeed} \quad \cond(\opA) \sim k^{1/2}
\eeq
if $\Omega_-$ is star-shaped with respect to a ball.
\item[(iv)] Let $\Omega_-$ be an Ikawa-like union of convex obstacles. Then, for $k\geq k_0$,
\beq \label{eq:cnb4}
k^{1/3} \lesssim \cond(\opA) \lesssim k^{1/3} [\log(2+k)]^2.
\eeq
\item[(v)] Let $\Omega_-$ be an $(R_0,R_1,a)$ parallel trapping obstacle. Then, where the upper bounds hold for all $k\geq k_0$ while the lower bounds  apply specifically for $k\in Q:= \{m\pi/a:m\in \mathbb{N}\}$, it holds that
\beq \label{eq:cnb5}
k^{3/2} \lesssim \cond(\opA) \lesssim k^{2+d/2}; \quad \mbox{indeed } k^{3/2} \lesssim \cond(\opA) \lesssim k^{5/2+\beta} \log(2+k)
\eeq
if $\Gamma$ is piecewise smooth, with $\beta =0$ if each component of $\Omega_-$ is either $C^\infty$ or star-shaped with respect to a ball, $\beta=1/4$  otherwise. For all $k\geq k_0$ the weaker lower bound holds that $\cond(\opA)\gtrsim k^{1/2}$.
\item[(vi)] Let $\Omega_-$ be $C^\infty$. Then there exists $\alpha>0$ such that, for $k\geq k_0$,
\beq \label{eq:cnb6}
k^{1/3} \lesssim \cond(\opA) \lesssim \exp(\alpha k).
\eeq
Further, whenever $\Omega_-$ permits elliptic trapping, allowing an elliptic closed broken geodesic $\gamma$, provided $\Gamma$ is analytic in neighbourhoods of the vertices of $\gamma$ and the local Poincar\'e map near $\gamma$ satisfies the additional conditions of \cite[(H1)]{CaPo:02}, it holds for every $q<2/11$ ($d=2$), $q<1/7$ ($d=3$), that there exists $\alpha^\prime>0$ such that
\beq \label{eq:cnb7}
\cond(\opA) \gtrsim \exp(\alpha^\prime k^q),
\eeq
for some unbounded sequence of positive wavenumbers $k$. Moreover, \eqref{eq:cnb7} holds with $q=1$ in the 2-d case of an elliptic cavity in the sense of \cite[Theorem 2.8]{BeChGrLaLi:11} (an example is Figure \ref{fig:examples}(a)).
\end{itemize}
\end{theorem}

\begin{proof}[Proof of Theorem \ref{cor:condno}]
To bound $\opA$ \eqref{eq:CFIEdef}  it is sufficient to obtain bounds on the operators $S_k$ and $D_k^\prime$ (and note that $D^\prime_k$ has the same norm as $D_k$ as an operator on $L^2(\Gamma)$ as $D_k^\prime$ is the adjoint of $D_k$ with respect to the real inner product on $L^2(\Gamma)$; see, e.g., \cite[Equation 2.37]{ChGrLaSp:12}).
Given $k_0>0$, for $k\geq k_0$, if $\Gamma$ is Lipschitz, then
\beq \label{eq:SDbs}
\|S_k\|_{L^2(\Gamma)\to L^2(\Gamma)} \lesssim k^{(d-3)/2} \quad \mbox{and} \quad \|D^\prime_k\|_{L^2(\Gamma)\to L^2(\Gamma)} \lesssim k^{(d-1)/2}
\eeq
\cite[Theorems 3.3, 3.5]{ChGrLaLi:09}. Furthermore, if $\Gamma$ is piecewise smooth then
\beq \label{eq:SDbs2}
k^{-1/2}\lesssim \|S_k\|_{L^2(\Gamma)\to L^2(\Gamma)} \lesssim k^{-1/2}\log(2+k) \quad \mbox{and} \quad k^{1/4}\lesssim \|D^\prime_k\|_{L^2(\Gamma)\to L^2(\Gamma)} \lesssim k^{1/4}\log(2+k),
\eeq
and if $\Gamma$ is piecewise smooth and has strictly positive curvature then
\beq \label{eq:SDbs3}
k^{-2/3}\lesssim \|S_k\|_{L^2(\Gamma)\to L^2(\Gamma)} \lesssim k^{-2/3}\log(2+k) \quad \mbox{and} \quad k^{1/6}\lesssim\|D^\prime_k\|_{L^2(\Gamma)\to L^2(\Gamma)} \lesssim k^{1/6}\log(2+k)
\eeq
\cite[Appendix A]{HaTa:15} (with the upper bounds on $S_k$ first given in \cite[Theorem 1.2]{GaSm:15}). The lower bound $\|S_k\|_{\LtG\rightarrow \LtG}\gtrsim k^{-1/2}$ holds when $\Gamma$ contains a line segment and is $C^2$ in a neighbourhood thereof by \cite[Theorem 4.2]{ChGrLaLi:09} in 2-d and \cite[Lemma 3.1]{GaSp:18} in 3-d.

These bounds imply that, with $\eta = ck$,
$\|\opA\|_{L^2(\Gamma)\to L^2(\Gamma)} \lesssim k^{1/3}\log(2+k)$ if $\Gamma$ is piecewise smooth with each piece having strictly positive curvature; is $\lesssim k^{1/2}\log(2+k)$ if $\Gamma$ is piecewise smooth; and is $\lesssim k^{(d-1)/2}$ in general. These same results imply that $\|\opA\|_{L^2(\Gamma)\to L^2(\Gamma)} \gtrsim k^{1/3}$ if $\Gamma$ is piecewise smooth; is $\gtrsim k^{1/2}$ if $\Gamma$ contains a line segment and is $C^2$ in a neighbourhood thereof. Furthermore, $\|\opA\|_{L^2(\Gamma)\to L^2(\Gamma)}\sim k^{1/3}$ for a ball (in 2-d and 3-d) by \cite{Gi:97,DoGrSm:07}, and, because of the compactness of $D_k^\prime$ (and $S_k$) on $L^2(\Gamma)$ when $\Gamma$ is $C^1$ \cite{FabJod78}, if a part of $\Gamma$ is $C^1$ then $\|\opA\|_{L^2(\Gamma)\to L^2(\Gamma)}\geq 1/2$ and $\|\opAinv\|_{L^2(\Gamma)\to L^2(\Gamma)}\geq 2$ for $k>0$ \cite[Lemma 4.1]{ChGrLaLi:09}.

The corollary follows by combining these estimates with the bounds on $\|\opAinv\|_{L^2(\Gamma)\to L^2(\Gamma)}$ summarised in Table \ref{tab:Ainv} (recalling the discussion of case (vi) in \S\ref{sec:other}).
\end{proof}

\

The theorem makes clear that the conditioning of $\opA$ (with $\eta$ proportional to $k$) depends strongly on the type of trapping. When $\Oi$ is a ball the conditioning grows precisely as $k^{1/3}$. The conditioning is worse than this for the mild hyperbolic trapping of an Ikawa-like union of convex obstacles, but at most by logarithmic factors. A $C^\infty$ nontrapping obstacle has slightly higher growth in condition number (proportional to $k^{1/2}$) if $\Gamma$ contains a line segment.

By contrast, $(R_0,R_1,a)$ parallel trapping obstacles (the main focus of this paper), have only polynomial growth in condition number, but at a faster rate than all the nontrapping cases considered in the above corollary, at least as fast as $k^{3/2}$ as $k$ increases through a particular unbounded sequence. Finally, if the obstacle allows a stable (elliptic) periodic orbit, then the condition number grows exponentially as $k$ increases through some unbounded sequence.

\bre[The history of studies of the conditioning of $\opA$]\label{rem:history}
The study of the conditioning of $\opA$, and its dependence on the choice of the coupling parameter $\eta$, and latterly also on the geometry of $\Gamma$, has a long history, dating back to the original studies by Kress and Spassov \cite{KrSp:83,Kr:85} for the case where $\Omega_-$ is a circle or sphere, these studies focussed on the low-wavenumber limit. The first rigorous (and sharp) high frequency bounds on $\cond(\opA)$, specifically for a circle/sphere and carried out using the Fourier analysis framework of \cite{KrSp:83}, were obtained in \cite{DoGrSm:07}, and rigorous results for high frequency for more general geometries were obtained in \cite{ChGrLaLi:09}, \cite{BeChGrLaLi:11}, and \cite{BaSpWu:16}.
\ere

\bre[Other choices of coupling parameter $\eta$]
Corollary \ref{cor:condno} focused on the case when the coupling parameter $\eta$ is chosen proportional to $k$ as this is the recommendation from various computational and theoretical studies \cite{Kr:85,Am:90,BrKu01,ChGrLaLi:09,GaMuSp:16}. For discussions of conditioning for other choices of $\eta$, and of the effect of choices of $\eta$ on the condition number and other aspects of the effectiveness of numerical solution methods, see \cite{Kr:85,Am:90,BrKu01,ChGrLaLi:09,BeChGrLaLi:11,Ma:16,BaSpWu:16,GaMuSp:16}.
\ere

\subsection{Proof of Corollary \ref{cor:hversion} (convergence of the $h$-BEM)}\label{sec:hBEMproof}

\begin{definition}[Shape-regular triangulation] \label{def:triang}
Suppose $\mathcal{T}$ is a triangulation of $\Gamma$ in the sense, e.g., of \cite{LoMe:11}, so that each element $K\in \mathcal{T}$ (with $K\subset \Gamma$) is the image of a reference element $\widehat K = \{\xi\in \R^{d-1}: 0<\xi_i<1, \sum_{i=1}^{d-1}\xi_i < 1\}$ under a $C^1$-diffeomorphism $F_K:\overline{\widehat K}\to \overline{K}$, with Jacobian $J_K := D F_K$. Then $\mathcal{T}$ is {\em shape-regular} if there exists a constant $c_S>0$ such that, for every $K\in \mathcal{T}$,
\beq \label{eq:lambda}
\frac{\sup_{\xi\in K}\lambda_K^{\mathrm{max}}(\xi)}{\inf_{\xi\in K}\lambda_K^{\mathrm{min}}(\xi)} \leq c_S,
\eeq
where $\lambda_K^{\mathrm{max}}$ and $\lambda_K^{\mathrm{min}}$ denote the maximum and minimum eigenvalues of $J_K^TJ_K$.
\end{definition}

\bpf[Proof of Corollary \ref{cor:hversion}]
With $p\geq 0$, define the boundary element space $\mathcal{S}^p(\mathcal{T}_h)$ as in \S\ref{sec:BEM}, and let $P_{hp}:L^2(\Gamma)\to \mathcal{S}^p(\mathcal{T}_h)$ be orthogonal projection.
The heart of the proof is the fact that if, for some $\delta>0$,
\beq \label{eq:2ndkind}
\|I-P_{hp}\|_{H^1(\Gamma)\to L^2(\Gamma)} \|D_k^\prime - \ri \eta S_k\|_{L^2(\Gamma)\to H^1(\Gamma)} \|\opAinv\|_{L^2(\Gamma)\to L^2(\Gamma)} \leq \frac{\delta}{1+\delta},
\eeq
then the Galerkin solution $v_{hp}$ of the variational problem \eqref{eq:galerkin} is well-defined and the quasi-optimal error estimate \eqref{eq:melenk} holds with
\beq \label{C3eq}
C_3 = \half (1+\delta) \|\opAinv\|_{L^2(\Gamma)\to L^2(\Gamma)};
\eeq
see \cite[Lemma 4.1]{GrLoMeSp:15}, \cite[Lemma 3.3]{GaMuSp:16}.

Since $\mathcal{S}^p(\mathcal{T}_h)\subset \mathcal{S}^0(\mathcal{T}_h)$ it is clear that $\|I-P_{hp}\|_{H^1(\Gamma)\to L^2(\Gamma)}\leq \|I-P_{h0}\|_{H^1(\Gamma)\to L^2(\Gamma)}$. The approximation result
\beq\label{eq:approx}
\|I-P_{h0}\|_{H^1(\Gamma)\to L^2(\Gamma)} \leq C h,
\eeq
with $C>0$ dependent only on the constant $c_S$ in \eqref{eq:lambda}, is proved in \cite[Theorem 1.4]{St:08} for the case when $\Gamma$ is piecewise smooth and each element $K\in \mathcal{T}_h$ is flat, and the argument extends to the case when $\Gamma$ is piecewise $C^1$.

Part (a) of Corollary \ref{cor:hversion} follows from combining \eqref{eq:2ndkind}, \eqref{eq:approx}, the bound on $ \|\opAinv\|_{L^2(\Gamma)\to L^2(\Gamma)}$ in \eqref{eq:mild}, and the bound
\beqs% \label{Lb3}
\|D_k^\prime - \ri \eta S_k\|_{L^2(\Gamma)\to H^1(\Gamma)} \lesssim k^{4/3}\log(2+k).
\eeqs
when $|\eta|\sim k$, $\Omega_-$ is $C^{2,\alpha}$ for some $\alpha\in (0,1)$, and $\Gamma$ additionally has strictly positive curvature; this last bound is proved in \cite[Theorem 1.5]{GaSp:18}.

Part (b) of Corollary \ref{cor:hversion} follows from combining \eqref{eq:2ndkind}, \eqref{eq:approx}, the bound on $ \|\opAinv\|_{L^2(\Gamma)\to L^2(\Gamma)}$ in \eqref{eq:Ainv_bound_main}, and the bound
\beqs
\|D_k^\prime - \ri \eta S_k\|_{L^2(\Gamma)\to H^1(\Gamma)} \lesssim k^{3/2}\log(2+k),
\eeqs
when $|\eta|\sim k$, $\Omega_-$ is $C^{2,\alpha}$ for some $\alpha\in (0,1)$, and $\Gamma$ is additionally piecewise smooth; this last bound is also proved in \cite[Theorem 1.5]{GaSp:18}.
\epf

\paragraph{Acknowledgements} The heart of this paper is a novel use of Morawetz's identities, and during the final stages of preparing this paper we learnt of
Cathleen Morawetz's death. We gratefully acknowledge here the sustained and
large influence her work has had on our own research, as well as her
wider contributions to the mathematics community.

VPS thanks Vladimir Kamotski for introducing him to the spatial Fourier-transform arguments that we used to prove Lemma \ref{lem:smysh}, and thanks Ilia Kamotski (University College London) for a suggestion that led to a strengthening of the results in \S\ref{sec:resol}.
EAS thanks Jeff Galkowski (Stanford) and Jared Wunsch (Northwestern) for useful discussions about the semiclassical literature on trapping, and acknowledges support from EPSRC grants EP/1025995/1 and EP/R005591/1. AG acknowledges the support of an EPSRC PhD Studentship held at the University of Reading. We are grateful to the anonymous referee for many helpful comments, including the suggestion to add the final example in \S\ref{sec:geometric}.

\footnotesize{
\bibliographystyle{plain}

%\bibliography{biblio_combined_sncwadditions}
%\begin{thebibliography}{10}
}

\end{document}